\documentclass[11pt,reqno]{amsart}
\pdftrailerid{}

\usepackage[margin=1.12in]{geometry}
\usepackage[T1]{fontenc}
\usepackage{lmodern}
\usepackage[expansion=false]{microtype}
\usepackage{mathtools,amssymb,amsfonts}
\usepackage{enumitem}
\usepackage{xcolor}
\usepackage{tikz}
\usetikzlibrary{calc,decorations.markings}
\usepackage{aliascnt}
\usepackage[colorlinks=true,linkcolor=blue!55!black,citecolor=green!35!black,urlcolor=blue!55!black]{hyperref}
\usepackage[nameinlink,capitalize]{cleveref}
\makeatletter
\renewcommand{\@defaultbiblabelstyle}[1]{[#1]}
\makeatother

\newtheorem{theorem}{Theorem}[section]
\newaliascnt{proposition}{theorem}
\newtheorem{proposition}[proposition]{Proposition}
\aliascntresetthe{proposition}
\newaliascnt{lemma}{theorem}
\newtheorem{lemma}[lemma]{Lemma}
\aliascntresetthe{lemma}
\newaliascnt{corollary}{theorem}
\newtheorem{corollary}[corollary]{Corollary}
\aliascntresetthe{corollary}
\theoremstyle{definition}
\newaliascnt{definition}{theorem}
\newtheorem{definition}[definition]{Definition}
\aliascntresetthe{definition}
\newaliascnt{convention}{theorem}
\newtheorem{convention}[convention]{Convention}
\aliascntresetthe{convention}
\theoremstyle{remark}
\newaliascnt{remark}{theorem}
\newtheorem{remark}[remark]{Remark}
\aliascntresetthe{remark}

\DeclareMathOperator{\Ad}{Ad}
\DeclareMathOperator{\End}{End}
\DeclareMathOperator{\FPdim}{FPdim}
\DeclareMathOperator{\Dim}{Dim}
\DeclareMathOperator{\Hom}{Hom}
\DeclareMathOperator{\Inv}{Inv}
\DeclareMathOperator{\Irr}{Irr}
\DeclareMathOperator{\Kar}{Kar}
\DeclareMathOperator{\rank}{rank}

\DeclareMathOperator{\Tr}{Tr}
\DeclareMathOperator{\Gal}{Gal}

\newcommand{\Q}{\mathbb Q}
\newcommand{\Z}{\mathbb Z}
\newcommand{\CC}{\mathbb C}
\newcommand{\Fib}{\mathrm{Fib}}
\newcommand{\HFib}{H_{\mathrm{Fib}}}
\newcommand{\KFib}{K_{\mathrm{Fib}}}
\newcommand{\CoxeterRigidityTheorem}{3.1}
\newcommand{\CoxeterConstructionTheorem}{1.3}
\newcommand{\CoxeterRankTwoProposition}{10.1}
\newcommand{\CoxeterClassificationCorollary}{1.7}

\newcommand{\YL}{\mathcal{YL}}
\renewcommand{\Vec}{\mathrm{Vec}}
\newcommand{\TT}{\mathcal{TT}_3}
\newcommand{\TTthree}{\TT}
\newcommand{\Ising}{\mathrm{Ising}}
\newcommand{\boxt}{\boxtimes}
\newcommand{\calA}{\mathcal A}
\newcommand{\calB}{\mathcal B}
\newcommand{\calC}{\mathcal C}
\newcommand{\calD}{\mathcal D}
\newcommand{\calE}{\mathcal E}
\newcommand{\calP}{\mathcal P}
\newcommand{\rtw}{\rtimes}
\newcommand{\ph}{\varphi}
\newcommand{\id}{\boldsymbol{1}}

\definecolor{A3}{RGB}{0,150,0}
\definecolor{theta}{named}{blue}
\definecolor{rho}{named}{red}
\definecolor{mu}{named}{orange}
\tikzset{
  unshaded/.style={fill=white},
  rho/.style={color=rho,solid},
  mu/.style={color=mu,solid}
}

\title[Finite Fibonacci--pointed composites]{Finite Composites of the Fibonacci Category\\with Pointed Fusion Categories}
\author{Huixuan He}
\address{H. He, Department of Mathematics,
Tsinghua University, Beijing, China}
\email{hhx22@mails.tsinghua.edu.cn}
\author{Zhengwei Liu}
\address{Z. Liu, Yau Mathematical Sciences Center
and Department of Mathematics, Tsinghua University,
and Beijing Institute of Mathematical Sciences and Applications,
Beijing, China}
\email{liuzhengwei@mail.tsinghua.edu.cn}
\date{}
\hypersetup{pdftitle={Finite Composites of the Fibonacci Category with Pointed Fusion Categories},pdfauthor={Huixuan He; Zhengwei Liu}}

\begin{document}
\raggedbottom

\begin{abstract}
We classify finite composites of the positive Fibonacci category with
arbitrary pointed fusion categories, retaining the entire source functor
and making no unitarity assumption on the target. Formal codegrees,
dimension fields, and algebraic integrality exclude all dihedral Fibonacci
rings $R_n$ with $n\ge4$ from categorification. Together with an independent
graphical uniqueness argument, this gives the finite two-Fibonacci
classification. For several marked Fibonacci generators, a geometric
height-and-polygon argument identifies the based ring with a standard-basis
Coxeter--Hecke ring. Transitive conjugation by the pointed image forces the
noncommutation graph to be empty or a perfect matching. We then obtain a
crossed-product classification in terms of a quotient group, a stabilizer,
labelled Fibonacci or rank-six $\TT$ blocks, and compatible cocycle and
categorical action data. This includes the Fibonacci--$\Z/3\Z$ composite
problem without target unitarity. We also establish the corresponding
forms and marked descent statements over characteristic-zero fields
admitting the chosen Fibonacci-type input.
\end{abstract}

\maketitle

\section{Introduction}\label{sec:introduction}

The composition problem for subfactors asks which isomorphism classes of
inclusions $N\subset M$ arise from chains $N\subset P\subset M$ with
prescribed finite-index components $N\subset P$ and $P\subset M$, as their
relative position varies. Standard invariants describe the associated
bimodule structure and provide a categorical approach to this problem.  In the bimodule description, the two components have
even parts on the common $P$ side.  They generate a tensor category of
$P$--$P$ bimodules which is a quotient of their free product, and which is
a unitary fusion category when the composite has finite depth.  Subfactor
composition thus leads to the classification of such categorical quotients,
with the component inclusions retained by their canonical algebra objects,
or Q-systems.  Izumi--Morrison--Penneys make this correspondence explicit
for the small-index composites
\cite[Sections~1.1 and~4]{IzumiMorrisonPenneys2013}.

Bisch and Haagerup studied the group-type chains
$R^H\subset R\subset R\rtimes K$, where finite groups $H$ and $K$ act
outerly on the hyperfinite $\mathrm{II}_1$ factor $R$.  They proved that
the composite has finite depth exactly when $H$ and $K$ generate a finite
group in $\operatorname{Out}(R)$ \cite{BischHaagerup1996}.  The first
non-group-like case combines the pointed category $\Vec_{\Z/2\Z}$,
the even part of the $A_3$ subfactor, with the Fibonacci category, the
even part of the $A_4$ subfactor.  Liu classified the standard invariants
of the $A_3$--$A_4$ and $A_4$--$A_4$ composites
\cite[Theorems~1.1 and~1.2]{Liu2015}.  These results naturally lead to
the next cyclic pointed case: the finite composites of Fibonacci with
$\Vec_{\Z/3\Z}$.

Edie--Michell states a classification in the unitary cyclic setting
and identifies the nonunitary two-Fibonacci problem as an obstacle to
removing unitarity
\cite[Theorem~1.1 and Remark~1.2]{EdieMichell2022}.
Our argument develops the two-generator classification and the
transitive-symmetry block decomposition independently. The companion
paper \cite{FibCoxeter} examines the intermediate three-generator
claims used in the published cyclic-classification proof.
We resolve the Fibonacci--$\Z/3\Z$ problem without a unitarity assumption
as part of a classification for arbitrary pointed fusion categories.
We begin with the two-Fibonacci classification.

Write $\Fib$ for the positive complex Fibonacci category, generated by
$\tau$ with $\tau^{\otimes2}\cong\mathbf1\oplus\tau$, and $\TT$ for the
rank-six category associated to the $2D2$ subfactor.  A finite quotient is a
dominant tensor functor with fusion target.  Equivalence of quotients
retains the entire source functor, including its ordered factors.

\begin{theorem}[Two Fibonacci factors]\label{thm:intro-two-fibonacci}
Up to equivalence retaining the ordered source functors, the finite
quotients of $\Fib*\Fib$ are exactly
\[
 \Fib,\qquad \Fib\boxtimes\Fib,\qquad\TT.
\]
The first is the fold quotient.  In the other two cases the distinguished
generators $a,b$ are nonisomorphic and their first alternating relations
are, respectively, $ab\cong ba$ and $aba\cong bab$.
No unitary or unitarizable structure is assumed on the target.
\end{theorem}

The full statement is \cref{thm:two-fib-classification}.
Our first contribution is a new number-theoretic proof excluding a first
alternating relation of length $n\ge4$.  It uses formal codegrees, their
number fields, and algebraic integrality.  This differs from Liu's
graph-planar-algebra analysis and from the jellyfish consistency
calculations of Izumi--Morrison--Penneys
\cite{Liu2015,IzumiMorrisonPenneys2013,IzumiMorrisonPenneys2016}.
The arithmetic argument depends only on the fusion rules, so it applies
without a unitary structure or a choice of rank-two monoidal realization.
For the surviving cases, the graphical proof of ordered monoidal uniqueness
is given in Appendix~\ref{app:two-fibonacci-uniqueness}.

We next consider a fusion category generated by several pairwise
nonisomorphic Fibonacci objects with specified source embeddings.  Join two generators by an edge when they
do not commute.  We prove that this noncommutation graph determines the
entire fusion ring: it is the standard-basis Hecke ring $\HFib(W_\Gamma)$ of the associated
Coxeter group (\cref{thm:coxeter-rigidity}).  Thus the local two-generator
relations determine all fusion rules.  Finiteness forces the graph to be a
forest; the precise presentation and proof are given in
\cref{sec:many-fibonacci}.
The proof uses Frobenius--Perron height, two-colour polygons, and local
path elimination. A different left--right matching argument in the
companion paper \cite{FibCoxeter} proves rigidity for general rigid based
rings, even in infinite rank. That paper also constructs the ADE examples
and proves their marked algebraic uniqueness; our focus here is the
arithmetic obstruction and the pointed-composite classification.

Group symmetry makes this description particularly rigid.  If a group acts
by tensor autoequivalences and transitively permutes the generating
Fibonacci objects, every vertex has the same degree.  A finite regular forest is either empty or a perfect
matching.  At the categorical level, the resulting composite is therefore
a tensor product of copies of $\Fib$ in the first case, or of $\TT$ in the
second (\cref{thm:many-fib-blocks}).  The equivalence preserves the labelled
Fibonacci embeddings.

To apply this structure theorem to the pointed-composite problem, let $G$
be any finite group and $\Omega$ any pointed associator.  In a finite composite
\begin{equation}\label{eq:composite-intro}
 F:\Fib*\Vec_G^\Omega\longrightarrow\calE,
\end{equation}
the pointed image acts transitively by conjugation on the distinct images
of the Fibonacci generator.  The preceding result determines the
subcategory $\calB$ generated by these conjugates.  Its block structure
then yields the following classification.

\begin{theorem}[Fibonacci--pointed composites, abbreviated]
\label{thm:intro-classification}
Every finite composite \eqref{eq:composite-intro} has a crossed-product
presentation
\[
 \calE\simeq\calB\rtw_\rho^\omega H,
\]
where $H$ is a quotient of $G$ and $\calB$ is a tensor power of $\Fib$ or
of $\TT$.  Equivalence classes retaining the source functor are determined
by the quotient, the stabilizer of the Fibonacci generator, the labelled
block category, and compatible cocycle and marked categorical action data,
with the equivalence relation of \cref{thm:fib-pointed-classification}.
\end{theorem}

The explicit data and their equivalence are given in
\cref{sec:classification}.  The classification also has a ground-field form over fields
of characteristic zero admitting the chosen Fibonacci-type category
(\cref{sec:characteristic-zero-fields}).

\medskip
\noindent
After the preliminaries and formal-codegree tools,
\cref{sec:two-fibonacci,sec:algebraic-uniqueness} establish the two-generator
classification.  \Cref{sec:many-fibonacci} treats several generators, and
\cref{sec:neutral-reduction,sec:group-actions,sec:classification} give the
pointed classification.  Examples, ground-field forms, and further
questions follow.  Appendix~\ref{app:two-fibonacci-uniqueness} contains the graphical
proof and presentation used for the two-generator and ground-field results.

\section{Tensor-categorical preliminaries}\label{sec:preliminaries}

\subsection{Conventions and marked Fibonacci objects}

\begin{convention}\label{conv:standing}
Except in \cref{sec:characteristic-zero-fields}, we work over $\CC$.
All categories and tensor functors are linear over the specified ground
field.  A fusion category is finite semisimple and
rigid, with simple tensor unit.  The free-product sources used below are
semisimple rigid tensor categories with simple unit and finite-dimensional
Hom spaces, but generally have infinitely many simple objects.  Tensor
functors are linear functors equipped with invertible structure maps
\[
 F_2(X,Y):F(X)\otimes F(Y)\longrightarrow F(X\otimes Y),
 \qquad F_0:\mathbf1\longrightarrow F(\mathbf1),
\]
satisfying the usual coherence axioms.  For a fusion category
$\calD$, $\Inv(\calD)$ denotes the group of isomorphism classes of invertible
simple objects; brackets such as $[u]$ denote the class of an invertible
object $u$.
\end{convention}

The symbol $\Fib$ denotes the usual unitarizable Fibonacci category.  A
\emph{marked Fibonacci object} in $\calC$ is a simple object $x$ together
with a specified tensor equivalence
\[
 \iota_x:\Fib\xrightarrow{\ \simeq\ }\langle x\rangle,
 \qquad \iota_x(\tau)=x.
\]
By contrast, an object satisfying the Fibonacci fusion rule is only an
object $x$ with $x^{\otimes2}\cong\mathbf1\oplus x$.

\begin{remark}[Marked type and Yang--Lee]\label{rem:marked-types}
Over $\CC$, the Moore--Seiberg rank-two classification, as recalled by
Ostrik, gives two monoidal categorifications of the fusion rule
$x^2=\mathbf1+x$: the usual Fibonacci category and its Yang--Lee Galois
conjugate $\YL$ \cite[Introduction, Main Theorem, and Section~2.5]{Ostrik2003}.
Consequently an ordered pair of fusion-rule objects may have source type
\[
 (\Fib,\Fib),\quad(\Fib,\YL),\quad(\YL,\Fib),\quad(\YL,\YL).
\]
The marked quotient problem in this paper uses only the first type.  The
based-ring and formal-codegree arguments in Sections~\ref{sec:formal-codegrees}
and~\ref{sec:two-fibonacci} depend only on the fusion rule; monoidal
uniqueness statements do depend on the marked source.  In particular, the
mixed source problems cannot be inferred from the common based ring $R_n$.
\end{remark}

\begin{lemma}[Rigidity of the usual Fibonacci marking]
\label{lem:fib-marking-rigidity}
Let $\iota_1,\iota_2:\Fib\to\calC$ be tensor embeddings into a fusion
category.  If
\[
 \iota_1(\tau)\cong\iota_2(\tau),
\]
then there is a unique monoidal natural isomorphism
$\iota_1\Rightarrow\iota_2$.
\end{lemma}

\begin{proof}
Let $a$ denote the common isomorphism class of the two nontrivial images and
let $\calA=\langle a\rangle$ be the replete tensor subcategory that it
generates.  Both embeddings corestrict to tensor equivalences
$\bar\iota_i:\Fib\xrightarrow{\simeq}\calA$.  After choosing a tensor
quasi-inverse to $\bar\iota_2$, their comparison is a tensor
autoequivalence of $\Fib$.  The rank-two form calculation of
\cite[Section~4.2.5]{EtingofGelaki2012} states that the Fibonacci/Yang--Lee
rank-two category has no nontrivial tensor autoequivalence.  Hence this
comparison is monoidally isomorphic to the identity, and transporting that
isomorphism through $\bar\iota_2$ gives a monoidal natural isomorphism
$\iota_1\Rightarrow\iota_2$.

For uniqueness, let $\lambda,\lambda':\iota_1\Rightarrow\iota_2$ be two such
isomorphisms.  Their ratio is a monoidal natural automorphism of $\iota_1$.
It is the identity on the unit and is a scalar, say $z$, on the nontrivial
simple object.  Its restrictions to the two summands $\mathbf1$ and $\tau$
of $\tau^{\otimes2}$ give respectively $z^2=1$ and $z^2=z$.  Thus $z=1$,
and the two transformations agree on the tensor generator and hence on all
of $\Fib$.
\end{proof}

\begin{remark}[Galois conjugation]
Until \cref{sec:characteristic-zero-fields}, the classification statements
in this paper are $\CC$-linear.  Let
$\upsilon\in\operatorname{Aut}(\CC)$ interchange the two embeddings of
$\Q(\sqrt5)$.  Twisting scalar multiplication and every structure morphism
defines ${}^{\upsilon}\calC$; the natural correspondence
$\calC\to{}^{\upsilon}\calC$ is $\upsilon$-semilinear, not a
$\CC$-linear tensor equivalence.  A single global automorphism conjugates
both source factors, so it sends a Fib--Fib pair to a YL--YL pair, not to a
mixed pair.  It also acts coefficientwise on all pointed and action data:
\[
 \Omega\mapsto\upsilon(\Omega),\qquad
 \omega\mapsto\upsilon(\omega),\qquad
 \eta\mapsto\upsilon(\eta),
\]
and sends the rank-six model $\TT$ to ${}^{\upsilon}\TT$.  Thus, for a
fixed cocycle $\Omega$, the direct correspondence is
\[
 \operatorname{Class}_{G}(\Fib,\upsilon^{-1}\Omega)
 \longrightarrow
 \operatorname{Class}_{G}(\YL,\Omega),
\]
where $\operatorname{Class}_{G}(\mathcal F,\Omega)$ denotes equivalence classes
of finite composites $\mathcal F*\Vec_G^\Omega\to\calE$ retaining the entire
source functor.  The correspondence depends on the chosen extension of
$\upsilon$.
Forms and descent over fields that need not be algebraically closed are
treated in \cref{sec:characteristic-zero-fields}.
\end{remark}

\subsection{Hecke notation}
Write $\ph=(1+\sqrt5)/2$ and $\KFib=\Q(\sqrt5)$.
For a Coxeter system $(W,S)$, let
\[
 \HFib(W)=\bigoplus_{w\in W}\Z h_w
\]
be the standard-basis Hecke ring with generator multiplication
\[
 h_sh_w=\begin{cases}
 h_{sw},&\ell(sw)=\ell(w)+1,\\
 h_{sw}+h_w,&\ell(sw)=\ell(w)-1.
 \end{cases}
\]
In particular $h_s^2=1+h_s$, and $h_w^*=h_{w^{-1}}$.
The positive dimension function is $d(h_w)=\ph^{\ell(w)}$; for a
fusion-category categorification it is the Frobenius--Perron dimension.
This is the standard Hecke basis, not the Kazhdan--Lusztig basis.
The alternative generators $T_s=\ph h_s$ satisfy
$(T_s-\ph^2)(T_s+1)=0$ over $\KFib$; this explains the notation
$H_{\ph^2}(W)$ used for that scalar-extended Hecke algebra below.
The basis and multiplication follow from the standard Hecke presentation
\cite[Chapter~4]{GeckPfeiffer2000}.

\subsection{The fixed rank-six model}

Let $\TT$ be the dual even part of the $2D2$ subfactor, with simple objects
\[
 \mathbf1,\ a,\ b,\ ab,\ ba,\ c,
 \qquad c=aba=bab,
\]
and Frobenius--Perron dimensions
\[
 1,\ \ph,\ \ph,\ \ph^2,\ \ph^2,\ \ph^3.
\]
Right multiplication by the two generators is
\begin{equation}\label{eq:TT-fusion-table}
\begin{array}{c|cc}
 X&X a&X b\\ \hline
 \mathbf1&a&b\\
 a&\mathbf1+a&ab\\
 b&ba&\mathbf1+b\\
 ab&c&a+ab\\
 ba&b+ba&c\\
 c&ab+c&ba+c
\end{array}
\end{equation}
Fix an embedding $\iota_a:\Fib\to\TT$.  The order-two planar automorphism of
the $2D2$ planar algebra acts on even projections and on their corner
morphism spaces.  Because it preserves tensor tangles, it induces a tensor
autoequivalence $\sigma$ of the additive Karoubi envelope.  Its square is
strictly the identity on the planar generator and hence supplies the
canonical action coherence
$\epsilon:\sigma^2\Rightarrow\operatorname{id}_{\TT}$.  Define
$\iota_b=\sigma\iota_a$; then $\sigma$ interchanges the two ordered marked
Fibonacci subcategories.  The model, fusion table, and planar involution
follow from
\cite[Section~5]{IzumiMorrisonPenneys2013},
\cite[Theorem~1.1 and Corollary~1.7]{MorrisonPenneys2015}, and the
planar-to-tensor correspondence of \cite[Theorem~A]{HPT2023}; compare
\cite[Example~2.8]{EdieMichell2022}.

\subsection{Free products and finite quotients}

Let $S$ be finite.  We write $\Fib^{*S}$ for the free product of labelled
copies of $\Fib$.  The factor embeddings are fully faithful and generate
$\Fib^{*S}$ under tensor products, direct sums, and summands
\cite[Section~3]{Yamagami2002}.

\begin{definition}[Finite quotient]
A finite quotient of a semisimple rigid $\CC$-linear tensor category with
simple unit is a faithful dominant tensor functor to a fusion category.
Here faithful means injective on Hom spaces; dominant means that every target
object is a summand of an object in the image.
\end{definition}

Faithfulness does not prevent two nonisomorphic source objects from having
isomorphic images.  Thus it is compatible with relations $u_n\cong v_n$ and
with quotient homomorphisms $G\twoheadrightarrow H$.

\begin{lemma}[Automatic faithfulness]\label{lem:automatic-faithfulness}
Every tensor functor from a semisimple rigid tensor category
with simple unit is faithful.
\end{lemma}

\begin{proof}
If $S$ is simple, then $\mathbf1$ is a summand of $S^*\otimes S$, so a tensor
functor cannot send $S$ to zero.  Let $0\ne f:X\to Y$.  Choose simple
summands $S\subset X$ and $T\subset Y$ for which the matrix entry
$p_Tfi_S:S\to T$ is nonzero.  A nonzero morphism between simple objects is
an isomorphism, and its image is a nonzero isomorphism.  Hence the image of
$f$ is nonzero.
\end{proof}

The free product has the following functorial property in the form needed
below: labelled tensor functors $F_s:\Fib\to\calD$ determine, up to a unique
monoidal natural isomorphism fixed on the factors, a tensor functor
$*_{s\in S}F_s:\Fib^{*S}\to\calD$.  This follows directly from the
alternating-word construction in \cite[Section~3]{Yamagami2002}: define the
functor on reduced words factor by factor and use the factor tensorators for
every reduction.  The factor pentagons give independence of the reduction
order.  In particular, a tensor-generating family of marked Fibonacci
objects in a fusion category $\calB$ gives a finite quotient
\[
 Q:\Fib^{*S}\longrightarrow\calB,
 \qquad \tau_s\longmapsto x_s.
\]

\subsection{Pointed functors and crossed products}

A pointed fusion category is $\Vec_G^\Omega$ for a finite group $G$ and a
normalized $3$-cocycle $\Omega\in Z^3(G,\CC^\times)$.  Consider a tensor
functor
\[
 \Vec_G^\Omega\longrightarrow\Vec_H^\omega,
 \qquad \omega\in Z^3(H,\CC^\times).
\]
If the functor is dominant, its map on simple objects is an epimorphism
$q:G\twoheadrightarrow H$.  We use the multiplicative inhomogeneous
differential
\[
 (d\eta)(g,h,k)=
 \frac{\eta(h,k)\eta(g,hk)}{\eta(gh,k)\eta(g,h)}.
\]
A normalized cochain $\eta\in C^2(G,\CC^\times)$ gives its tensorator
precisely when
\begin{equation}\label{eq:eta-differential}
 d\eta=\frac{\Omega}{q^*\omega}.
\end{equation}

Let $\rho$ be a normalized categorical action of $H$ on $\calB$.  Thus each
$\rho_h$ is a tensor autoequivalence with specified tensor structure
\[
 J^h_{X,Y}:\rho_h(X)\otimes\rho_h(Y)\longrightarrow
 \rho_h(X\otimes Y),
 \qquad J^h_0:\mathbf1\longrightarrow\rho_h(\mathbf1),
\]
and there are compositors $\mu_{h,k}:\rho_h\rho_k\Rightarrow\rho_{hk}$
satisfying the usual unit and action pentagons.  The maps $J^h$ are part of
the action datum.  For a normalized
$\omega\in Z^3(H,\CC^\times)$, the crossed product
$\calB\rtw_\rho^\omega H$ is the additive category whose objects are
finite direct sums of homogeneous objects $(X,h)$, with morphisms between
homogeneous objects given by
\[
 \Hom((X,h),(Y,k))=
 \begin{cases}\Hom_\calB(X,Y),&h=k,\\0,&h\ne k,\end{cases}
\]
morphisms between direct sums given by matrices, unit $(\mathbf1,e)$,
and tensor product on homogeneous objects
\[
 (X,h)\otimes(Y,k)=\bigl(X\otimes\rho_h(Y),hk\bigr).
\]
The tensor product extends to finite direct sums by distributivity.
Its associator on homogeneous objects is $\omega(h,k,\ell)$ times the composite
\begin{align*}
 &(X\otimes\rho_hY)\otimes\rho_{hk}Z
   \longrightarrow X\otimes(\rho_hY\otimes\rho_{hk}Z)\\
 &\longrightarrow X\otimes(\rho_hY\otimes\rho_h\rho_kZ)
   \longrightarrow X\otimes\rho_h(Y\otimes\rho_kZ),
\end{align*}
where the arrows are respectively the associator of $\calB$,
$\operatorname{id}\otimes(\operatorname{id}\otimes\mu_{h,k,Z}^{-1})$, and
$\operatorname{id}\otimes J^h_{Y,\rho_kZ}$.  The action coherence and
$d\omega=1$ give the pentagon.  The associator and unit constraints
extend to finite direct sums componentwise.

\subsection{Adjoint subcategories and dimensions}

The adjoint subcategory $\Ad(\calE)$ is generated by all simple summands of
$X\otimes X^*$ as $X$ ranges over the simple objects of $\calE$.
If $x^2\cong\mathbf1\oplus x$, then $x\cong x^*$ and
$x\in\Ad(\calE)$.

\begin{lemma}[Fibonacci simplicity]\label{lem:fib-simplicity}
If an object $x$ has Grothendieck class $[x]^2=[\mathbf1]+[x]$, then $x$ is
simple, self-dual, and $\FPdim(x)=\ph$.
\end{lemma}

\begin{proof}
FP-dimension gives $d^2=1+d$, hence $d=\ph$.  A sum of two nonzero objects
has FP-dimension at least $2$, contradicting $\ph<2$.  The unit summand of
$x\otimes x$ then gives $x\cong x^*$ by rigidity.
\end{proof}

Hom-space dimension, Frobenius--Perron dimension, squared norm, and
intrinsic global dimension are distinct.  We write
$\dim_\CC\Hom(X,Y)$, $\FPdim(X)$, $|X|^2$, and
\[
 \Dim(\calC)=\sum_{X\in\Irr(\calC)}|X|^2,
\]
respectively.

\section{Formal codegrees and arithmetic obstructions}\label{sec:formal-codegrees}

Let $(R,B)$ be a finite-rank based ring with involution $b\mapsto b^*$.
For an irreducible representation $E$ of $R\otimes_{\mathbb Z}\CC$ with
character $\chi_E$, its formal codegree is the scalar by which the central element
\[
 \sum_{b\in B}\operatorname{Tr}_E(b)\,b^*
\]
acts on $E$.  Equivalently,
\begin{equation}\label{eq:formal-codegree-character}
 f_E=\frac{1}{\dim E}\sum_{b\in B}\chi_E(b)\chi_E(b^*).
\end{equation}

An algebraic integer $\alpha$ is a \emph{$d$-number} if the principal ideal $(\alpha)$ in the ring of all algebraic integers is invariant under $\Gal(\overline{\Q}/\Q)$.  We use the following three formal-codegree obstructions.

\begin{theorem}[Formal-codegree obstructions]\label{thm:ostrik-package}
Let $R=K_0(\calC)$ for a fusion category $\calC$.
\begin{enumerate}[label=\textup{(\roman*)}]
\item Every formal codegree of $R$ is a $d$-number~\cite[Theorem~1.2]{Ostrik2009}.
\item For every $E\in\Irr(R)$, the quotient $\Dim(\calC)/f_E$ is an algebraic integer~\cite[Corollary~2.14]{Ostrik2015}; this is the meaning of ``$f_E$ divides $\Dim(\calC)$'' in that result.
\item If $\calC$ is spherical, every formal codegree lies in the number field generated by the spherical dimensions of its simple objects~\cite[Corollary~2.15]{Ostrik2015}.
\end{enumerate}
If $\widetilde{\calC}$ is the sphericalization, let
\[
 \pi:K_0(\widetilde{\calC})\twoheadrightarrow K_0(\calC)
\]
be the homomorphism induced by the forgetful functor.  For an irreducible
representation $E$ of $K_0(\calC)$, set $\widetilde E:=E\circ\pi$.  Then
\begin{equation}\label{eq:sphericalization-codegree}
 f_{\widetilde E}=2f_E.
\end{equation}
Here $\Dim(\calC)$ is the intrinsic global dimension (the sum of the squared
norms of the simple objects), which is the convention used in
\cite{Ostrik2015}.  The factor $2$ in the last identity is the explicit
sphericalization calculation of \cite[Section~3.2]{Ostrik2009}.
\end{theorem}

\section{Two Fibonacci generators: based rings and categorifiability}\label{sec:two-fibonacci}

\subsection{Alternating-word standard form}\label{sec:standard-rings}

Let $a,b$ be self-dual based-ring elements satisfying
\[
 a^2=1+a,
 \qquad b^2=1+b.
\]
Write $u_k$ and $v_k$ for the alternating words of length $k$ beginning
with $a$ and $b$.  Let $R_\infty$ be the free Fibonacci based ring, with
basis $1,u_k,v_k$ $(k\geq1)$, and for $n\geq2$ put
\[
 R_n=R_\infty/(u_n-v_n),\qquad c_n:=u_n=v_n.
\]

\begin{proposition}\label{prop:standard-form}
Let $\calC$ be a semisimple rigid $\CC$-linear tensor category with simple
unit and finite-dimensional Hom spaces, tensor-generated by distinct simple
objects $a,b$ satisfying $a^2\cong\mathbf1\oplus a$ and
$b^2\cong\mathbf1\oplus b$.
If $u_m\not\cong v_m$ for every $m$, then
\[
 \Irr(\calC)=\{\mathbf1,u_m,v_m:m\geq1\},
 \qquad K_0(\calC)\cong R_\infty.
\]
Otherwise let $n$ be least with $u_n\cong v_n$.  Then
\[
 \Irr(\calC)=
 \{\mathbf1,u_1,\ldots,u_{n-1},v_1,\ldots,v_{n-1},c_n\},
 \qquad K_0(\calC)\cong R_n.
\]
In particular, in the second case $\calC$ is automatically a fusion
category.
\end{proposition}

\begin{proof}
Put $\langle X,Y\rangle=\dim_\CC\Hom(X,Y)$.  The generators are self-dual
by rigidity.  We induct up to the first equality of alternating words.
At the step from length $k$ to length $k+1$, assume that the alternating
words of lengths at most $k$ are simple and pairwise nonisomorphic.
Moving an endpoint across a Hom space, using
$x^2\cong\mathbf1\oplus x$, and applying these lower-length Hom vanishings
gives
\begin{equation}\label{eq:endomorphism-induction}
 \langle u_{k+1},u_{k+1}\rangle=1+\langle u_k,v_k\rangle,
 \qquad
 \langle v_{k+1},v_{k+1}\rangle=1+\langle v_k,u_k\rangle.
\end{equation}
If the first equality occurs at length $n$, these induction identities are
used for $1\le k<n$; if there is no equality, they apply at every step.
The same endpoint moves reduce every pairing of a new word with an earlier
word to an off-diagonal pairing already known to vanish.  Starting from the
distinct simples $u_1=a$ and $v_1=b$, strong induction therefore shows that,
until the first equality $u_n\cong v_n$, all alternating words are simple
and pairwise nonisomorphic; at each length only the pair $u_k,v_k$ can be
the first equality.  This is an adaptation of the Frobenius-reciprocity method in
\cite[Proposition~3.2, Corollary~3.3, and Section~5]{IzumiMorrisonPenneys2013}
(compare also \cite[Proposition~2.2]{IzumiMorrisonPenneys2016}).
The cited propositions concern a unitary invertible--Fibonacci pair,
whereas Section~5 discusses two Fibonacci factors. The Hom-space induction
above is the argument used here; it requires neither unitarity, finiteness,
nor Frobenius--Perron dimensions.

If there is no equality, reducing every $aa$ and $bb$ shows that the listed
alternating words exhaust the simples and gives $R_\infty$.  If $n$ is
minimal, replacing the length-$n$ prefix by the other alternating word gives
\begin{equation}\label{eq:n-plus-one-reductions}
 u_{n+1}\cong v_{n-1}\oplus c_n,
 \qquad v_{n+1}\cong u_{n-1}\oplus c_n.
\end{equation}
These identities reduce all longer alternating words.  Hence the displayed
$2n$ simples exhaust $\calC$, and their fusion rules are precisely those of
$R_n$.
\end{proof}

\begin{corollary}\label{cor:free-two-fib-quotient}
Let $Q:\Fib*\Fib\to\calC$ be a faithful dominant tensor functor to a
semisimple rigid $\CC$-linear tensor category with simple unit and
finite-dimensional Hom spaces.  If $Q(u_m)\not\cong Q(v_m)$ for every
$m\geq1$, then $Q$ is a tensor equivalence.
\end{corollary}

\begin{proof}
The source and target have the same simple-object list by
\cref{prop:standard-form}.  On each pair of simples the induced Hom map is
therefore an isomorphism (both spaces are zero or both are one-dimensional),
so $Q$ is fully faithful; dominance gives essential surjectivity.
\end{proof}

The coincident-generator case gives the single Fibonacci category and will
be recorded as $n=1$ when needed.

In the notation of \cref{sec:preliminaries}, the same quadratic and
alternating braid presentations identify
\[
 R_n\cong\HFib(I_2(n))\quad(n\ge2),\qquad
 R_\infty\cong\HFib(I_2(\infty)),
\]
with $a,b$ corresponding to $h_s,h_t$. The standard bases agree by
\cref{prop:standard-dihedral-model}; in the infinite case they are the
reduced alternating words. We retain $R_n$ for the arithmetic calculations.

\subsection{The finite standard rings}

\begin{proposition}[Standard dihedral model]\label{prop:standard-dihedral-model}
For $n\geq2$, $R_n$ is free of rank $2n$ with basis
\begin{equation}\label{eq:Rn-basis}
 B_n=\{1,u_1,\ldots,u_{n-1},v_1,\ldots,v_{n-1},c_n\}.
\end{equation}
After $T_a=\ph a$ and $T_b=\ph b$, its scalar extension to
$\KFib$ is the dihedral Iwahori--Hecke algebra
$H_{\ph^2}(I_2(n))$.
\end{proposition}

\begin{proof}
We reduce words by induction on length. An adjacent repeated letter is
shortened using $a^2=1+a$ or $b^2=1+b$. If the word is alternating and
has length greater than $n$, apply the length-$n$ braid relation to its
first $n$ letters. The last letter of the replaced segment now agrees
with the next letter, so a quadratic relation shortens the word.
Alternating words of length at most $n$ belong to $B_n$, with the two
length-$n$ words identified as $c_n$. Thus every word is a nonnegative
integral combination of $B_n$.  The irreducible modules in
\cref{lem:character-completeness} have squared dimensions summing to $2n$;
hence the spanning set is linearly independent and is an integral basis.
Reversal gives its based-ring involution.  Finally
\[
 (T_a-\ph^2)(T_a+1)=(T_b-\ph^2)(T_b+1)=0,
\]
and $u_n=v_n$ is the length-$n$ braid relation, which is the stated Hecke
presentation.
\end{proof}

\subsection{Irreducible representations}

Set $\lambda_+=\ph$ and $\lambda_-=-\ph^{-1}$.  These are the roots of $x^2=x+1$.  When $n$ is odd there are two one-dimensional characters, with $(a,b)=(\lambda_+,\lambda_+)$ and $(\lambda_-,\lambda_-)$.  When $n$ is even, all four choices $(\lambda_\epsilon,\lambda_\delta)$ occur.

The two-dimensional representations are indexed by
\[
 1\le j\le\left\lfloor\frac{n-1}{2}\right\rfloor.
\]
Write $\alpha_j=\pi j/n$.  In $E_j$ the alternating character values satisfy
\begin{equation}\label{eq:alternating-character}
 \chi_j(u_k)=\chi_j(v_k)=c_k^{(j)},
 \qquad
 c_k^{(j)}=
 \begin{cases}
  2\cos(k\alpha_j),&k\text{ even},\\[2mm]
  \dfrac{\sin(k\alpha_j)}{\sin\alpha_j},&k\text{ odd}.
 \end{cases}
\end{equation}

\begin{lemma}[Explicit irreducible modules]\label{lem:character-completeness}
For $1\leq j\leq\lfloor(n-1)/2\rfloor$, put
\[
 \alpha=\frac{\pi j}{n},\qquad t=2\cos(2\alpha),\qquad
 p=\frac{t+\ph^{-1}}{\sqrt5},\qquad r=1+p-p^2,
\]
and define
\begin{equation}\label{eq:explicit-two-dimensional-matrices}
 A_j=\begin{pmatrix}\ph&0\\0&-\ph^{-1}\end{pmatrix},
 \qquad
 B_j=\begin{pmatrix}p&1\\r&1-p\end{pmatrix}.
\end{equation}
Then $a\mapsto A_j$, $b\mapsto B_j$ defines an irreducible representation $E_j$ of $R_n$.  Together with the one-dimensional representations listed above, these are all irreducible representations of $R_n\otimes\CC$, and their alternating character values are those of \eqref{eq:alternating-character}.
\end{lemma}

\begin{proof}
Direct calculation gives
\begin{equation}\label{eq:explicit-matrix-identities}
 A_j^2=I+A_j,\quad B_j^2=I+B_j,\quad
 \det(A_jB_j)=1,\quad \Tr(A_jB_j)=t,
\end{equation}
and
\begin{equation}\label{eq:explicit-braid-reduction}
 A_jB_jA_j-B_jA_jB_j=(t+1)(A_j-B_j).
\end{equation}
Write $X=A_jB_j$.  Cayley--Hamilton says
\begin{equation}\label{eq:X-chebyshev}
 X^m=U_{m-1}(t/2)X-U_{m-2}(t/2)I,
\end{equation}
where $U_k$ is the Chebyshev polynomial of the second kind and $U_{-1}=0$.  The eigenvalues of $X$ are $e^{\pm2i\alpha}$.

If $n=2m$, then $X^m=(-1)^jI$.  Since $(B_jA_j)^m=A_j^{-1}X^mA_j$, the two alternating words of length $n$ agree.  If $n=2m+1$, then \eqref{eq:explicit-braid-reduction} and \eqref{eq:X-chebyshev} give
\[
 X^mA_j-B_jX^m
 =\bigl((t+1)U_{m-1}(t/2)-U_{m-2}(t/2)\bigr)(A_j-B_j).
\]
With $t=2\cos(2\alpha)$, the scalar in parentheses is
\[
 \frac{\sin((2m+1)\alpha)}{\sin\alpha}=0.
\]
Thus $u_n=v_n$ in both parities.

The two eigenspaces of $A_j$ are the coordinate lines.  The upper-right entry of $B_j$ is $1$.  Moreover $r=0$ would force $p\in\{\ph,-\ph^{-1}\}$ and hence $t\in\{3,-2\}$, whereas $-2<t<2$ for the stated range of $j$.  Thus neither coordinate line is $B_j$-invariant, proving irreducibility.

For an even alternating word of length $2s$,
\[
 \Tr(X^s)=2\cos(2s\alpha).
\]
For an odd word put $c_{s}^{\mathrm{odd}}=\Tr(X^sA_j)$.  Cayley--Hamilton gives
\[
 c_{s+1}^{\mathrm{odd}}=t c_{s}^{\mathrm{odd}}-c_{s-1}^{\mathrm{odd}},\qquad c_{0}^{\mathrm{odd}}=1,\qquad c_{1}^{\mathrm{odd}}=t+1,
\]
so $c_{s}^{\mathrm{odd}}=\sin((2s+1)\alpha)/\sin\alpha$.  Cyclicity of trace gives the same values for words beginning with $b$.  This proves \eqref{eq:alternating-character}.

The numbers $t=2\cos(2\pi j/n)$ in the chosen range are distinct, so the $E_j$ are pairwise nonisomorphic.  Finally,
\[
 2\cdot1^2+\frac{n-1}{2}\cdot2^2=2n\quad(n\text{ odd}),
 \qquad
 4\cdot1^2+\left(\frac n2-1\right)2^2=2n\quad(n\text{ even}).
\]
The length-reduction procedure in the proof of
\cref{prop:standard-dihedral-model} expresses every word in the
$2n$ standard words, so the presented complex algebra has
dimension at most $2n$.  The displayed pairwise nonisomorphic irreducible modules give
a semisimple quotient of dimension at least the sum above, namely $2n$.
Consequently the algebra has dimension $2n$, its radical is zero, and the
displayed representations are complete.
\end{proof}

\subsection{Two-dimensional formal codegrees}

\begin{proposition}\label{prop:two-dimensional-codegrees}
For $1\le j\le\lfloor(n-1)/2\rfloor$, the formal codegree of $E_j$ is
\begin{equation}\label{eq:two-fib-codegree}
 f_{n,j}=n\left(1+\frac{1}{4\sin^2(\pi j/n)}\right).
\end{equation}
\end{proposition}

\begin{proof}
Reversal is the based-ring involution.  Thus $u_k^*=u_k$ and $v_k^*=v_k$ for odd $k$, while $u_k^*=v_k$ and $v_k^*=u_k$ for even $k$; also $c_n^*=c_n$.  Since \eqref{eq:alternating-character} is real and takes the same value on each paired word, \eqref{eq:formal-codegree-character} gives
\begin{equation}\label{eq:codegree-expanded}
 2f_{n,j}=4
 +8\!\sum_{\substack{1\le k<n\\k\ \mathrm{even}}}\!\cos^2(k\alpha_j)
 +\frac{2}{\sin^2\alpha_j}
   \!\sum_{\substack{1\le k<n\\k\ \mathrm{odd}}}\!\sin^2(k\alpha_j)
 +(c_n^{(j)})^2.
\end{equation}
For even $n$ the required finite sums are
\[
 \sum_{\substack{1\le k<n\\k\ \mathrm{even}}}\cos^2(k\alpha_j)=\frac n4-1,
 \qquad
 \sum_{\substack{1\le k<n\\k\ \mathrm{odd}}}\sin^2(k\alpha_j)=\frac n4,
 \qquad (c_n^{(j)})^2=4.
\]
For odd $n$ they are
\[
 \sum_{\substack{1\le k<n\\k\ \mathrm{even}}}\cos^2(k\alpha_j)=\frac{n-2}{4},
 \qquad
 \sum_{\substack{1\le k<n\\k\ \mathrm{odd}}}\sin^2(k\alpha_j)=\frac n4,
 \qquad (c_n^{(j)})^2=0.
\]
These identities follow by writing sine and cosine squares in terms of $e^{2ik\alpha_j}$ and summing the resulting finite geometric series.  In either parity, the right-hand side of \eqref{eq:codegree-expanded} is
\[
 2n+\frac{n}{2\sin^2\alpha_j}.
\]
Division by two proves \eqref{eq:two-fib-codegree}.
\end{proof}

\subsection{Sphericalization and the dimension field}

Recall $\KFib=\Q(\sqrt5)$ and $\psi=(1-\sqrt5)/2=-\ph^{-1}$.

\begin{lemma}[Reality of spherical dimensions]\label{lem:real-spherical-dimensions}
Every categorical dimension in a spherical fusion category over $\CC$ is
real.
\end{lemma}

\begin{proof}
Let $d$ be the dimension character of the Grothendieck ring and let $x$ be a
simple basis element.  Sphericality gives $d(x^*)=d(x)$.  In the regular
representation, $L_{x^*}=L_x^{\mathsf T}$, so $L_{x+x^*}$ is a real
symmetric matrix.  The number
$d(x+x^*)=2d(x)$ is one of its eigenvalues and is therefore real.
Additivity gives the assertion for every object.
\end{proof}

\begin{proposition}[Dimension-field obstruction]\label{prop:Rn-dimension-field}
If $R_n$ is the Grothendieck ring of a fusion category, then every formal
codegree of $R_n$ belongs to $\KFib$.
\end{proposition}

\begin{proof}
Let $P:\widetilde{\calC}\to\calC$ be the sphericalization and let
$\varepsilon$ be its central sign object, so that
$P(\varepsilon)\cong\id$ and $\dim(\varepsilon)=-1$.  Choose a simple lift
$A$ of $a$.  Every simple object $X$ of $\calC$ has exactly two simple lifts,
which differ by tensoring with $\varepsilon$.  Since $P(A^2)\cong
a^2\cong\id\oplus a$, with both summands of multiplicity one, $A^2$ contains
one lift of $\id$ and one lift of $a$.  Thus there are $e,f\in\{0,1\}$ such
that
\[
 A^2\cong\varepsilon^e\oplus\varepsilon^fA.
\]
Writing $d=\dim(A)\in\mathbb R$ by
\cref{lem:real-spherical-dimensions}, we obtain
$d^2=(-1)^e+(-1)^fd$.  If $e=1$, this is one of
$d^2-d+1=0$ and $d^2+d+1=0$, neither of which has a real root.  Hence
$e=0$.  Replacing $A$ by $\varepsilon A$ when $f=1$, we may arrange
\[
 A^2\cong\id\oplus A,
 \qquad \dim(A)\in\{\ph,\psi\}\subset \KFib.
\]
The same argument gives a lift $B$ of $b$ with
$B^2\cong\id\oplus B$ and $\dim(B)\in\{\ph,\psi\}$.

Every basis element in \eqref{eq:Rn-basis} is represented by an alternating
word in $a,b$.  The corresponding word in $A,B$ is simple: otherwise its
image under $P$ would decompose although the underlying alternating word is
simple.  It is a lift of that basis element and has dimension in $\KFib$; its
other lift is obtained by tensoring with $\varepsilon$ and has the negative
dimension.  These are all the simple objects of $\widetilde{\calC}$.  Thus
the field generated by their spherical dimensions is contained in $\KFib$.

Pull an irreducible representation $E$ of $R_n$ back along
$K_0(\widetilde{\calC})\twoheadrightarrow K_0(\calC)$.  Its formal
codegree is $2f_E$ by \eqref{eq:sphericalization-codegree}.  Applying
\cref{thm:ostrik-package}(iii) to $\widetilde{\calC}$ gives
$2f_E\in \KFib$, and hence $f_E\in \KFib$.
\end{proof}

\subsection{Squared norms and intrinsic global dimension}

For $\lambda\in\{\lambda_+,\lambda_-\}$, let $\chi_\lambda$ be the diagonal
one-dimensional character with $\chi_\lambda(a)=\chi_\lambda(b)=\lambda$,
and put $r_\lambda=\lambda^2$.  Its formal codegree is
\[
 D_n(r_\lambda)=1+2\sum_{k=1}^{n-1}r_\lambda^k+r_\lambda^n
       =(1+r_\lambda)\frac{r_\lambda^n-1}{r_\lambda-1}.
\]
Consequently
\begin{equation}\label{eq:diagonal-codegrees}
 D_n^+=\sqrt5(\ph^{2n}-1),
 \qquad
 D_n^-=\sqrt5(1-\ph^{-2n}),
 \qquad
 \frac{D_n^+}{D_n^-}=\ph^{2n}.
\end{equation}
For even $n$, the two mixed characters have
\begin{equation}\label{eq:mixed-codegree}
 D_n^{\mathrm{mix}}=\frac{5n}{2}.
\end{equation}

\begin{lemma}[Possible intrinsic dimensions]\label{lem:possible-intrinsic-dimensions}
If $K_0(\calC)\cong R_n$, then the squared norm of each generator belongs to
$\{\ph^2,\ph^{-2}\}$.  Moreover,
\begin{equation}\label{eq:possible-global-dimensions}
 \Dim(\calC)\in
 \begin{cases}
  \{D_n^+,D_n^-\},&n\text{ odd},\\
  \{D_n^+,D_n^-,D_n^{\mathrm{mix}}\},&n\text{ even}.
 \end{cases}
\end{equation}
\end{lemma}

\begin{proof}
The lift calculation in the proof of \cref{prop:Rn-dimension-field} shows
that a simple object satisfying the Fibonacci fusion rule has squared norm
$\ph^2$ or $\psi^2=\ph^{-2}$.  Squared norms multiply on simple tensor
products.  Summing the corresponding monomials over
\eqref{eq:Rn-basis} gives $D_n^+$ when both choices are $\ph^2$, $D_n^-$
when both are $\ph^{-2}$, and $D_n^{\mathrm{mix}}$ when the choices differ.
If $n$ is odd, the two sides of $u_n=v_n$ contain unequal numbers of $a$ and
$b$, so equality of their squared norms forces the two choices to agree.
If $n$ is even, the relation imposes no further condition.
\end{proof}

\subsection{Uniform reduction of the parameter}

\begin{proposition}\label{prop:Rn-candidates}
If $n\ge4$ and $R_n$ is categorifiable, then
\[
 n\in\{4,5,6,10\}.
\]
\end{proposition}

\begin{proof}
Let $t=2\cos(2\pi/n)$.  The representation $E_1$ exists for every
$n\ge3$, and \eqref{eq:two-fib-codegree} can be rewritten as
\begin{equation}\label{eq:codegree-mobius}
 f_{n,1}=n\frac{3-t}{2-t}.
\end{equation}
Conversely,
$t=(3n-2f_{n,1})/(n-f_{n,1})$, so
\[
 \Q(f_{n,1})=\Q(t)=\Q(\zeta_n+\zeta_n^{-1}).
\]
By \cref{prop:Rn-dimension-field}, this real cyclotomic field is contained
in $\KFib$.  For $n\ge3$, complex conjugation is the nontrivial automorphism of
$\Q(\zeta_n)$ over its fixed field $\Q(t)$: indeed, $\zeta_n$ is nonreal
and satisfies $X^2-tX+1=0$.  Consequently
$[\Q(t):\Q]=\phi(n)/2$, where $\phi$ is Euler's totient function, and hence
$\phi(n)\le4$.

We solve this inequality without a finite cutoff.  If a prime $p$ divides
$n$, then $p-1$ divides $\phi(n)$, so $p\in\{2,3,5\}$.  If $5\mid n$,
its exponent is one, no factor $3$ can occur, and the exponent of $2$ is at
most one; this gives $n=5,10$.  Otherwise $n=2^a3^b$ with $b\le1$;
if $b=0$ then $a\le3$, and if $b=1$ then $a\le2$.  Thus the values at least
four are
\[
 4,5,6,8,10,12.
\]
For these parameters the elements $t$ are, respectively,
\[
 0,\quad\frac{\sqrt5-1}{2},\quad1,\quad\sqrt2,
 \quad\frac{1+\sqrt5}{2},\quad\sqrt3.
\]
The fields $\Q(\sqrt2)$ and $\Q(\sqrt3)$ are not subfields of the
quadratic field $\KFib$: a quadratic field has no proper intermediate field,
and the three square classes $2,3,5$ in $\Q^\times/(\Q^\times)^2$ are
distinct.  This excludes $8$ and $12$.
\end{proof}

\subsection{The four remaining parameters}

\begin{proposition}\label{prop:exclude-4-5-6-10}
None of $R_4,R_5,R_6,R_{10}$ is categorifiable.
\end{proposition}

\begin{proof}
First suppose $5\mid n$ and take $j=n/5$.  Then
\begin{equation}\label{eq:five-codegree}
 f_{n,n/5}=\frac n{10}(15+\sqrt5).
\end{equation}
Let $\sigma$ be the nontrivial automorphism of $\KFib$, extended to
$\overline\Q$.  If this formal codegree were a $d$-number, invariance of its
principal ideal would make
\[
 \frac{f_{n,n/5}}{\sigma(f_{n,n/5})}
 =\frac{15+\sqrt5}{15-\sqrt5}
 =\frac{23+3\sqrt5}{22}
\]
an algebraic integer.  Its field trace is $23/11$, which is not an integer.
This contradicts \cref{thm:ostrik-package}(i), and excludes $n=5,10$.

For $n=4$, we have $f_{4,1}=6$, while
\[
 \frac{D_4^\pm}{6}=\frac{35\pm15\sqrt5}{4},
 \qquad
 \frac{D_4^{\mathrm{mix}}}{6}=\frac53.
\]
The first two numbers have nonintegral field trace $35/2$, and the last is
not an integer.  Thus none is an algebraic integer.

For $n=6$, we have $f_{6,1}=12$, while
\[
 \frac{D_6^\pm}{12}=30\pm\frac{40}{3}\sqrt5,
 \qquad
 \frac{D_6^{\mathrm{mix}}}{12}=\frac54.
\]
The first two numbers have nonintegral field norm $100/9$, and the last is
not an integer.  In both cases this contradicts the algebraic integrality of
$\Dim(\calC)/f_E$ in \cref{thm:ostrik-package}(ii), because
\cref{lem:possible-intrinsic-dimensions} exhausts the possible values of
$\Dim(\calC)$.
\end{proof}

\begin{theorem}[Categorifiable dihedral Fibonacci rings]\label{thm:arithmetic-classification}
For $n\ge2$, if $R_n$ is the Grothendieck ring of a fusion category, then
$n=2$ or $3$.  Both values occur: $R_2$ is realized by
$\Fib\boxt\Fib$, and $R_3$ by the chosen category $\TT$.
\end{theorem}

\begin{proof}
For $n\ge4$, \cref{prop:Rn-candidates} leaves only $4,5,6,10$, and
\cref{prop:exclude-4-5-6-10} excludes all four.  The two stated models
realize $n=2,3$; see \cite{Liu2015,MorrisonPenneys2015} for the rank-six
realization.
\end{proof}

\begin{theorem}[Two-generator based rings]
\label{thm:arithmetic-main}
Let $\calC$ be a semisimple rigid $\CC$-linear tensor category with simple
unit and finite-dimensional Hom spaces, generated by two nonisomorphic
simple objects $a,b$ satisfying
\[
 a^{\otimes2}\cong\mathbf1\oplus a,
 \qquad b^{\otimes2}\cong\mathbf1\oplus b.
\]
Either $K_0(\calC)\cong R_\infty$, or $K_0(\calC)\cong R_n$ for a unique
$n\ge2$ and $\calC$ is a fusion category.  In the latter case
categorifiability forces $n\in\{2,3\}$.  For a marked quotient of
$\Fib*\Fib$, the first alternative means that the quotient functor is a
tensor equivalence.
\end{theorem}

\begin{proof}[Proof of \cref{thm:arithmetic-main}]
The dichotomy and automatic finiteness are
\cref{prop:standard-form}.  In the finite branch,
\cref{thm:arithmetic-classification} gives $n\in\{2,3\}$.
The free product realizes $R_\infty$, while $\Fib\boxtimes\Fib$ and $\TT$
realize $R_2$ and $R_3$.
\end{proof}

\section{The finite two-Fibonacci quotients}
\label{sec:algebraic-uniqueness}

Fix the usual complex Fibonacci category $\Fib$.  Let $\rho,\mu$ be the
nonisomorphic simple images of two specified tensor embeddings of $\Fib$
into a fusion category $\calC$, and suppose that they tensor-generate
$\calC$.  Let $n$ be the length of their first alternating relation.
The standard form of \cref{prop:standard-form} and the arithmetic
obstruction of \cref{thm:arithmetic-classification} give $n=2$ or $3$.

\begin{theorem}[Ordered monoidal uniqueness]
\label{thm:marked-uniqueness}
If $n=2$, then $\calC\simeq\Fib\boxtimes\Fib$.  If $n=3$, then
$\calC\simeq\TTthree$.  In both cases the tensor equivalence intertwines
the ordered pair of source embeddings.
\end{theorem}

For $n=2,3$, the skein-theoretic uniqueness methods of Liu and of
Izumi--Morrison--Penneys admit an algebraic formulation without a unitarity
assumption on the target \cite[Theorem~5.3]{Liu2015}
\cite[Section~5]{IzumiMorrisonPenneys2013}.
We give the proof in Appendix~\ref{app:two-fibonacci-uniqueness}, using the latter
presentation with inverse intertwiners and the spherical trace pairing.
Automatic pseudo-unitarity supplies the canonical spherical structure.

An alternative algebraic proof of marked monoidal uniqueness for the two
surviving cases is given in
\cite[Proposition~\CoxeterRankTwoProposition]{FibCoxeter}, using regular
Fibonacci actions and module-naturality constraints. Unlike the
skein-theoretic argument in Appendix~\ref{app:two-fibonacci-uniqueness},
that method extends to global uniqueness for Coxeter--Hecke
categorifications. Neither argument assumes that the candidate target is
unitary. We retain the graphical proof as an independent argument and for
its presentation used in the ground-field constructions of
\cref{sec:characteristic-zero-fields}.

\begin{theorem}[Finite two-Fibonacci quotients]
\label{thm:two-fib-classification}
Let $Q:\Fib*\Fib\to\calC$ be a dominant tensor functor with fusion target.
If the two distinguished images are isomorphic, then $Q$ is equivalent, as
an ordered quotient, to the fold functor
$\nabla:\Fib*\Fib\to\Fib$ determined by two copies of
$\operatorname{id}_{\Fib}$.
Otherwise exactly one of the following holds:
\begin{enumerate}[label=\textup{(\roman*)}]
\item the first alternating relation is $ab\cong ba$ and
$\calC\simeq\Fib\boxtimes\Fib$;
\item the first alternating relation is $aba\cong bab$ and
$\calC\simeq\TT$.
\end{enumerate}
In all three cases the equivalence intertwines the ordered source embeddings.
\end{theorem}

\begin{proof}
Let $a,b$ be the two distinguished images.  If $a\cong b$, the first marked
embedding is a tensor equivalence $\Fib\simeq\calC$, since its image
tensor-generates $\calC$.  By \cref{lem:fib-marking-rigidity}, the second
marked embedding is uniquely monoidally isomorphic to the first.
Transporting through a quasi-inverse to the first embedding identifies both
restrictions of $Q$ with $\operatorname{id}_{\Fib}$; the free-product
universal property therefore identifies $Q$ with $\nabla$ as an ordered
quotient.  Otherwise \cref{prop:standard-form} gives a unique $n\ge2$ with
$K_0(\calC)\cong R_n$; \cref{thm:arithmetic-classification}
gives $n=2$ or $3$, and \cref{thm:marked-uniqueness} gives the stated
ordered equivalences.
\end{proof}

\section{Coxeter rigidity and transitive Fibonacci families}
\label{sec:many-fibonacci}

Let $\calB$ be a fusion category generated by pairwise nonisomorphic
marked Fibonacci objects $a_s$, $s\in S$, for a finite set $S$.
The specified source embeddings induce a dominant tensor functor
\begin{equation}\label{eq:many-fib-quotient}
 Q:\Fib^{*S}\longrightarrow\calB,
 \qquad \tau_s\longmapsto a_s.
\end{equation}
For distinct $s,t\in S$, \cref{thm:two-fib-classification} shows that either
$a_sa_t\cong a_ta_s$, or the ordered pair $(a_s,a_t)$ generates a copy of
$\TT$.  Let $\Gamma$ be the graph on $S$ whose edges are the noncommuting
pairs, and put
\[
 m_{st}=\begin{cases}
 2,&\{s,t\}\notin E(\Gamma),\\
 3,&\{s,t\}\in E(\Gamma).
 \end{cases}
\]
The associated Coxeter group is
\begin{equation}\label{eq:coxeter-group}
 W_\Gamma=\left\langle s\ (s\in S)\ \middle|\
 s^2=1,\ (st)^{m_{st}}=1\right\rangle .
\end{equation}
For the Coxeter system $(W_\Gamma,S)$, write
$\HFib(W_\Gamma)$ for the standard-basis normalization of its
Hecke ring with basis $h_w$, $w\in W_\Gamma$.  With $\ell$ denoting Coxeter
length, multiplication by a simple generator is given by
\begin{equation}\label{eq:normalized-hecke-multiplication}
 h_sh_w=
 \begin{cases}
  h_{sw},&\ell(sw)=\ell(w)+1,\\
  h_{sw}+h_w,&\ell(sw)=\ell(w)-1.
 \end{cases}
\end{equation}
Thus $h_s^2=1+h_s$.  This is the normalization of
$H_{\ph^2}(W_\Gamma)$ used in \cref{prop:standard-dihedral-model}; see
\cite[Chapter~4]{GeckPfeiffer2000}.

\begin{theorem}[Coxeter rigidity]\label{thm:coxeter-rigidity}
The group $W_\Gamma$ is finite, and there is a based-ring isomorphism
\begin{equation}\label{eq:coxeter-ring-identification}
 K_0(\calB)\cong\HFib(W_\Gamma),
 \qquad [a_s]\longmapsto h_s.
\end{equation}
In particular,
\[
 \rank K_0(\calB)=|W_\Gamma|,
 \qquad \FPdim(h_w)=\ph^{\ell(w)}.
\]
Consequently no proper quotient of $\HFib(W_\Gamma)$ is a fusion
ring in which the images of the $h_s$ remain pairwise distinct
Fibonacci basis elements, retain the graph $\Gamma$, and generate the
based ring.
\end{theorem}

\begin{proof}
Write $R=K_0(\calB)$, let $\mathcal I$ be its standard basis, and let
$\tau_0:R\to\mathbb Z$ be the coefficient of the tensor unit.  The inner
product
\[
 \langle x,y\rangle=\tau_0(x^*y)
\]
makes $\mathcal I$ an orthonormal basis.

We first record the two-generator subrings.  For $s\ne t$, the relations
give a positive homomorphism from the normalized dihedral Hecke ring of
type $I_2(m_{st})$ to $R$.  Its standard basis is
\[
 1,a_s,a_t,a_sa_t
\]
when $m_{st}=2$, and
\[
 1,a_s,a_t,a_sa_t,a_ta_s,a_sa_ta_s
\]
when $m_{st}=3$.  The coefficient of the unit vanishes on every displayed
nonunit element.  Indeed,
\[
 \tau_0(a_sa_t)=\tau_0(a_ta_s)=0
\]
because $a_s,a_t$ are distinct self-dual basis elements, and in the second
case
\[
 \tau_0(a_sa_ta_s)
 =\tau_0(a_s^2a_t)
 =\tau_0(a_t)+\tau_0(a_sa_t)=0.
\]
Here the first equality uses cyclicity of the unit trace.
Hence the pullback of $\tau_0$ is the standard unit trace, so the images
of the displayed standard basis elements are orthonormal. Write one such
image as $z=\sum_{X\in\mathcal I}n_XX$, where $n_X\in\mathbb Z_{\ge0}$.
Since $\mathcal I$ is orthonormal,
\[
 1=\langle z,z\rangle=\sum_{X\in\mathcal I}n_X^2.
\]
Exactly one coefficient is $1$ and all others are $0$, so $z$ itself
belongs to $\mathcal I$. Orthogonality between different images then
makes them distinct.  The two-generator based subring $R_{st}$ consequently has
rank $2m_{st}$ and is the standard rank-four or rank-six dihedral ring.

Let $N_s$ be the matrix of left multiplication by $a_s$ on $\mathcal I$.
Frobenius reciprocity makes $N_s$ symmetric, and
\[
 N_s^2=N_s+I.
\]
If the diagonal entry in the column indexed by $X\in\mathcal I$ is $d$,
then the $(X,X)$ entry of this equation gives
\begin{equation}\label{eq:fib-matrix-column}
 d^2+\sum_{Y\ne X}(N_s)_{Y,X}^2=d+1.
\end{equation}
Thus $d\in\{0,1\}$ and there is exactly one off-diagonal entry, equal to
one, in each column.  Symmetry pairs the basis elements, and the
off-diagonal entry of $N_s^2=N_s+I$ shows that the two diagonal entries in
each pair sum to one.  Hence every block of $N_s$ has the form
\begin{equation}\label{eq:fibonacci-matching-block}
 \begin{pmatrix}0&1\\1&1\end{pmatrix}.
\end{equation}
Let $P_s$ be the fixed-point-free involution of $\mathcal I$ given by the
off-diagonal entries, and let $D_s=N_s-P_s$.  The diagonal matrix $D_s$
marks the upper endpoint of every $P_s$-edge.

The union of the coloured $P_s$-edges is connected: its component
containing the unit is closed under multiplication by every $a_s$ and under
taking basis constituents, and the $a_s$ generate $R$ as a based ring.
Along a coloured edge $X\leftrightarrow Y$, either
\[
 Y=a_sX\qquad\text{or}\qquad Y=(a_s-1)X.
\]
Since $a_s(a_s-1)=(a_s-1)a_s=1$, every $X\in\mathcal I$ is a unit of the
underlying ring $R$.  This statement concerns ring multiplication; it does
not say that the corresponding simple object is tensor-invertible.

The common Frobenius--Perron eigenvector of the $N_s$ has ratio $\ph$
between the lower and upper endpoints of every coloured edge.  Starting at
the tensor unit therefore gives
\begin{equation}\label{eq:coxeter-height}
 \FPdim(X)=\ph^{h(X)},
 \qquad h(X)\in\mathbb Z_{\ge0},\qquad h(1)=0,
\end{equation}
and every coloured edge changes $h$ by one.  Nonnegativity of $h$ follows
from $\FPdim(X)\ge1$.

Fix $s\ne t$ and a basis element $X$.  Let $\mathcal O_{st}(X)$ be its
connected component under the two matchings $P_s,P_t$.  In the complexified
ring,
\begin{equation}\label{eq:two-colour-orbit}
 \operatorname{span}_{\mathbb C}\mathcal O_{st}(X)
 =(R_{st}\otimes_{\mathbb Z}\mathbb C)X.
\end{equation}
The inclusion from left to right follows from
$P_iY\in\{a_iY,(a_i-1)Y\}$; the reverse inclusion follows from
$a_iY\in\mathbb Z\{Y,P_iY\}$.  Right multiplication by $X$
is injective, so
\begin{equation}\label{eq:two-colour-size}
 |\mathcal O_{st}(X)|=\rank R_{st}=2m_{st}.
\end{equation}
Thus every two-colour component is a square or a hexagon.  The two
matchings cannot share an edge, since a component containing a shared edge
would have only two vertices.

On such a component,
\begin{equation}\label{eq:peak-trace}
 \Tr\bigl(N_sN_t\vert_{\operatorname{span}_{\mathbb C}
 \mathcal O_{st}(X)}\bigr)
 =\Tr\bigl(D_sD_t\vert_{\operatorname{span}_{\mathbb C}
 \mathcal O_{st}(X)}\bigr).
\end{equation}
Indeed, $P_s,P_t$ have no fixed points and no common edge, so all other
terms in the expansion of $(P_s+D_s)(P_t+D_t)$ have zero trace.  The
right-hand side counts the common upper endpoints, namely the local peaks
of the height function.  By \eqref{eq:two-colour-orbit}, right
multiplication by $X$ identifies the left $R_{st}$-module on the component
with the regular module of $R_{st}$.  The trace in
\eqref{eq:peak-trace} is consequently
\[
 \Tr(L_{a_s}L_{a_t}\vert_{R_{st}})=1.
\]
For $m_{st}=2$ this is the product of the traces of the two Fibonacci
matrices.  For $m_{st}=3$, put $c=a_sa_ta_s$; left multiplication by
$a_sa_t$ on the ordered basis
$1,a_s,a_t,a_sa_t,a_ta_s,c$ is
\[
\begin{array}{c|cccccc}
z&1&a_s&a_t&a_sa_t&a_ta_s&c\\ \hline
(a_sa_t)z&a_sa_t&c&a_s+a_sa_t&a_ta_s+c&1+a_s+c&
 a_t+a_ta_s+a_sa_t+c,
\end{array}
\]
whose trace is one. Each square or hexagon therefore has a unique peak.
To count valleys, traverse the cycle and record the height increments,
each of which is $+1$ or $-1$. A peak is a change from $+1$ to $-1$,
and a valley is a change from $-1$ to $+1$. These two kinds of sign
change alternate in a cyclic sequence, so their numbers are equal.
There is consequently a unique valley.
The two paths from the valley to the peak are monotone and have length
$m_{st}$; after subtracting the valley height, the cyclic height profiles are
\[
 (0,1,2,1)\qquad\text{and}\qquad(0,1,2,3,2,1),
\]
respectively.

The permutations $P_s$ now satisfy the Coxeter relations in
\eqref{eq:coxeter-group}, and hence define a transitive action of
$W_\Gamma$ on $\mathcal I$.  We show that the stabilizer of the tensor unit
is trivial.  Represent an element of this stabilizer by a closed coloured
path based at $1$, and choose an occurrence of its largest positive height.
The adjacent vertices both have height one less.  If the two adjacent
edges have the same colour, delete the immediate backtrack.  If their
colours are $s\ne t$, replace the two-edge segment through the unique peak
of its square or hexagon by the complementary path in that polygon.  The
replacement uses exactly the Coxeter relation for $s,t$, and all its
new interior vertices have smaller height.

This operation strictly decreases the number of occurrences of the current
maximum height and never introduces a larger height.  The lexicographically
ordered tuple of the numbers of vertices at each positive height therefore
decreases, so the procedure terminates.  A closed path based at height zero
with no positive vertex is constant.  The original word is consequently
the identity in $W_\Gamma$.  The action is free as well as transitive, and
thus
\begin{equation}\label{eq:coxeter-regular-action}
 \mathcal I\cong W_\Gamma,
 \qquad \rank R=|W_\Gamma|.
\end{equation}
Since $R$ is finite, so is $W_\Gamma$.

It remains to identify the multiplication, not only the rank.  Direct every
coloured edge from its upper endpoint to its lower endpoint.  This rewriting
system terminates by \eqref{eq:coxeter-height}.  If two different downward
edges leave the same vertex, the corresponding two-colour polygon has that
vertex as its unique peak, and the two descending paths around the polygon
meet at its unique valley.  Thus the system is locally confluent.  It has a
single terminal vertex: adjacent vertices have the same terminal vertex,
and the full coloured graph is connected; the terminal vertex is therefore
$1$.  Hence every $X$ has a descending path of length $h(X)$ to $1$, while
every path from $1$ to $X$ has length at least $h(X)$.  Under
\eqref{eq:coxeter-regular-action}, this says that $h(h_w)=\ell(w)$.
Reading the two cases of \eqref{eq:fibonacci-matching-block} now gives
exactly \eqref{eq:normalized-hecke-multiplication}.  This proves
\eqref{eq:coxeter-ring-identification}.  A proper fusion-ring quotient
retaining the stated generators and graph would satisfy the same argument
and have the same standard basis and multiplication, so no such quotient
exists.
\end{proof}

Theorem~\ref{thm:coxeter-rigidity} has a geometric proof: the
Frobenius--Perron height gives two-colour polygons with a unique peak and
valley, and local path elimination recovers the Coxeter action.
The companion paper \cite[Theorem~\CoxeterRigidityTheorem]{FibCoxeter}
proves a more general statement for rigid based rings, without a finite-rank
or categorifiability assumption, by commuting left and right matchings and
the exchange condition. Neither proof is used in place of the other.

Since $W_\Gamma$ is finite and all off-diagonal exponents are two or three,
the classification of finite simply-laced Coxeter systems implies that
$\Gamma$ is a disjoint union of finite ADE Dynkin diagrams
\cite[Section~6.9]{Davis2008}. The weaker forest consequence below is the
form needed for the transitive-symmetry argument, and we give its direct
proof.

\begin{corollary}\label{cor:coxeter-forest}
The graph $\Gamma$ is a forest.
\end{corollary}

\begin{proof}
Suppose first that $\Gamma$ is a cycle $C_n$, $n\ge3$.  On
\[
 V=\{(x_1,\ldots,x_n)\in\mathbb R^n:\textstyle\sum_i x_i=0\}
\]
let $r_i$, $i<n$, exchange the $i$th and $(i+1)$st coordinates, and set
\[
 r_n(x_1,x_2,\ldots,x_n)
   =(x_n+1,x_2,\ldots,x_{n-1},x_1-1).
\]
These affine transformations satisfy the Coxeter relations of $C_n$.  The
first $n-1$ generators contain the coordinate transposition
$p=(1\ n)$, and
\[
 r_np:x\longmapsto x+(1,0,\ldots,0,-1)
\]
has infinite order.  Thus $W_{C_n}$ is infinite, contrary to
\cref{thm:coxeter-rigidity}.  If $\Gamma$ contains a cycle, a shortest one
has no chord.  The fusion subcategory generated by its vertices has that
induced cycle as its noncommutation graph, and the same argument gives a
contradiction.
\end{proof}

\begin{remark}\label{rem:coxeter-path}
Transitivity cannot be omitted from the matching statement below.  For
$\Gamma=A_3$, one has $W_\Gamma\cong S_4$, and
$\HFib(W_\Gamma)$ is a rank-$24$ fusion ring generated by three
Fibonacci basis elements whose noncommutation graph is a path.  Its categorification is provided by the companion construction
\cite[Theorem~\CoxeterConstructionTheorem]{FibCoxeter}.
The path $A_3$ has no vertex-transitive automorphism group: its middle
vertex and endpoints have different degrees. Hence no family of tensor
autoequivalences can transitively permute these three marked generators,
so this example does not contradict the transitive matching theorem.
\end{remark}

\begin{corollary}[Transitive matching]\label{cor:transitive-matching}
Suppose a group $L$ acts on $\calB$ by tensor autoequivalences and
transitively permutes the family $(a_s)_{s\in S}$.  Then $\Gamma$ is empty
or a perfect matching.
\end{corollary}

\begin{proof}
If $S$ is empty, then $\calB=\Vec$ and the assertion is immediate.
Assume henceforth that $S$ is nonempty.
Tensor autoequivalences preserve the two pair types, so $\Gamma$ is a
finite vertex-transitive forest.  If it has $N$ vertices, $c$ connected
components, and common degree $d$, then
\[
 Nd=2|E(\Gamma)|=2(N-c),
 \qquad d=2-\frac{2c}{N}<2.
\]
Thus $d=0$ or $d=1$.
\end{proof}

\begin{corollary}\label{cor:neutral-no-invertibles}
The tensor unit is the only invertible simple object of $\calB$.
\end{corollary}

\begin{proof}
By \cref{thm:coxeter-rigidity}, the simple objects have Frobenius--Perron
dimensions $\ph^{\ell(w)}$, $w\in W_\Gamma$.  Only $w=1$ has dimension one.
\end{proof}

We next pass from the graph to a tensor decomposition.  Assume that
$\Gamma$ is a matching $M$, and let $S_0$ be its set of unmatched vertices.
For $s\in S_0$, write $\calC_s=\langle a_s\rangle\simeq\Fib$.  For
$e=\{s,t\}\in M$, write $\calC_e=\langle a_s,a_t\rangle\simeq\TT$, with
the ordering inherited from the two labelled source factors.  Put
$I=S_0\sqcup M$.

\begin{proposition}[Exact block factorization]
\label{prop:exact-block-factorization}
Choose an ordering $I=\{i_1,\ldots,i_r\}$ and set
\[
 \calA_j=\langle\calC_{i_1},\ldots,\calC_{i_j}\rangle,
 \qquad \calA_0=\Vec.
\]
For every $j$, multiplication induces an equivalence of finite semisimple
linear categories
\begin{equation}\label{eq:exact-block-factorization}
 m_j:\calA_{j-1}\boxtimes\calC_{i_j}\longrightarrow\calA_j,
 \qquad X\boxtimes Y\longmapsto X\otimes Y.
\end{equation}
Moreover, every object of $\calA_{j-1}$ commutes up to isomorphism with
every object of $\calC_{i_j}$.
\end{proposition}

\begin{proof}
The simple objects in a Fibonacci block are represented by $\mathbf1,a$,
and those in a $\TT$ block by
\[
 \mathbf1,a,b,ab,ba,aba.
\]
Every block simple is therefore represented by a word in its distinguished
Fibonacci generators.  Generators belonging to distinct blocks commute as
objects, so such words can be reordered.  It follows in particular that
objects belonging to distinct generated blocks commute up to isomorphism;
no coherent choice of these isomorphisms is used at this stage.

We prove the assertion by induction on $j$.  Assume that the simples of
$\calA_{j-1}$ are precisely the pairwise nonisomorphic products of simples
from the first $j-1$ blocks.  Every nontrivial fusion subcategory of $\Fib$
contains its distinguished Fibonacci object.  Every nontrivial fusion
subcategory of $\TT$ contains at least one of its two distinguished
Fibonacci objects: this follows immediately for $a,b$; for $ab$ or $ba$ it
follows after multiplying by the dual; and for $aba$ it follows from the
fusion table \eqref{eq:TT-fusion-table}.  Consequently, if
$\calA_{j-1}\cap\calC_{i_j}$ were nontrivial, a distinguished Fibonacci
object of the new block would belong to $\calA_{j-1}$.

The only simples of $\calA_{j-1}$ having Frobenius--Perron dimension $\ph$
are the distinguished Fibonacci objects in the preceding blocks: every
other nonunit product has dimension at least $\ph^2$.  This contradicts the
pairwise nonisomorphism of the labelled generators.  Hence
\begin{equation}\label{eq:block-intersection}
 \calA_{j-1}\cap\calC_{i_j}=\Vec.
\end{equation}

For simples $X,X'\in\calA_{j-1}$ and $Y,Y'\in\calC_{i_j}$, rigidity gives
\begin{equation}\label{eq:block-hom-computation}
 \Hom(XY,X'Y')\cong \Hom(X'^*X,Y'Y^*).
\end{equation}
The left argument on the right-hand side belongs to $\calA_{j-1}$ and the
right argument belongs to $\calC_{i_j}$.  By
\eqref{eq:block-intersection}, this space is one-dimensional when
$X\cong X'$ and $Y\cong Y'$, and is zero otherwise.  Thus the objects $XY$
are simple and pairwise nonisomorphic.  Their additive span is closed under
tensor products and duals because words from distinct blocks can be
reordered.  It contains all labelled generators, hence equals $\calA_j$.
Therefore $m_j$ is an equivalence of linear categories; equivalently, this
is the exact-factorization criterion of
\cite[Theorem~3.8]{Gelaki2017}.
\end{proof}

The preceding proposition supplies the Abelian equivalence, but its
multiplication functor does not yet carry a tensorator.  The required
coherence follows from the rigidity of twisted Deligne products.

\begin{lemma}[Tensor splitting of a block factorization]
\label{lem:coherent-block-multiplication}
Let $\calA$ and $\calD$ be fusion subcategories of a fusion category
$\calC$ over $\CC$.  Suppose multiplication
$\calA\boxtimes\calD\to\calC$ is an equivalence of linear categories and
$XY\cong YX$ for every $X\in\calA$, $Y\in\calD$.  If the universal grading
groups of $\calA$ and $\calD$ are trivial, there is a tensor equivalence
\[
 \Phi:\calA\boxtimes\calD\xrightarrow{\ \simeq\ }\calC
\]
whose restrictions to $\calA$ and $\calD$ are the given tensor-subcategory
embeddings.
\end{lemma}

\begin{proof}
The embedded subcategories, the Abelian multiplication equivalence, and the
object isomorphisms $XY\cong YX$ make $\calC$ a twisted Deligne product of
$\calA$ and $\calD$ in the sense of
\cite[Introduction]{ENO2026}.  By \cite[Theorem~3.15]{ENO2026}, a twisted
Deligne product of two fusion categories with trivial universal grading
groups is equivalent to their ordinary Deligne product, tautologically on
both embedded factors.  Taking the equivalence in the displayed direction
gives $\Phi$ and its tensorator.
\end{proof}

\begin{theorem}[Transitive several-Fibonacci block decomposition]
\label{thm:many-fib-blocks}
Suppose a group $L$ acts on $\calB$ by tensor autoequivalences and
transitively permutes the labelled Fibonacci objects $(a_s)_{s\in S}$.  The
noncommutation graph $\Gamma$ is empty or a perfect matching $M$, and
there is a tensor equivalence
\begin{equation}\label{eq:many-fib-factorization}
 \calB\simeq
 \begin{cases}
  \Fib^{\boxtimes S},&E(\Gamma)=\varnothing,\\[2mm]
  \TT^{\boxtimes M},&\Gamma\text{ is a perfect matching }M,
 \end{cases}
\end{equation}
intertwining all labelled source embeddings.
\end{theorem}

\begin{proof}
The graph statement is \cref{cor:transitive-matching}, and
\cref{prop:exact-block-factorization} gives the exact factorization by the
embedded blocks.  Every intermediate category $\calA_j$ has trivial
universal grading group.  Indeed, it is generated by Fibonacci objects
$x$, and $x\subset x^2$ forces the degree of $x$ in any grading to satisfy
$g=g^2$, hence $g=1$.  Apply
\cref{lem:coherent-block-multiplication} successively.  At each stage the
tensor equivalence restricts to the given embeddings of the two factors,
so the final equivalence retains every marked source factor.
\end{proof}

\section{Conjugates, grading, and crossed products}
\label{sec:neutral-reduction}

Let
\[
 F:\Fib*\Vec_G^\Omega\longrightarrow\calE
\]
be a finite composite and put $a=F(\tau)$.  Let
$q:G\twoheadrightarrow H$ be the epimorphism onto the group of isomorphism
classes in the pointed image.  Choose invertible representatives $u_h\in\calE$,
with $u_e=\mathbf1$, and choose $\bar u_h:=u_{h^{-1}}$ as a left dual of
$u_h$, with
\[
 \operatorname{coev}_h:\mathbf1\longrightarrow\bar u_h\otimes u_h,
 \qquad
 \operatorname{ev}_h:u_h\otimes\bar u_h\longrightarrow\mathbf1.
\]
Choose multiplication maps
\[
 m_{h,k}:u_h\otimes u_k\xrightarrow{\ \simeq\ }u_{hk}.
\]
Their associativity defect is a normalized cocycle $\omega$ determined by
\begin{align}\label{eq:representative-omega}
 m_{h,k\ell}(\operatorname{id}\otimes m_{k,\ell})
 \alpha_{u_h,u_k,u_\ell}
 =\omega(h,k,\ell)\,
 m_{hk,\ell}(m_{h,k}\otimes\operatorname{id}).
\end{align}
With the convention of \eqref{eq:eta-differential}, the pointed tensorator is
a cochain $\eta$ satisfying $d\eta=\Omega/(q^*\omega)$.
We use the parenthesized convention
\[
 \Ad_{u_h}(X)=(u_h\otimes X)\otimes\bar u_h.
\]

Set
\[
 K=\{h\in H:u_h\otimes a\otimes\bar u_h\cong a\},
 \qquad S=H/K,
\]
and let $\iota:\Fib\to\calE$ be the original marked embedding.  For every
$k\in K$, \cref{lem:fib-marking-rigidity} supplies a unique monoidal natural
isomorphism
\begin{equation}\label{eq:stabilizer-marking-cell}
 \lambda_k:\Ad_{u_k}\iota\Longrightarrow\iota.
\end{equation}
Thus the object stabilizer $K$ is also the stabilizer of the marked
embedding. For $u_e=\mathbf1$, take the canonical duality of the unit.
The unit constraints then give a monoidal natural isomorphism
$\Ad_{u_e}\iota\Rightarrow\iota$. By the same uniqueness, this is
$\lambda_e$; under these unit identifications, $\lambda_e=\id_\iota$.

Choose representatives $r_s\in H$ for $s\in S$, with $r_K=e$, and define
\[
 a_s=u_{r_s}\otimes a\otimes\bar u_{r_s},
 \qquad \iota_s=\Ad_{u_{r_s}}\iota,
 \qquad
 \calB=\langle a_s:s\in H/K\rangle_{\otimes,\oplus,\mathrm{summands}}.
\]
Distinct cosets give nonisomorphic objects.  Each $a_s$ carries the marked
embedding $\iota_s$.  For $h\in H$ and $s\in H/K$, put
\[
 k(h,s)=r_{hs}^{-1}hr_s\in K.
\]
Put $\rho_h=\Ad_{u_h}$, and let
$m_{h,k}^{\vee}:\bar u_{hk}\to\bar u_k\otimes\bar u_h$ be the left-dual
mate of $m_{h,k}$.  The compositor
$\mu_{h,k}:\rho_h\rho_k\Rightarrow\rho_{hk}$ is obtained, after
reassociation, by applying $m_{h,k}$ on the left and
$(m_{h,k}^{\vee})^{-1}$ on the right.  Define
\begin{equation}\label{eq:constructed-marking-transport}
\begin{aligned}
 \xi_{h,s}:\rho_h\iota_s
 &\xRightarrow{\ \mu_{h,r_s}\iota\ }
   \rho_{hr_s}\iota
  =\rho_{r_{hs}k(h,s)}\iota\\
 &\xRightarrow{\ (\mu_{r_{hs},k(h,s)}\iota)^{-1}\ }
   \rho_{r_{hs}}\rho_{k(h,s)}\iota
 \xRightarrow{\ \rho_{r_{hs}}\lambda_{k(h,s)}\ }
   \iota_{hs}.
\end{aligned}
\end{equation}
Both sides of the marked-action coherence equation are monoidal natural
isomorphisms between the same two marked Fibonacci embeddings.  Uniqueness
in \cref{lem:fib-marking-rigidity} therefore makes the two transformations
equal.  The same uniqueness shows that changing a coset representative
produces the uniquely equivalent marked family.

\begin{lemma}[Conjugate free-product quotient]
\label{lem:neutral-free-product}
There is a finite quotient
\[
 Q:\Fib^{*(H/K)}\longrightarrow\calB,
 \qquad \tau_s\longmapsto a_s.
\]
Conjugation preserves $\calB$ and acts transitively on the distinguished
objects.
\end{lemma}

\begin{proof}
The free-product property gives $Q$, and the definition of $\calB$ gives
dominance.  It is faithful by \cref{lem:automatic-faithfulness}.  The
multiplication maps and the associator give the natural isomorphisms
$\Ad_{u_g}\circ\Ad_{u_h}\cong\Ad_{u_{gh}}$ described by the compositors
$\mu_{g,h}$.  Thus conjugation induces the left action on $H/K$, which is
transitive.
\end{proof}

Conjugation by the $u_h$ gives tensor autoequivalences of $\calB$ and acts
transitively on the labelled family.  Hence
\cref{thm:many-fib-blocks} applies.  Writing $S=H/K$ and $\Gamma$ for
the noncommutation graph on $S$, we obtain
\begin{equation}\label{eq:block-identification}
 \calB\simeq
 \begin{cases}
  \Fib^{\boxtimes S},&E(\Gamma)=\varnothing,\\[2mm]
  \TT^{\boxtimes M},&\Gamma\text{ is a perfect matching }M.
 \end{cases}
\end{equation}
The equivalence retains all marked source embeddings.  In the second case,
the perfect matching $M$ is $H$-invariant, since conjugation preserves the
ordered two-generator type.

\begin{lemma}[Normal forms and separation]
\label{lem:normal-form-separation}
Every simple object of $\calE$ is isomorphic to $X\otimes u_h$ for a simple
$X\in\calB$ and $h\in H$.  If
\[
 X\otimes u_h\cong Y\otimes u_k
\]
with $X,Y\in\Irr(\calB)$, then $h=k$ and $X\cong Y$.
\end{lemma}

\begin{proof}
Moving every pointed letter to the right, using
$u_h\otimes X\cong\Ad_{u_h}(X)\otimes u_h$, gives the asserted normal form.
If $Xu_h\cong Yu_k$, put $g=hk^{-1}$.  Frobenius reciprocity gives
\[
 0\ne\Hom_\calE(Xu_g,Y)\cong\Hom_\calE(u_g,X^*Y).
\]
Thus the invertible object $u_g$ lies in $\calB$.  By
\cref{cor:neutral-no-invertibles}, $g=e$ and $X\cong Y$.
\end{proof}

\begin{proposition}[Faithful grading]\label{prop:B-equals-adjoint}
The decomposition
\[
 \calE=\bigoplus_{h\in H}\calB u_h
\]
is a faithful $H$-grading, its identity component is $\calB$, and
$\calB=\Ad(\calE)$.  Moreover,
\[
 \Inv(\calE)=\{[u_h]\mid h\in H\}\cong H.
\]
\end{proposition}

\begin{proof}
Normal-form separation makes the components disjoint and nonzero, and
conjugation gives $(\calB u_h)(\calB u_k)\subseteq\calB u_{hk}$.  Every
$a_s$ is a summand of $a_sa_s^*$, so $\calB\subseteq\Ad(\calE)$.  Conversely
$(Xu_h)(Xu_h)^*\in\calB$, hence $\Ad(\calE)\subseteq\calB$.  If $Xu_h$ is
invertible, then $X$ is invertible in $\calB$, hence $X=\mathbf1$.
\end{proof}

\begin{theorem}[Crossed-product presentation]
\label{thm:automatic-crossed-product}
The choices $(u_h,m_{h,k})$ determine a categorical conjugation action
$\rho$ on $\calB$ and a graded tensor equivalence
\[
 \Phi:\calB\rtw_\rho^\omega H\xrightarrow{\ \simeq\ }\calE,
 \qquad \Phi(X,h)=X\otimes u_h.
\]
Changing the representatives or multiplication maps changes the coherent
outer action by an equivalent one and therefore gives an equivalent crossed
product.
\end{theorem}

\begin{proof}
Put $\rho_h(X)=(u_h\otimes X)\otimes\bar u_h$.  Associativity, evaluation,
and coevaluation equip $\rho_h$ with tensor structure $J^h$, while
$m_{h,k}$ and the inverse of its left-dual mate give the compositors
$\mu_{h,k}:\rho_h\rho_k\Rightarrow\rho_{hk}$.  The scalar defect in
$m_{h,k}$ occurs once on each side of conjugation and cancels, so these
compositors satisfy the categorical-action pentagon.  Equation
\eqref{eq:representative-omega} is the scalar associativity defect, so the
crossed-product convention of Section~\ref{sec:preliminaries} applies.  The
canonical conjugation map
\[
 c_{h,Y}:u_h\otimes Y\longrightarrow\rho_h(Y)\otimes u_h
\]
is obtained by reassociating and inserting $\operatorname{coev}_h$.
The tensorator of $\Phi$ is
\[
 (X\otimes u_h)\otimes(Y\otimes u_k)
 \xrightarrow{\ \alpha,\ \operatorname{id}_X\otimes c_{h,Y}\otimes
 \operatorname{id}_{u_k}\ }
 (X\otimes\rho_hY)\otimes(u_h\otimes u_k)
 \xrightarrow{\ \operatorname{id}\otimes m_{h,k}\ }
 (X\otimes\Ad_{u_h}Y)\otimes u_{hk}.
\]
Its pentagon is \eqref{eq:representative-omega}.  Normal-form separation
shows that $\Phi$ is fully faithful and dominant.  Choice-independence at the
level of coherent outer actions is the biequivalence of
\cite[Theorem~4.1]{Galindo2011}.  Here $\Inv(\calB)=1$ by
\cref{cor:neutral-no-invertibles}, so no additional invertible object in the
neutral component enters this specialization.
Rescaling the chosen evaluation and coevaluation maps changes these formulas
by the monoidal cells and scalar cochains allowed in the action equivalence
of \cref{sec:classification}; it therefore does not change the classified
composite.
\end{proof}

\section{Equivariant perfect matchings}
\label{sec:group-actions}

For the conjugate family indexed by $H/K$, the Coxeter argument in
\cref{cor:transitive-matching} shows that the noncommutation graph is empty
or a perfect matching.  Tensor autoequivalences preserve commutation and the
ordered two-generator tensor subcategories, so the matching is
$H$-invariant.  It remains to describe such equivariant perfect matchings.

\begin{lemma}[Normalizer-quotient parametrization]
\label{lem:normalizer-involutions}
Fixed-point-free $H$-equivariant involutions of $H/K$ are in bijection with
nontrivial elements
\[
 \bar n\in N_H(K)/K,
 \qquad \bar n^2=1.
\]
\end{lemma}

\begin{proof}
An equivariant map $f:H/K\to H/K$ is determined by $f(K)=nK$, and
equivariance forces $f(hK)=hnK$.  This is well-defined exactly when
$n\in N_H(K)$.  Composition corresponds to the opposite multiplication
convention in $N_H(K)/K$, which does not change the set of involutions.  The
map squares to the identity exactly when $n^2\in K$.  If it fixes some
$hK$, then cancellation in $hnK=hK$ gives $nK=K$; conversely, $n\in K$
gives the identity.  Thus the nontrivial involutions are precisely the
fixed-point-free ones.
\end{proof}

It follows that exactly one of the following holds:
\begin{align}
 \calB_A&=\Fib^{\boxtimes H/K},\label{eq:type-A-B}\\
 \calB_B&=\TT^{\boxtimes(H/K)/\langle\bar n\rangle},
 \label{eq:type-B-B}
\end{align}
where the second case is parametrized by a nontrivial involution
$\bar n\in N_H(K)/K$.

The expression in \eqref{eq:type-B-B} is used as an $H/K$-marked category.
For each block $O=\{s,s\bar n\}$ its two tensor embeddings of $\Fib$ are
labelled by the two elements of $O$.  One may choose an ordering to
identify the block with the fixed ordered model
$(\TT,\iota_a,\iota_b)$; reversing the ordering is implemented by
precomposing with $\sigma$, and the fixed cell
$\epsilon:\sigma^2\Rightarrow\operatorname{id}_{\TT}$ identifies a double
reversal.  The ordering is therefore auxiliary and is not an additional
classification invariant.  An action used below is a full categorical
action together with coherent intertwiners of all labelled embeddings; it
is not specified merely by a permutation of the labels or by a list of
block reversals.

\section{Classification of finite Fibonacci--pointed composites}
\label{sec:classification}

Fix a finite group $G$, a normalized cocycle
$\Omega\in Z^3(G,\CC^\times)$, and the category $\Fib$ from
\cref{sec:preliminaries}.

\begin{definition}[Classification datum]
\label{def:classification-datum}
An admissible datum consists of the following.
\begin{enumerate}[label=\textup{(\roman*)},leftmargin=2.8em]
\item An epimorphism $q:G\twoheadrightarrow H$, a subgroup $K\le H$, a
normalized cocycle $\omega\in Z^3(H,\CC^\times)$, and a normalized
$2$-cochain $\eta\in C^2(G,\CC^\times)$ such that
\[
 d\eta=\frac{\Omega}{q^*\omega}.
\]
\item Put $S=H/K$.  An $S$-marked block category
$(\calB,\{\iota_s\}_{s\in S})$ of one of the following forms:
\begin{align*}
 \textup{(A)}\quad &\calB=\Fib^{\boxtimes H/K};\\
 \textup{(B)}\quad &\calB=\TT^{\boxtimes(H/K)/\langle\bar n\rangle},
\end{align*}
where each
\[
 \iota_s:\Fib\longrightarrow\calB
\]
is a specified tensor embedding.  In \textup{(A)} these are the factor
embeddings; write $a_s=\iota_s(\tau)$.  In \textup{(B)}, $\bar n$ is a
nontrivial element of order two
in $N_H(K)/K$, and the factor indexed by $\{s,s\bar n\}$ carries the two
marked embeddings $\iota_s,\iota_{s\bar n}$.  Choosing the opposite ordering
of that pair precomposes its identification with the fixed model $\TT$ by
$\sigma$; the cell $\epsilon:\sigma^2\Rightarrow\operatorname{id}_{\TT}$
makes this change involutive.  No additional ordering is retained beyond
the labelled embeddings.
\item A categorical action $(\rho,J,\mu)$ of $H$ on $\calB$: tensor
autoequivalences $\rho_h$ with specified tensor structures
\[
 J^h_{X,Y}:\rho_h(X)\otimes\rho_h(Y)\longrightarrow\rho_h(X\otimes Y),
 \qquad J^h_0:\mathbf1\longrightarrow\rho_h(\mathbf1),
\]
coherent
monoidal natural isomorphisms
$\mu_{h,k}:\rho_h\rho_k\Rightarrow\rho_{hk}$, and monoidal natural
isomorphisms of marked embeddings
\begin{equation}\label{eq:marked-action-cell}
 \xi_{h,s}:\rho_h\iota_s\Longrightarrow\iota_{hs}.
\end{equation}
They are normalized at $e\in H$ and satisfy
\begin{equation}\label{eq:marked-action-coherence}
 \xi_{hk,s}\circ(\mu_{h,k}\iota_s)
 =\xi_{h,ks}\circ(\rho_h\xi_{k,s})
 \qquad(h,k\in H,\ s\in S).
\end{equation}
Thus the datum records the action on the marked tensor functors, not only on
their generating objects.  In type \textup{(B)}, a reversal of a chosen
presentation of a block is encoded by the corresponding pair of cells
\eqref{eq:marked-action-cell} and by the fixed $(\sigma,\epsilon)$.
\end{enumerate}
\end{definition}

The tensorators $J^h$, the transformations $\mu_{h,k}$, the cells $\xi_{h,s}$,
and all their coherence are part of the datum; an action on labels alone is
insufficient.

\begin{lemma}[No hidden neutral invertible data]
\label{lem:no-hidden-neutral-data}
For either marked block category in \cref{def:classification-datum},
\[
 \Inv(\calB)=1,
 \qquad
 \operatorname{Aut}_{\otimes}(\operatorname{id}_{\calB})=1,
 \qquad
 \Inv(Z(\calB))=1.
\]
\end{lemma}

\begin{proof}
The first assertion is \cref{cor:neutral-no-invertibles}.  Let
$z$ be a monoidal natural automorphism of $\operatorname{id}_{\calB}$.
On the simple object $a_s=\iota_s(\tau)$ write
$z_{a_s}=c_s\operatorname{id}_{a_s}$.  The two summands of
$a_s^{\otimes2}\cong\mathbf1\oplus a_s$ give
$c_s^2=1$ and $c_s^2=c_s$, respectively, hence $c_s=1$.  The objects $a_s$
tensor-generate $\calB$, so naturality with the inclusions and projections of
simple summands gives $z=\operatorname{id}$.  Finally, an invertible object
of $Z(\calB)$ has invertible underlying object and hence underlying object
$\mathbf1$.  A half-braiding on $\mathbf1$ is the same as a monoidal natural
automorphism of $\operatorname{id}_{\calB}$, so it is trivial.
\end{proof}

These three triviality statements determine the form of graded equivalences
used below.  The first prevents a homogeneous unit from acquiring a
nontrivial invertible factor in the neutral component.  The second eliminates
monoidal automorphisms of the neutral identity functor, and the third
eliminates central invertible lifts.  Consequently the remaining homogeneous
coherence is scalar.

\begin{lemma}[Marked scalar normal form for crossed-product equivalences]
\label{lem:marked-galindo-specialization}
Let $(\calB,\rho,\omega)$ and $(\calB',\rho',\omega')$ be data of the
preceding form, and let $\theta:H\xrightarrow{\sim}H'$ be a group
isomorphism.  Up to graded monoidal natural isomorphism, a
$\theta$-graded tensor equivalence
\[
 P:\calB\rtw_\rho^\omega H
   \longrightarrow\calB'\rtw_{\rho'}^{\omega'}H'
\]
is determined by a tensor equivalence $T:\calB\to\calB'$, monoidal natural
isomorphisms
\[
 \gamma_h:T\rho_h\Longrightarrow\rho'_{\theta(h)}T,
\]
and a normalized cochain $\kappa\in C^2(H,\CC^\times)$ satisfying
\begin{align}
 \gamma_{hk}\circ T(\mu_{h,k})
 &=\mu'_{\theta(h),\theta(k)}T\circ
   (\rho'_{\theta(h)}\gamma_k)\circ(\gamma_h\rho_k),
 \label{eq:scalar-normal-form-action}\\
 d\kappa&=\frac{\omega}{\theta^*\omega'}.
 \label{eq:scalar-normal-form-cocycle}
\end{align}
Explicitly, $P(X,h)=(T(X),\theta(h))$, and, writing $T_2$ for the
tensorator of $T$, its tensorator on degrees $(h,k)$ is
\begin{equation}\label{eq:scalar-normal-form-tensorator}
 \kappa(h,k)\,
 T_2(X,\rho_hY)\circ
 \bigl(\operatorname{id}_{T(X)}\otimes\gamma_{h,Y}^{-1}\bigr).
\end{equation}
Conversely, these formulas define a $\theta$-graded tensor equivalence.
There are no additional neutral invertible or central parameters.
\end{lemma}

\begin{proof}
After regrading the target by $\theta$, the description of $1$- and $2$-cells
in \cite[Section~4 and Theorem~4.1]{Galindo2011} specializes as follows.
Restricting $P$ to the neutral
component gives $T$.  In Galindo's normal form the image of the homogeneous
unit $(\mathbf1,h)$ contains an invertible object of $\calB'$.  It is
$\mathbf1$ by \cref{lem:no-hidden-neutral-data}; choose this identification
for every $h$.  The mixed components of the tensorator of $P$ then give the
monoidal transformations $\gamma_h$, while its component on
$(\mathbf1,h)\otimes(\mathbf1,k)$ is a nonzero endomorphism of the simple
unit and hence a scalar $\kappa(h,k)$.

The monoidal-functor coherence for mixed triples is precisely
\eqref{eq:scalar-normal-form-action}.  When all three objects are
homogeneous units it is \eqref{eq:scalar-normal-form-cocycle},
with the differential convention of \eqref{eq:eta-differential}.  The
remaining tensorator components are forced by the mixed pentagons and give
\eqref{eq:scalar-normal-form-tensorator}.  Conversely, the monoidality of
$T$ and $\gamma$, together with
\eqref{eq:scalar-normal-form-action}--\eqref{eq:scalar-normal-form-cocycle},
verifies the tensorator pentagon directly.

A different choice of the identifications of the homogeneous units changes
$\kappa$ by the coboundary of a normalized scalar $1$-cochain; this is
exactly a $2$-cell in Galindo's description.  The last two assertions of
\cref{lem:no-hidden-neutral-data} exclude any further central lift or
neutral natural-automorphism parameter.
The marked cells $t_s$ and the source cochain $\lambda_g$ used below are
additional compatibility data for the retained source functor.  They are
verified directly by the monoidal-naturality equations below and are not
part of the unmarked crossed-product biequivalence itself.
\end{proof}

\begin{theorem}[Fibonacci--pointed classification]
\label{thm:fib-pointed-classification}
Equivalence classes of composites
\[
 F:\Fib*\Vec_G^\Omega\longrightarrow\calE
\]
are in bijection with equivalence classes of admissible classification data.
Here equivalence of data is the relation specified in
\cref{sec:equivalence-data}.
The target associated to a datum is
\[
 \calE=\calB\rtw_\rho^\omega H,
\]
and the restrictions of $F$ are determined by
\[
 F(\tau)=\iota_K(\tau),
 \qquad F(g)=(\mathbf1,q(g)),
\]
with tensorator $\eta$ on the pointed factor.
\end{theorem}

\begin{proof}
Let $(\calE,F)$ be a composite.  Its pointed restriction determines the
epimorphism $q:G\twoheadrightarrow H$, the cocycle $\omega$ on the pointed
image, and the tensorator cochain $\eta$ with
$d\eta=\Omega/(q^*\omega)$.  The stabilizer of $a=F(\tau)$ is $K$, and
\cref{lem:neutral-free-product} identifies the conjugates of $a$ with the
transitive $H$-set $H/K$.

 The conjugated source embeddings give the markings $\iota_s$, and
 \eqref{eq:constructed-marking-transport} gives the cells $\xi_{h,s}$;
 associativity and uniqueness give
\eqref{eq:marked-action-coherence}.  The $H$-action is transitive on the
marked family, so \cref{thm:coxeter-rigidity,cor:transitive-matching} show
that its noncommutation graph is empty or a perfect matching.  The exact
factorization and tensor-splitting parts of \cref{thm:many-fib-blocks}
identify $\calB$ with the corresponding marked block product.  In the empty
case $\calB$ is of type \textup{(A)}.  In the perfect case
\cref{lem:normalizer-involutions} supplies the involution $\bar n$ and
$\calB$ is of type \textup{(B)} as an $S$-marked category.  Conjugation
supplies the full categorical action $(\rho,J,\mu)$.  Finally,
\cref{thm:automatic-crossed-product} gives
\[
 \calE\simeq\calB\rtw_\rho^\omega H.
\]

Conversely, $q$ and $\eta$ define a dominant tensor functor
$\Vec_G^\Omega\to\Vec_H^\omega$.  The marking $\iota_K$ is the
distinguished tensor embedding $\Fib\to\calB$.  The free-product universal property
combines them into
\[
 F:\Fib*\Vec_G^\Omega\longrightarrow\calB\rtw_\rho^\omega H.
\]
Its image contains all homogeneous units $(\mathbf1,h)$, $h\in H$, and, by the cells $\xi_{h,K}$, every marked conjugate
$a_{hK}$; these generate the crossed product, so $F$ is dominant; it is faithful by
\cref{lem:automatic-faithfulness}.  The two constructions are inverse up
to the equivalence of data described below and
\cref{lem:marked-galindo-specialization}.
\end{proof}

\subsection{Equivalence of classification data}\label{sec:equivalence-data}

Use the multiplicative inhomogeneous cochain differential.  Let
\[
 \mathfrak D=(H,K,q,\omega,\eta,\calB,\iota,\rho,J,\mu,\xi),
 \qquad
 \mathfrak D'=(H',K',q',\omega',\eta',\calB',\iota',\rho',J',\mu',\xi')
\]
be admissible data.  An equivalence $\mathfrak D\simeq\mathfrak D'$ consists
of:
\begin{enumerate}[label=\textup{(\roman*)},leftmargin=2.8em]
\item a group isomorphism $\theta:H\xrightarrow\sim H'$ with
$q'=\theta q$;
\item a $\theta$-equivariant bijection
$\beta:H/K\xrightarrow\sim H'/K'$ with $\beta(K)=K'$ which carries the
matching to the matching; equivalently, in type \textup{(B)},
$\beta(s\bar n)=\beta(s)\bar n'$;
\item a tensor equivalence $T:\calB\to\calB'$, monoidal natural
isomorphisms of markings
\[
 t_s:T\iota_s\Longrightarrow\iota'_{\beta(s)},
\]
and monoidal natural isomorphisms
\[
 \gamma_h:T\rho_h\Longrightarrow\rho'_{\theta(h)}T
\]
satisfying, with standard whiskering understood,
\begin{equation}\label{eq:action-equivalence-coherence}
 \gamma_{hk}\circ T(\mu_{h,k})
 =\mu'_{\theta(h),\theta(k)}T\circ
   (\rho'_{\theta(h)}\gamma_k)\circ(\gamma_h\rho_k);
\end{equation}
and the marked compatibility equations
\begin{equation}\label{eq:marked-equivalence-coherence}
 t_{hs}\circ T(\xi_{h,s})
 =\xi'_{\theta(h),\beta(s)}\circ
   (\rho'_{\theta(h)}t_s)\circ(\gamma_h\iota_s).
\end{equation}
In type \textup{(B)}, this equation is imposed on both marked embeddings of
every $\TT$ factor.  Thus a functor which reverses a chosen presentation of
a factor is recorded by the corresponding $t_s,t_{s\bar n}$ and the fixed
swap $\sigma$, without a separate choice of ordering;
\item a normalized $2$-cochain $\kappa\in C^2(H,\CC^\times)$ and a
normalized $1$-cochain $\lambda\in C^1(G,\CC^\times)$ such that
\begin{equation}\label{eq:explicit-cochain-equivalence}
 d\kappa=\frac{\omega}{\theta^*\omega'},
 \qquad
 \eta\,(q^*\kappa)=\eta'\,d\lambda.
\end{equation}
\end{enumerate}
The tensorator of the induced equivalence on homogeneous components is
multiplied by $\kappa(h,k)$.  Its pentagon is the first equation in
\eqref{eq:explicit-cochain-equivalence}.  On pointed simples, the tensorator
of the composite of the source functor with the induced equivalence is
\[
 \eta(g,h)\kappa(q(g),q(h)),
\]
whereas the tensorator of the second source functor is $\eta'(g,h)$.
Consequently the monoidal-naturality equation for components
$\lambda_g\operatorname{id}_{q'(g)}$ is
\[
 \lambda_{gh}\eta(g,h)\kappa(q(g),q(h))
 =\eta'(g,h)\lambda_g\lambda_h,
\]
which is exactly the second equation in
\eqref{eq:explicit-cochain-equivalence}.  Equations
\eqref{eq:action-equivalence-coherence} and
\eqref{eq:marked-equivalence-coherence}, together with
\cref{lem:marked-galindo-specialization}, give a tensor equivalence of the
crossed products and the component $t_K$ on the $\Fib$ source factor.  Hence
equivalent data produce equivalent composites.

Conversely, let an equivalence of composites be given.  The identities
$\calB=\Ad(\calE)$ and
$\Inv(\calE)=\{[u_h]\mid h\in H\}$ from
\cref{prop:B-equals-adjoint} recover the neutral component and induce the
group isomorphism $\theta$, its restriction $T$, the equality
$q'=\theta q$, and the equivariant coset map $\beta$.  Commutation of the
marked generators recovers the matching, so $\beta$ preserves the type
\textup{(B)} involution.  After regrading by $\theta$,
\cref{lem:marked-galindo-specialization} supplies the action intertwiners
$\gamma_h$ and the scalar cochain $\kappa$, with precisely
\eqref{eq:action-equivalence-coherence} and the first equation of
\eqref{eq:explicit-cochain-equivalence}.

The monoidal natural isomorphism of source functors restricts on the
$\Fib$ factor to $t_K$.  Since $\iota_s$ is the marked conjugate of
$\iota_K$, the tensorator of the target equivalence transports $t_K$ to
cells $t_s$ for all $s\in H/K$.  The monoidal-functor pentagon makes this
transport independent of a representative of $s$ and gives
\eqref{eq:marked-equivalence-coherence}.  On the pointed simples the same
source isomorphism gives scalars $\lambda_g$, and its tensorator equation is
the second equation of \eqref{eq:explicit-cochain-equivalence}.

Changing the representatives $u_h$, their multiplication maps, or the coset
representatives changes only $(T,t,\gamma,\kappa,\lambda)$ by the $2$-cells
described in \cref{lem:marked-galindo-specialization}.  Choosing the opposite
presentation of a $\TT$ factor precomposes its marked identification with
$\sigma$; $\epsilon$ identifies a double change, and the two cells $t_s$ in
\eqref{eq:marked-equivalence-coherence} record the change.  Thus no
ordering of a paired block or choice of representatives gives an additional class, and an
equivalence of composites recovers precisely the stated equivalence class
of data.

\section{Schematic examples and constraints}\label{sec:examples}

The examples below record only the group-theoretic shapes forced by the
classification.  They are schematic: a concrete composite still requires
the cocycle, tensorator cochain, and compatible marked categorical action
from \cref{def:classification-datum}.

Suppose first that $H$ is abelian.  Then $N_H(K)/K=H/K$, so type
\textup{(B)} of \cref{def:classification-datum} is determined by a
nonidentity element of order two
\[
 t\in H/K,\qquad t^2=e.
\]
The two possible identity components are
\[
 \Fib^{\boxtimes H/K}
 \quad\text{and}\quad
 \TT^{\boxtimes(H/K)/\langle t\rangle}.
\]

Let $G=C_N$.  Its pointed image is $H=C_r$ for some $r\mid N$, and a
transitive orbit $H/K$ is cyclic of order $d$ for some $d\mid r$.  The
second block category occurs exactly when $d$ is even.  After choosing the
categorical action and compatible cocycle, the crossed products have the
forms
\[
 \Fib^{\boxtimes d}\rtw C_r
 \quad\text{and}\quad
 \TT^{\boxtimes d/2}\rtw C_r.
\]
The element of order two pairs opposite labels, agreeing with the cyclic
action described in \cite[Remark~4.9]{EdieMichell2022}.

For $H=C_2$, the possible orbit sizes are one and two.  The identity
component is therefore one of $\Fib$, $\Fib\boxt\Fib$, and $\TT$.  For
$H=C_3$, a nontrivial transitive orbit has no equivariant fixed-point-free
involution, so an orbit of size three gives $\Fib^{\boxtimes3}$.  For a
nonabelian group, the possible pairings of $H/K$ are parametrized by the
order-two elements of $N_H(K)/K$, as in
\cref{lem:normalizer-involutions}.

\section{Forms over fields of characteristic zero}
\label{sec:characteristic-zero-fields}

The preceding classification is complex-linear.  This section determines
the corresponding quotient problems over an arbitrary field of
characteristic zero.  The two-Fibonacci problem has no additional forms
once its ordered source functor is retained.  The Fibonacci--pointed
classification has the same structural form, although its scalar
cohomology and categorical actions must be taken over the ground field.

\subsection{The rank-two input and scalar extension}

Let $k$ be a field of characteristic zero.
Here $\mathcal F_{\tau,k}$ denotes the model with the explicitly chosen
scalar parameter $\tau\in k$; this local scalar notation is distinct from
the object $\tau$ of the fixed complex category $\Fib$. Both embeddings of
$\KFib$ are retained in this section.  A finite semisimple tensor
category over $k$ will mean a $k$-linear rigid semisimple tensor category
with finite-dimensional Hom spaces, finitely many isomorphism classes of
simple objects, and $\End(\mathbf1)=k$.  It is \emph{split} if
\[
 \End(X)=k
 \qquad\text{for every simple object }X.
\]
No splitness assumption is included in the definition.  If $K/k$ is a
field extension, put
\begin{equation}\label{eq:scalar-extension-category}
 \calC_K=\Kar(\calC\otimes_kK).
\end{equation}
For objects already defined over $k$, scalar extension gives
\begin{equation}\label{eq:hom-scalar-extension}
 \Hom_{\calC_K}(X_K,Y_K)
 \cong \Hom_\calC(X,Y)\otimes_kK.
\end{equation}

\begin{definition}[Rank-two model and marked composites over $k$]
\label{def:fibonacci-type-k}
A Fibonacci-type category over $k$ is a finite semisimple rigid $k$-linear
tensor category with simple objects $\mathbf1,x$ and fusion rule
$x^{\otimes2}\cong\mathbf1\oplus x$.
Fix $\tau\in k$ with $\tau^2=\tau+1$.  The category
$\mathcal F_{\tau,k}$ is the split skeletal $k$-linear tensor category with
simple objects $\mathbf1,x$, fusion rule
\[
 x^{\otimes2}=\mathbf1\oplus x,
\]
normalized unit constraints, and fixed multiplicity-one fusion bases.  All
$1\times1$ associator blocks in these bases are $1$, and its only
nonscalar block is
\begin{equation}\label{eq:fibonacci-f-matrix-definition}
 F^{xxx}_x=
 \begin{pmatrix}
  b&b\\
  1&-b
 \end{pmatrix},
 \qquad b=\tau^{-1}=\tau-1.
\end{equation}
The pentagon equations reduce to $b^2+b=1$.

A $\tau$-marked Fibonacci object of a $k$-linear tensor category $\calC$ is
a tensor embedding $\iota:\mathcal F_{\tau,k}\to\calC$, not merely the
isomorphism class of $\iota(x)$.  A finite ordered two-Fibonacci composite
is a faithful dominant tensor functor
\[
 Q:\mathcal F_{\tau,k}*\mathcal F_{\tau,k}\longrightarrow\calC,
\]
and a finite Fibonacci--pointed composite is a faithful dominant tensor
functor
\[
 F:\mathcal F_{\tau,k}*\Vec_{G,k}^{\Omega}\longrightarrow\calE.
\]
In both cases equivalence retains the entire source functor.
\end{definition}

The input category is necessarily split.  Indeed, rigidity gives
$x^*\cong x$, and hence
\[
 \End(x)\cong\Hom(\mathbf1,x^*\otimes x)
 \cong\Hom(\mathbf1,\mathbf1\oplus x)=k.
\]

\begin{proposition}[Field of definition of the rank-two category]
\label{prop:fibonacci-field-of-definition}
A Fibonacci-type category exists over $k$ if and only if
$t^2-t-1$ has a root in $k$, equivalently if and only if $\sqrt5\in k$.
After a root
\[
 \tau\in k,\qquad \tau^2=\tau+1,
\]
is chosen, there is a unique split form with this parameter, up to tensor
equivalence.  Denote it by $\mathcal F_{\tau,k}$.
\end{proposition}

\begin{proof}
The equivalence between the two field conditions follows from
$\sqrt5=2\tau-1$ and $\tau=(1+\sqrt5)/2$.  For existence use the skeletal
presentation in \eqref{eq:fibonacci-f-matrix-definition}; its pentagons use
only $b^2+b=1$, so it is defined over $\Q(\tau)\subset k$.  After adjoining
$\sqrt b$, this gauge is equivalent
to the usual symmetric gauge; the latter square root is therefore not a
field-of-definition condition.

Conversely, after extension to an algebraic closure, the Moore--Seiberg
rank-two classification, as recalled by Ostrik, identifies the category
with the Yang--Lee category or its Galois conjugate
\cite[Introduction and Main Theorem]{Ostrik2003}.  The M\"uger squared norm of the
nontrivial simple object is
\begin{equation}\label{eq:fibonacci-squared-norm}
 |x|^2=\tau^2=1+\tau=\frac{3\mathbin\pm\sqrt5}{2},
 \qquad (2|x|^2-3)^2=5.
\end{equation}
Thus every field of definition contains $\sqrt5$.  Etingof--Gelaki prove
this necessity and the existence of a split form over $\Q(\sqrt5)$ in
\cite[Section~4.2.5]{EtingofGelaki2012}.
These field-of-definition results are theirs.  To the authors' knowledge,
the scalar-extension argument for the ordered two-generator quotient below
does not occur in the cited literature.  Their result also
shows that fixing the embedding
$\Q(\sqrt5)\hookrightarrow k$, or equivalently the chosen root $\tau$,
leaves no further split form.
\end{proof}

\begin{lemma}[Rigidity of a ground-field Fibonacci marking]
\label{lem:k-fib-marking-rigidity}
Let $\iota_1,\iota_2:\mathcal F_{\tau,k}\to\calC$ be tensor embeddings.  If
$\iota_1(x)\cong\iota_2(x)$, then there is a unique monoidal natural
isomorphism $\iota_1\Rightarrow\iota_2$.
\end{lemma}

\begin{proof}
After scalar extension to an algebraic closure $\bar k$, the two embeddings
have the same rank-two image.  The algebraically closed version of
\cref{lem:fib-marking-rigidity}, equivalently the triviality result of
\cite[Section~4.2.5]{EtingofGelaki2012}, gives a unique monoidal cell between
them.  Every Galois conjugate of this cell has the same source and target,
so uniqueness makes it invariant.  Finite-dimensional Hom-space descent
therefore gives a $k$-linear natural isomorphism.  Its monoidality identities
descend because scalar extension is faithful, and uniqueness descends for
the same reason.
\end{proof}

For the ground-field argument, use the presentation of the same
quadratic field $\KFib=\Q(\sqrt5)$ as
\[
 \KFib=\Q(\vartheta),
 \qquad \vartheta^2=\vartheta+1.
\]
A choice of $\tau$ gives an embedding
$\KFib\hookrightarrow k$ by $\vartheta\mapsto\tau$.

We will use the following standard spreading-out observation.  It is
included to avoid assuming that the field $k$ itself embeds in $\CC$.

\begin{lemma}[Spreading out finite tensor data]
\label{lem:spreading-out-tensor-data}
Let $\calC$ be a finite semisimple rigid tensor category over $k$, equipped
with finitely many tensor functors and natural transformations.  All these
data have a model over a finitely generated subfield $k_0\subset k$ such
that scalar extension from $k_0$ to $k$ recovers the original data.  If the
data contain $\mathcal F_{\tau,k}$, then $k_0$ may be chosen to contain the
image of $\KFib$.
\end{lemma}

\begin{proof}
Choose representatives of the finitely many simple objects, bases of all
Hom spaces, tensor-product decomposition maps, associators, unit
constraints, evaluation and coevaluation maps, and the matrices describing
the additional functors and transformations.  Include the inverses of all
specified isomorphisms, the idempotents witnessing dominance, and
separability idempotents for the finite-dimensional endomorphism algebras.
Only finitely many elements of $k$ occur among their coefficients.  The
subfield generated by these elements and by $\KFib$ is finitely
generated over $\Q$.  The defining identities are polynomial identities in
the chosen matrices, while the displayed inverse and separability data
preserve rigidity and semisimplicity.  They therefore define the required
$k_0$-model, whose scalar extension is the original datum.
\end{proof}

Every finitely generated extension of $\KFib$ admits an embedding
into $\CC$ extending either prescribed embedding of $\KFib$.
Consequently a finite datum over $k$ can be compared with the complex
classification after first applying
\cref{lem:spreading-out-tensor-data}.

\subsection{Two Fibonacci factors over \texorpdfstring{$k$}{k}}

Let
\[
 Q:\mathcal F_{\tau,k}*\mathcal F_{\tau,k}\longrightarrow\calC
\]
be a finite ordered composite, and let $a,b$ be the two distinguished
images.  If $a\cong b$, set $n=1$.  Otherwise let $u_m,v_m$ be the two
alternating words of length $m$ beginning with $a,b$, and let $n$ be the
least integer for which $u_n\cong v_n$.

\begin{theorem}[Automatic splitness]
\label{thm:automatic-splitness-over-k}
Every finite ordered composite of two copies of
$\mathcal F_{\tau,k}$ is split.
\end{theorem}

\begin{proof}
Apply \cref{lem:spreading-out-tensor-data} to obtain a model
$Q_0:\mathcal F_{\tau_0,k_0}*\mathcal F_{\tau_0,k_0}\to\calC_0$, where
  $\KFib\subset k_0\subset k$.  Choose an embedding
$\iota:k_0\hookrightarrow\CC$ for which $\iota(\tau_0)=\ph$.  The scalar
extension $Q_\iota$ is a finite quotient of the two specified complex
Fibonacci categories.  If the two generating images coincide, put
$\mathcal W_1=\{\mathbf1,a\}$.  Otherwise
\cref{prop:standard-form} supplies a least $n\ge2$ and the list
\begin{equation}\label{eq:alternating-word-list-k}
 \mathcal W_n=
 \{\mathbf1,u_1,\ldots,u_{n-1},v_1,\ldots,v_{n-1},u_n\},
 \qquad u_n\cong v_n.
\end{equation}
Every member $W_\CC$ of the applicable list is simple, and the list exhausts
the complex simple objects.

For each such word, \eqref{eq:hom-scalar-extension} gives
\[
 \End_{\calC_0}(W)\otimes_{k_0}\CC
 \cong\End_{(\calC_0)_\CC}(W_\CC)\cong\CC.
\]
Thus $\End_{\calC_0}(W)=k_0$, so $W$ is already a split simple object over
$k_0$.  If $X$ is any simple object of $\calC_0$, choose a simple
constituent of $X_\CC$.  It is isomorphic to the scalar extension of some
$W$ in the applicable list, and therefore
\[
 \Hom_{\calC_0}(X,W)\otimes_{k_0}\CC\ne0.
\]
Hence $\Hom_{\calC_0}(X,W)\ne0$, and simplicity gives $X\cong W$.  The
listed words therefore exhaust the simple objects of $\calC_0$ and prove
that $\calC_0$ is split.  They remain simple and exhaustive after extension
from $k_0$ to $k$, proving that $\calC$ is split.
\end{proof}

This argument uses the entire list of word representatives in
\eqref{eq:alternating-word-list-k}.  For a general composite of split tensor
categories, geometric simples can instead occur in nontrivial Galois orbits
or acquire division endomorphism algebras, so splitness of the factors alone
does not imply splitness of the composite.

\begin{proposition}[Invariance of the shortest relation]
\label{prop:shortest-relation-base-change}
Let $K/k$ be a field extension.  Whenever $u_m$ and $v_m$ are simple,
\[
 u_m\cong v_m
 \quad\Longleftrightarrow\quad
 (u_m)_K\cong(v_m)_K.
\]
Consequently the least alternating-relation parameter is unchanged by
scalar extension.
\end{proposition}

\begin{proof}
Only descent of an isomorphism requires proof.  If the scalar extensions
are isomorphic, then
\[
 \Hom_\calC(u_m,v_m)\otimes_kK
 \cong\Hom_{\calC_K}((u_m)_K,(v_m)_K)\ne0.
\]
Thus $\Hom_\calC(u_m,v_m)\ne0$.  The two objects are simple by hypothesis, so
a nonzero morphism between them is an isomorphism.
\end{proof}

We next construct the three standard quotients over the minimal field.  The
first is $\mathcal F_{\vartheta,\KFib}$, and the second is its
Deligne square.  For the third, use the generators and relations of
Appendix~\ref{app:two-fibonacci-uniqueness} over $\KFib$.  The Fibonacci circle
and resolution coefficients are $\vartheta$ and $\vartheta^{-1}$, and the
inverse-coupon calculation gives
\[
 \omega_U=\sigma_U=1.
\]
Hence every inverse, rotation, roof, and train relation has coefficients in
$\KFib$.  Let $\calP_{3,\KFib}$ be the additive spherical
category defined by this raw diagram presentation, before quotienting by
negligibles.  For every two boundary words $X,Y$ and every field extension
$L/\KFib$, extension of coefficients gives
\begin{equation}\label{eq:raw-presentation-base-change}
 \Hom_{\calP_{3,\KFib}}(X,Y)\otimes_{\KFib}L
 \cong
 \Hom_{\calP_{3,L}}(X_L,Y_L).
\end{equation}
Indeed, both sides are the free vector space on the same diagrams modulo
the scalar extensions of the same local relations.

Write
\[
 V=\Hom_{\calP_{3,\KFib}}(X,Y),\qquad
 W=\Hom_{\calP_{3,\KFib}}(Y,X)
\]
and let $\beta:W\times V\to \KFib$ be the spherical pairing
$\beta(w,v)=\operatorname{tr}(wv)$.  The component of the negligible ideal
is
\[
 \mathcal N(X,Y)=\operatorname{rad}_V(\beta)
 =\ker\bigl(V\longrightarrow W^*\bigr).
\]
The inclusion
$\mathcal N(X,Y)\otimes L\subseteq\mathcal N_L(X_L,Y_L)$ is immediate.
Conversely, write an element of the extended radical as
$z=\sum_i v_i\otimes\lambda_i$ with the finitely many $\lambda_i$ linearly
independent over $\KFib$.  For every $w\in W$,
\[
 0=\beta_L(w\otimes1,z)=\sum_i\beta(w,v_i)\lambda_i,
\]
so every $v_i$ already lies in $\operatorname{rad}_V(\beta)$.  Thus
\begin{equation}\label{eq:negligible-base-change}
 \mathcal N(X,Y)\otimes_{\KFib}L
 =\mathcal N_L(X_L,Y_L),
\end{equation}
and consequently
\begin{equation}\label{eq:quotient-hom-base-change}
 \Hom_{\calP_{3,\KFib}/\mathcal N}(X,Y)
 \otimes_{\KFib}L
 \cong
 \Hom_{\calP_{3,L}/\mathcal N_L}(X_L,Y_L).
\end{equation}

Taking the additive Karoubi completion now defines
\[
 \mathcal{TT}_{3,\vartheta,\KFib}.
\]
Under $\vartheta\mapsto\ph$, \eqref{eq:quotient-hom-base-change} and the
complex uniqueness theorem identify its scalar extension with $\TT$.
For the six words
\[
 \mathcal W_3=\{\mathbf1,a,b,ab,ba,aba\},
\]
the complex Hom matrix is diagonal with scalar diagonal entries.  Faithful
base change in \eqref{eq:quotient-hom-base-change} therefore gives
\[
 \Hom(W,W')=
 \begin{cases}
  \KFib,&W=W',\\
  0,&W\ne W',
 \end{cases}
 \qquad(W,W'\in\mathcal W_3).
\]
The Fibonacci splittings and $aba\cong bab$ reduce every word to a direct
sum of these six.  Hence the endomorphism algebra of every such sum is a
product of matrix algebras over $\KFib$, and every idempotent splits
as a sum of the six displayed objects.  The Karoubi completion introduces
no further simple type.  Rigidity is inherited from the spherical
presentation, so $\mathcal{TT}_{3,\vartheta,\KFib}$ is a split
fusion category.  Put
\begin{equation}\label{eq:TT-form-over-k}
 \mathcal{TT}_{3,\tau,k}
 =\mathcal{TT}_{3,\vartheta,\KFib}
   \otimes_{\KFib,\,\vartheta\mapsto\tau}k.
\end{equation}

The uniqueness of these models over a general field concerns the ordered
quotient datum, not the target category with its two generators forgotten.
For a finite ordered family of tensor functors
\[
 \mathbf q=(q_i:\mathcal S_i\longrightarrow\calD)_{i=1}^r,
\]
write $\underline{\operatorname{Aut}}(\calD;\mathbf q)$ for the categorical
group whose objects are tensor autoequivalences $T$ equipped with monoidal
cells $\alpha_i:Tq_i\Rightarrow q_i$.  An arrow
$(T,\alpha_i)\to(T',\alpha_i')$ is a monoidal natural isomorphism
$d:T\Rightarrow T'$ satisfying
$\alpha_i=\alpha_i'\circ(dq_i)$ for every $i$.

\begin{lemma}[Factorwise ordered free-product data]
\label{lem:factorwise-free-product-data}
Let $q:\mathcal S_1*\mathcal S_2\to\calD$ be a tensor functor and let
$q_i=q|_{\mathcal S_i}$.  Restriction to the two factors identifies the
categorical groups
\[
 \underline{\operatorname{Aut}}(\calD,q)
 \simeq
 \underline{\operatorname{Aut}}(\calD;q_1,q_2).
\]
It also identifies forms and Galois descent data for the full ordered
functor with those for the ordered pair $(q_1,q_2)$, including their
$1$-cells and $2$-cells.
\end{lemma}

\begin{proof}
The free-product universal property says that a tensor functor out of
$\mathcal S_1*\mathcal S_2$ is the same datum as its two restrictions.  A monoidal
natural transformation is likewise uniquely determined by, and uniquely
extends from, its restrictions to the two factors.  Apply these two
statements after every scalar twist and to every functor and cell occurring
in an automorphism, a form, or a descent datum.  The extensions preserve
vertical and horizontal composition by uniqueness, so the identification
also respects modifications and all Galois coherence equations.
\end{proof}

\begin{lemma}[Rigidity of the standard ordered quotients]
\label{lem:ordered-quotient-automorphisms}
After extension to an algebraically closed field of characteristic zero,
the categorical group
$\underline{\operatorname{Aut}}(\calD;q_1,q_2)$ of the two ordered factor
functors in each of the three standard quotients is contractible.
\end{lemma}

\begin{proof}
For the rank-two quotient, both factor functors are the identity of
$\mathcal F_{\tau,k}$.  Given $(T,\alpha_1,\alpha_2)$, the first cell gives
a monoidal isomorphism
$\beta:T\Rightarrow\operatorname{id}$.  The difference between
$\beta$ and $\alpha_2$ is a monoidal automorphism of
$\operatorname{id}_{\mathcal F_{\tau,k}}$, which is trivial by
\cite[Section~4.2.5]{EtingofGelaki2012}.  Thus $\beta=\alpha_2$ as well,
and uniqueness follows.

For the Deligne square write its ordered factors as $\calA$ and $\calD$, with
nontrivial simple objects $a$ and $b$, and let
$c_{Y,X}:Y\otimes X\to X\otimes Y$ be the standard interchange.  Denote the
tensorator of $T$ by $J$.  The two cells $\alpha_1,\alpha_2$ transport $c$ to
the coherent interchange
\begin{equation}\label{eq:transported-interchange}
 c^T_{Y,X}=(\alpha_{1,X}\otimes\alpha_{2,Y})J_{X,Y}^{-1}
 T(c_{Y,X})J_{Y,X}(\alpha_{2,Y}^{-1}\otimes\alpha_{1,X}^{-1}).
\end{equation}
The object $ab$ is simple, so $c^T_{b,a}=r c_{b,a}$ for some
$r\in k^\times$.  The two hexagon identities and naturality for the
summands $\mathbf1$ and $a$ of $a^2$ give, respectively,
\[
 r^2=1,
 \qquad r^2=r.
\]
Thus $r=1$, and additivity and naturality give $c^T_{Y,X}=c_{Y,X}$ for all
$X\in\calA$ and $Y\in\calD$.

For $X\in\calA$ and $Y\in\calD$, define
\begin{equation}\label{eq:deligne-square-beta}
 \beta_{XY}=(\alpha_{1,X}\otimes\alpha_{2,Y})J_{X,Y}^{-1}:T(XY)\longrightarrow XY.
\end{equation}
The linear Deligne-product property extends these maps uniquely to a natural
isomorphism $\beta:T\Rightarrow\operatorname{id}$.  Its tensor coherence on
$XY\otimes X'Y'$ is the tensor-functor pentagon for $J$, the monoidality of
$\alpha_1,\alpha_2$, and \eqref{eq:transported-interchange} with $c^T=c$.
Hence $\beta$ is monoidal and restricts to $\alpha_i$ on each factor.

For $\mathcal{TT}_3$, transport the tensor structure through
$\alpha_1,\alpha_2$, so that $T$ fixes both Fibonacci calculi.  The remaining
generating Hom space is one-dimensional; hence $T(U)=cU$ and
$T(V)=c^{-1}V$ for some nonzero scalar $c$.  Preservation of the normalized
rotation relation $\mathcal R_-(U)=V$ first gives $c^2=1$.  The residual
value $c=-1$ changes both local coefficients $x,y$ in
\cref{lem:local-jellyfish-scalars} to their negatives, whereas the fixed
presentation has $y=\sigma_U=1$ by \eqref{eq:n3-omega-sigma-one}.  Hence
$c=1$.  The presentation and closed-diagram evaluation in
Appendix~\ref{app:two-fibonacci-uniqueness} then identify $T$ with the identity on the
whole target, compatibly with both factor cells.

In every case the two factor images tensor-generate the target.  Hence the
displayed identification is unique, and an automorphism of the identity
object that is the identity on the two factors is the identity on every
object and morphism.  This proves contractibility.
\end{proof}

\begin{lemma}[Descent of geometric ordered rigidity]
\label{lem:descent-geometric-ordered-rigidity}
Let $L$ be a field of characteristic zero, let $\calD$ and the finite ordered
family $\mathbf q$ be defined over $L$, and let $\overline L$ be an algebraic
closure.  If
$\underline{\operatorname{Aut}}(\calD_{\overline L};\mathbf q_{\overline L})$
is contractible, then
$\underline{\operatorname{Aut}}(\calD;\mathbf q)$ is contractible.
\end{lemma}

\begin{proof}
For an object $(T,\alpha_i)$ of the latter categorical group, geometric
contractibility gives a unique compatible monoidal cell
$\beta_{\overline L}:T_{\overline L}\Rightarrow\operatorname{id}$.  Every
$L$-automorphism of $\overline L$ fixes its source, target, and compatibility
conditions, and therefore fixes $\beta_{\overline L}$ by uniqueness.  Its
finitely many matrix entries consequently lie in $L$.  They define a cell
$\beta$ over $L$; naturality, monoidality, and compatibility with all
$\alpha_i$ descend because scalar extension is faithful.  The same argument
applied to arrows proves uniqueness and excludes nontrivial automorphisms of
the identity object.
\end{proof}

\begin{lemma}[Effective descent for a finite ordered family]
\label{lem:effective-ordered-descent}
Let $L/K$ be a finite Galois extension with group $\Gamma$.  Let
$\mathcal S_1,\ldots,\mathcal S_r$ be finite semisimple rigid tensor categories over
$K$, let $\calD$ be a finite semisimple rigid tensor category over $L$, and
let $q_i:(\mathcal S_i)_L\to\calD$ be tensor functors.  An ordered descent datum
consists of
\begin{align*}
 \Phi_\sigma &: {}^\sigma\calD\longrightarrow\calD,\\
 \alpha_{\sigma,i} &: \Phi_\sigma{}^\sigma q_i\Longrightarrow q_i,\\
 J_{\sigma,\tau} &: \Phi_\sigma{}^\sigma\Phi_\tau
   \Longrightarrow\Phi_{\sigma\tau},
\end{align*}
where the $\Phi_\sigma$ are tensor equivalences, all displayed cells are
monoidal, the datum is normalized at the identity, and
\begin{align}
 \alpha_{\sigma\tau,i}\circ
 (J_{\sigma,\tau}\,{}^{\sigma\tau}q_i)
 &=\alpha_{\sigma,i}\circ
   (\Phi_\sigma{}^\sigma\alpha_{\tau,i}),
 \label{eq:ordered-descent-source-coherence}\\
 J_{\sigma\tau,\upsilon}\circ
 (J_{\sigma,\tau}\,{}^{\sigma\tau}\Phi_\upsilon)
 &=J_{\sigma,\tau\upsilon}\circ
 (\Phi_\sigma{}^\sigma J_{\tau,\upsilon}).
 \label{eq:ordered-descent-pentagon}
\end{align}

A $1$-cell from this datum to a primed one is
$(R,t_i,c_\sigma)$, where $R:\calD\to\calD$ is a tensor equivalence,
$t_i:Rq_i\Rightarrow q_i$ and
$c_\sigma:R\Phi_\sigma\Rightarrow\Phi'_\sigma{}^\sigma R$ are monoidal
cells, satisfying
\begin{align}
 t_i\circ R(\alpha_{\sigma,i})
 &=\alpha'_{\sigma,i}\circ(\Phi'_\sigma{}^\sigma t_i)
   \circ(c_\sigma{}^\sigma q_i),
 \label{eq:ordered-descent-equivalence-source}\\
 c_{\sigma\tau}\circ R(J_{\sigma,\tau})
 &=(J'_{\sigma,\tau}{}^{\sigma\tau}R)
   \circ(\Phi'_\sigma{}^\sigma c_\tau)
   \circ(c_\sigma{}^\sigma\Phi_\tau).
 \label{eq:ordered-descent-equivalence-compositor}
\end{align}
A $2$-cell $(R,t_i,c_\sigma)\Rightarrow(R',t_i',c_\sigma')$ is a monoidal
natural isomorphism $d:R\Rightarrow R'$ such that
\begin{equation}\label{eq:ordered-descent-modification}
 t_i=t_i'\circ(dq_i),
 \qquad
 c_\sigma'\circ(d\Phi_\sigma)
  =(\Phi'_\sigma{}^\sigma d)\circ c_\sigma.
\end{equation}
Composition is the usual whiskered vertical composition.  The homotopy
$1$-groupoid, obtained by identifying $1$-cells joined by a $2$-cell, is
equivalent to the groupoid of $K$-forms of
$(\calD;q_1,\ldots,q_r)$.  If
$\underline{\operatorname{Aut}}(\calD;\mathbf q)$ is contractible, this
groupoid is empty or contractible.
\end{lemma}

\begin{proof}
Choose simple skeleta and finite bases for the Hom spaces of $\calD$ and the
$\mathcal S_i$.  Galois descent for finite-dimensional vector spaces descends the
composition, tensor product, associator, unit and duality maps, together with
each $q_i$ and its tensorator.  Equations
\eqref{eq:ordered-descent-source-coherence} and
\eqref{eq:ordered-descent-pentagon} are precisely their compatibility
conditions.  The descended finite-dimensional endomorphism algebras remain
semisimple because they are separable in characteristic zero.  Applying the
same Hom-space descent to $R,t_i,c_\sigma$ and to $d$ proves full faithfulness
on $1$- and $2$-cells and gives the inverse construction by scalar extension.
This is the finite ordered-functor specialization of
\cite[Sections~3.1--3.5]{EtingofGelaki2012}.

Fix one descent datum.  The comparison with any other takes values in
$\underline{\operatorname{Aut}}(\calD;\mathbf q)$, and its coherence ambiguity
is an automorphism of the identity object of that categorical group.  If the
categorical group is contractible, the comparison and its coherence are
unique up to a unique $2$-cell.  Hence the homotopy groupoid is contractible
whenever it is nonempty.
\end{proof}

\begin{proposition}[Descent rigidity for the ordered quotient datum]
\label{prop:ordered-quotient-descent}
Let $K$ be a field of characteristic zero and let $\overline K$ be an
algebraic closure.  If one of the three standard ordered quotients over
$\overline K$ has a $K$-form, that form is unique up to an equivalence
intertwining the entire ordered source functor.
\end{proposition}

\begin{proof}
Restrict the ordered quotient to its two rank-two factors.  Two proposed
$K$-forms and a geometric comparison between them are defined over some
finite Galois extension $L/K$.  By
\cref{lem:ordered-quotient-automorphisms,lem:descent-geometric-ordered-rigidity},
the ordered factor family over $L$
has contractible automorphism categorical group.  The two descent data are
therefore uniquely equivalent by \cref{lem:effective-ordered-descent}.
This equivalence descends the ordered factor functors to $K$, and
\cref{lem:factorwise-free-product-data} extends it uniquely to the full
free-product source functor.
\end{proof}

\begin{theorem}[Two-Fibonacci classification over $k$]
\label{thm:char-zero-two-fib}
Let $k$ be a field of characteristic zero, choose
$\tau\in k$ with $\tau^2=\tau+1$, and fix
$\mathcal F_{\tau,k}$.  Up to equivalence of ordered quotient tensor
categories, the finite quotients
\[
 \mathcal F_{\tau,k}*\mathcal F_{\tau,k}\longrightarrow\calC
\]
are exactly
\begin{equation}\label{eq:char-zero-two-fib-list}
 \begin{array}{ccl}
  n=1&:&\mathcal F_{\tau,k},\\
  n=2&:&\mathcal F_{\tau,k}\boxtimes_k\mathcal F_{\tau,k},\\
  n=3&:&\mathcal{TT}_{3,\tau,k}.
 \end{array}
\end{equation}
Every target in this list is split, and no finite quotient exists for
$n\ge4$.
\end{theorem}

\begin{proof}
The target is split by \cref{thm:automatic-splitness-over-k}, and
\cref{prop:shortest-relation-base-change} preserves its shortest relation.
Spread the quotient out over a finitely generated field
$\KFib\subset k_0\subset k$ and embed $k_0$ in $\CC$ with
$\vartheta\mapsto\ph$.  The complex classification gives exactly
$n=1,2,3$ and excludes $n\ge4$.

Let $\overline{k_0}$ be an algebraic closure of $k_0$.  The quotient over
$\overline{k_0}$ and the corresponding standard model become equivalent
after extension to $\CC$.  Such an equivalence is a finite system of
matrices satisfying polynomial equations over $\overline{k_0}$, and hence
already exists over $\overline{k_0}$.  The two quotients are therefore
$k_0$-forms of the same ordered quotient datum.  Their equivalence descends
by \cref{prop:ordered-quotient-descent}.  Extending it from $k_0$ to $k$
gives one of the three ordered equivalences in
\eqref{eq:char-zero-two-fib-list}.  The standard forms exist by the
construction preceding \cref{lem:ordered-quotient-automorphisms}.
\end{proof}

The two embeddings $\KFib\hookrightarrow\CC$ send $\vartheta$ to
$\ph$ and $1-\ph$.  They carry the formulas above to the Fibonacci and
Yang--Lee Galois-conjugate realizations, respectively.  The theorem fixes
the chosen rank-two input and makes no tensor-equivalence identification
between these two complex specializations.

\subsection{Fibonacci--pointed composites over \texorpdfstring{$k$}{k}}

Let $G$ be a finite group and let $\Omega\in Z^3(G,k^\times)$.  Write
$\Vec_{G,k}^{\Omega}$ for the split pointed $k$-linear category with
associator $\Omega$.

\begin{definition}[$k$-admissible marked datum]
\label{def:k-admissible-datum}
A $k$-admissible datum consists of the following $k$-linear data.
\begin{enumerate}[label=\textup{(\roman*)},leftmargin=2.8em]
\item An epimorphism $q:G\twoheadrightarrow H$, a subgroup $K\le H$, a
normalized $\omega\in Z^3(H,k^\times)$, and a normalized
$\eta\in C^2(G,k^\times)$ satisfying
\begin{equation}\label{eq:k-admissible-cochains}
 d\eta=\frac{\Omega}{q^*\omega}.
\end{equation}
\item With $S=H/K$, an $S$-marked block category whose labelled embeddings
$\iota_s:\mathcal F_{\tau,k}\to\calB$ identify it with exactly one of
\begin{equation}\label{eq:k-neutral-blocks}
 \mathcal F_{\tau,k}^{\boxtimes H/K}
 \quad\text{or}\quad
 \mathcal{TT}_{3,\tau,k}^{\boxtimes(H/K)/\langle\bar n\rangle},
\end{equation}
where in the second case $\bar n$ is a nonidentity element of order two in
$N_H(K)/K$ and each factor retains both ordered marked embeddings.
\item A categorical action $(\rho,J,\mu)$ of $H$ on $\calB$, together with
monoidal marked cells
\[
 \xi_{h,s}:\rho_h\iota_s\Longrightarrow\iota_{hs},
\]
normalized at the identity and satisfying the action pentagon and
\eqref{eq:marked-action-coherence} over $k$.
\end{enumerate}

An equivalence of two such data is a group isomorphism $\theta:H\to H'$ with
$q'=\theta q$, a $\theta$-equivariant bijection
$\beta:H/K\to H'/K'$ satisfying $\beta(K)=K'$ and preserving the matching,
a $k$-linear tensor
equivalence $T:\calB\to\calB'$, and monoidal cells
\[
 t_s:T\iota_s\Rightarrow\iota'_{\beta(s)},
 \qquad
 \gamma_h:T\rho_h\Rightarrow\rho'_{\theta(h)}T,
\]
satisfying \eqref{eq:action-equivalence-coherence} and
\eqref{eq:marked-equivalence-coherence}, together with normalized cochains
\[
 \kappa\in C^2(H,k^\times),
 \qquad \lambda\in C^1(G,k^\times)
\]
satisfying \eqref{eq:explicit-cochain-equivalence}.  Thus the whole marked
source functor, not only its generating objects, is retained.
\end{definition}

We first establish the block decomposition over the ground field.  In the
statement below, the local two-generator categories are equipped with their
ordered source embeddings.

\begin{proposition}[Several-Fibonacci blocks over $k$]
\label{prop:char-zero-many-fib}
Let $\calB$ be a finite semisimple rigid $k$-linear tensor category generated
by pairwise nonisomorphic marked Fibonacci objects $a_s$, $s\in S$.
Suppose a group $L$ acts by $k$-linear tensor autoequivalences and
transitively permutes this family.  Then the noncommutation graph
$\Gamma$ is empty or a perfect matching $M$, the category $\calB$ is split,
and there is a $k$-linear tensor equivalence
\begin{equation}\label{eq:char-zero-many-fib-factorization}
 \calB\simeq
 \begin{cases}
  \mathcal F_{\tau,k}^{\boxtimes_k S},&E(\Gamma)=\varnothing,\\[2mm]
  \mathcal{TT}_{3,\tau,k}^{\boxtimes_k M},&\Gamma\text{ is a perfect matching }M,
 \end{cases}
\end{equation}
that intertwines every labelled source embedding.
\end{proposition}

\begin{proof}
Extend to an algebraic closure $\bar k$.  The objects $(a_s)_{\bar k}$
remain simple and pairwise nonisomorphic by
\eqref{eq:hom-scalar-extension}, and they generate $\calB_{\bar k}$ after
idempotent completion.  The action of $L$ also extends and remains
transitive.  By \cref{thm:char-zero-two-fib}, every pair generates the
rank-four commuting block or the rank-six $\mathcal{TT}_{3,\tau,\bar k}$
block.  The based-ring proof of
\cref{thm:coxeter-rigidity,cor:transitive-matching}, which is independent of
the algebraically closed characteristic-zero ground field, applies to
$K_0(\calB_{\bar k})$.  Hence its noncommutation graph is empty or a perfect
matching.  The pair type is unchanged by scalar extension, so this is also
the graph over $k$.  The local ordered quotient theorem supplies
object-level commutation between distinct blocks over $k$.

We next repeat the exact-factorization induction over $k$, taking some care
not to assume that the ambient category is already split.  Suppose the
subcategory $\calA$ generated by the preceding blocks has the asserted split
simple-word basis, and let $\calD$ be the new split Fibonacci or
$\mathcal{TT}_{3,\tau,k}$ block.  Extend the three categories to an algebraic
closure $\bar k$.  Characteristic zero makes finite-dimensional semisimple
algebras separable, so scalar extension preserves semisimplicity; the
specified simple objects of $\calA$ and $\calD$ remain simple by
\eqref{eq:hom-scalar-extension}.  If
$\calA_{\bar k}\cap\calD_{\bar k}$ were nontrivial, the argument in
\cref{prop:exact-block-factorization} would put a distinguished Fibonacci
generator of $\calD_{\bar k}$ among the simple words of
$\calA_{\bar k}$ having Frobenius--Perron dimension $\ph$.  It would
therefore be isomorphic to an earlier labelled generator.  The corresponding
Hom space is the scalar extension of a Hom space over $k$, so this
isomorphism would already exist over $k$, a contradiction.  Hence the
intersection is trivial after scalar extension and over $k$.

For simple $X,X'\in\calA$ and $Y,Y'\in\calD$, the rigidity calculation
\eqref{eq:block-hom-computation} now gives
\[
 \Hom_{\calB}(XY,X'Y')=
 \begin{cases}
  k,&X\cong X'\text{ and }Y\cong Y',\\
  0,&\text{otherwise}.
 \end{cases}
\]
Thus the products $XY$ are split simple objects and are pairwise
nonisomorphic.  Their additive span is closed under tensor products and
duals by the word commutation relations, and it contains the generators.
It therefore exhausts the next generated subcategory.  Induction proves
that $\calB$ is split and that multiplication gives an exact factorization
by the embedded blocks as a semisimple $k$-linear category.

It remains to construct the tensorator over $k$.  Extend one stage
$\calA\calD$ of the exact factorization to $\bar k$.  The two factors have
trivial universal grading groups: each is generated by Fibonacci objects
$x$ satisfying $x\subset x^2$, so the degree of every such generator is
trivial.  Hence \cite[Theorem~3.15]{ENO2026} supplies a coherent
interchange
\[
 c_{Y,X}:Y\otimes X\longrightarrow X\otimes Y
 \qquad(X\in\calA_{\bar k},\ Y\in\calD_{\bar k})
\]
and thereby a tensor structure on the multiplication equivalence that
retains the two embedded factors.

This interchange is unique.  Indeed, if $c'$ is another one, then for
simple $X,Y$ exact factorization makes $c'_{Y,X}$ a scalar
$r(X,Y)\in\bar k^\times$ times $c_{Y,X}$.  Naturality and the two hexagon
identities imply
\[
 r(Z,Y)=r(X,Y)r(X',Y)
 \quad\text{whenever }Z\subset X\otimes X',
\]
and the analogous identity in the second variable.  Thus $r$ factors
through a bicharacter of the two universal grading groups
\cite[Lemma~2.1]{ENO2026}; both groups are trivial, so $r=1$.

Every Galois conjugate of $c$ is another coherent interchange retaining the
same two embedded factors, and hence equals $c$.  The coordinates of $c$ in
the one-dimensional $k$-spaces
$\Hom_{\calB}(YX,XY)$ consequently lie in $k$.  Naturality, the hexagon
identities, and the unit identities descend faithfully.  The resulting
$k$-linear tensorator makes multiplication a tensor equivalence retaining
both factors.  Applying this argument successively to the exact
factorization proves \eqref{eq:char-zero-many-fib-factorization} with all
labels preserved.
\end{proof}

\begin{theorem}[Fibonacci--pointed classification over $k$]
\label{thm:char-zero-pointed}
Let $k$ be a field of characteristic zero containing $\sqrt5$, choose
$\tau\in k$ with $\tau^2=\tau+1$, and fix
$\mathcal F_{\tau,k}$.  Let $G$ be a finite group and let
$\Omega\in Z^3(G,k^\times)$.  Equivalence classes of finite composites
\[
 F:\mathcal F_{\tau,k}*\Vec_{G,k}^{\Omega}\longrightarrow\calE
\]
are in bijection with equivalence classes of $k$-admissible classification
data.  Every target $\calE$ is split, and the datum reconstructs it as
\[
 \calE\simeq\calB\rtw_\rho^\omega H,
\]
where $\calB$ is one of the two block categories in
\eqref{eq:k-neutral-blocks}.
\end{theorem}

\begin{proof}
The pointed restriction has image $\Vec_{H,k}^\omega$ for an epimorphism
$q:G\twoheadrightarrow H$.  An invertible simple object has endomorphism
algebra $k$, and each multiplication space between chosen invertible
representatives is one-dimensional.  Their associativity defect and the
pointed tensorator therefore give exactly the $k^\times$-valued data in
\eqref{eq:k-admissible-cochains}.

Let $x$ be the nontrivial simple object of $\mathcal F_{\tau,k}$.
Conjugating $a=F(x)$ by these invertible representatives produces the
labelled family indexed by $H/K$, where
\[
 K=\{h\in H:\Ad_{u_h}(a)\cong a\}.
\]
For every $k_0\in K$, \cref{lem:k-fib-marking-rigidity} gives the unique
$k$-linear monoidal cell
\[
 \lambda_{k_0}:\Ad_{u_{k_0}}\iota\Longrightarrow\iota.
\]
Choose representatives $r_s$ and put
$k_0(h,s)=r_{hs}^{-1}hr_s$.  The same composite as in
\eqref{eq:constructed-marking-transport},
\[
 \rho_h\Ad_{u_{r_s}}\iota
 \xRightarrow{\mu_{h,r_s}}
 \rho_{hr_s}\iota
 \xRightarrow{\mu_{r_{hs},k_0(h,s)}^{-1}}
 \rho_{r_{hs}}\rho_{k_0(h,s)}\iota
 \xRightarrow{\rho_{r_{hs}}\lambda_{k_0(h,s)}}
 \Ad_{u_{r_{hs}}}\iota,
\]
defines $k$-linear cells $\xi_{h,s}$.  Their coherence and independence of
the representatives follow from the uniqueness in
\cref{lem:k-fib-marking-rigidity}.  Thus the conjugation action of $H$ is
transitive on the marked embeddings, and every two labelled generators form a finite ordered
two-Fibonacci quotient.  Thus \cref{thm:char-zero-two-fib} supplies the
commute-or-$\mathcal{TT}_{3,\tau,k}$ alternative, and
\cref{prop:char-zero-many-fib} identifies the
conjugate-generated category, with every labelled embedding retained, with
the corresponding split block product in \eqref{eq:k-neutral-blocks}.

The normal-form proof of \cref{lem:normal-form-separation} now gives every
simple object of $\calE$ uniquely as $X\otimes u_h$ with $X$ simple in
$\calB$.  Consequently
\[
 \End_\calE(X\otimes u_h)\cong\End_\calB(X)=k,
\]
so $\calE$ is split.  The grading and crossed-product construction in
\cref{prop:B-equals-adjoint,thm:automatic-crossed-product} are algebraic and
give $\calE\simeq\calB\rtw_\rho^\omega H$ over $k$.  The proof of
\cref{lem:marked-galindo-specialization} likewise uses only the absence of
neutral and central invertible data and the fact that
$\End(\mathbf1)=k$; its scalar parameters therefore lie in $k^\times$.
This proves both the classification and its stated equivalence relation.
\end{proof}

The theorem is an intrinsic classification over $k$.  It does not identify
its set of equivalence classes with the complex one.  In particular, the
maps induced by scalar extension on group cohomology need not be injective,
and distinct $k$-linear categorical actions can become equivalent after
extension.  The quotient, cochains, and action appearing in the theorem
retain precisely this ground-field information.

\section{Conclusion and future directions}\label{sec:conclusion}

The two-Fibonacci classification supplies the local monoidal input for the
several-generator Coxeter system.  Finiteness makes its noncommutation graph
a forest, and conjugation by the pointed image makes the graph transitive and
hence a matching.  Exact factorization and tensor splitting identify the
resulting blocks before the normal-form argument produces a faithful grading
and a crossed-product description of every finite Fibonacci--pointed
composite.
The classification retains the source marking, the quotient group and
stabilizer, the cochain data, and the full marked categorical action.  Over
an arbitrary field of characteristic zero the same description holds after
fixing the rank-two input, with all scalar and categorical data defined over
the ground field.

\subsection{Infinite composites of finitely many Fibonacci factors}

Let $m\ge2$ and consider semisimple rigid tensor quotients of
$\Fib^{*m}$ with finite-dimensional Hom spaces but infinitely many simple
objects.  For $m=2$, the alternating-word induction has a sharp consequence:
if no finite relation $u_n\cong v_n$ occurs, every reduced alternating word
remains simple and distinct, and \cref{cor:free-two-fib-quotient} identifies
the marked quotient with the free product $\Fib*\Fib$.  Thus the two-factor
infinite problem does not produce an
additional family beyond the free product.

For pairwise nonisomorphic marked positive generators in the split
locally finite semisimple rigid complex-linear setting, the companion
paper \cite[Corollary~\CoxeterClassificationCorollary]{FibCoxeter}
identifies the infinite-rank possibilities with the Coxeter systems whose
off-diagonal exponents are $2,3,\infty$ and whose groups are infinite.
In particular, cycles $C_n$, $n\ge3$, with exponent three on adjacent
vertices give infinite-rank categories, although every two-generator
subcategory is a fusion category. Their construction and marked
uniqueness are established there. The forest restriction in
\cref{cor:coxeter-forest} applies to fusion targets and is therefore
consistent with these examples. Mixed Fibonacci/Yang--Lee branch
assignments lie outside this fixed-positive-branch classification.

\subsection{Fibonacci--Ising composites}

Fix the standard unitary complex Ising fusion category $\Ising$, including
its associator and labelled simple objects.  A natural next problem is to
classify faithful dominant tensor functors $\Fib*\Ising\to\calE$ with
fusion target, retaining both source embeddings and taking equivalence of
the entire marked source functor.  This replaces the pointed factor in the
present classification by an Ising factor.

\subsection{Ising--pointed composites}

For a fixed pointed fusion category $\calP$, one can instead ask for the
classification of faithful dominant tensor functors
$\Ising*\calP\to\calE$ with fusion target, retaining the Ising source
embedding and the pointed source functor, up to equivalence of the entire
marked source functor.
In the Fibonacci--pointed classification, $\Fib$ and the finite
several-Fibonacci composites have no nontrivial invertible simple objects
by \cref{cor:neutral-no-invertibles}.  Thus $\Inv(\calB)=1$ separates the
pointed labels in the normal forms $X\otimes u_h$, a key step in the
crossed-product reconstruction.  By contrast, $\Ising$ contains a
nontrivial invertible object $\psi$.  The subcategory generated by the
pointed conjugates of the Ising image contains the image of $\psi$ and may
have a nontrivial intersection with the pointed image.
Consequently the separation argument of
\cref{lem:normal-form-separation} does not apply directly.

\appendix
\definecolor{A3}{RGB}{0,150,0}
\definecolor{A4}{rgb}{1,0.5,0}
\definecolor{theta}{named}{blue}
\definecolor{rho}{named}{red}
\definecolor{mu}{named}{black}
\colorlet{unshaded}{white}

\tikzset{
  unshaded/.style={fill=white},
  ttplain/.style={thick,line cap=round,line join=round},
  ttrho/.style={thick,rho,line cap=round,line join=round},
  ttmu/.style={thick,mu,line cap=round,line join=round},
  tttheta/.style={thick,theta,line cap=round,line join=round},
}

\newcommand{\TTCurrentN}{2}
\newcommand{\TTSetN}[1]{\gdef\TTCurrentN{#1}}

\newcommand{\TTDownBundleCounts}[5]{%
  \pgfmathtruncatemacro{\TTRhoCount}{max(0,#3)}%
  \pgfmathtruncatemacro{\TTMuCount}{max(0,#4)}%
  \pgfmathtruncatemacro{\TTTotalCount}{\TTRhoCount+\TTMuCount}%
  \pgfmathtruncatemacro{\TTBundleShift}{#5}%
  \ifnum\TTRhoCount>\TTMuCount
    \pgfmathtruncatemacro{\TTBundleShift}{0}
  \else\ifnum\TTMuCount>\TTRhoCount
    \pgfmathtruncatemacro{\TTBundleShift}{1}
  \fi\fi
  \ifnum\TTTotalCount>0
    \foreach \tti in {1,...,\TTTotalCount}{%
      \pgfmathsetmacro{\ttx}{.18*(\tti-(\TTTotalCount+1)/2)}%
      \pgfmathtruncatemacro{\ttparity}{mod(\tti+\TTBundleShift,2)}%
      \ifnum\ttparity=1
        \draw[ttrho] ($#1+(\ttx,0)$) -- ($#1+(\ttx,#2)$);%
      \else
        \draw[ttmu] ($#1+(\ttx,0)$) -- ($#1+(\ttx,#2)$);%
      \fi
    }%
  \fi
}
\newcommand{\TTDownBundleA}[4]{\TTDownBundleCounts{#1}{#2}{#3}{#4}{0}}
\newcommand{\TTDownBundleB}[4]{\TTDownBundleCounts{#1}{#2}{#3}{#4}{1}}
\newcommand{\TTBottomRemainderA}[4]{%
  \pgfmathtruncatemacro{\TTRemRho}{max(0,\TTCurrentN-(#3))}%
  \pgfmathtruncatemacro{\TTRemMu}{max(0,\TTCurrentN-(#4))}%
  \TTDownBundleA{#1}{#2}{\TTRemRho}{\TTRemMu}%
}
\newcommand{\TTBottomRemainderB}[4]{%
  \pgfmathtruncatemacro{\TTRemRho}{max(0,\TTCurrentN-(#3))}%
  \pgfmathtruncatemacro{\TTRemMu}{max(0,\TTCurrentN-(#4))}%
  \TTDownBundleB{#1}{#2}{\TTRemRho}{\TTRemMu}%
}

\newcommand{\TTCouponLabel}[1]{#1}
\newcommand{\TTDualCouponLabel}[1]{%
  \begingroup\edef\ttduallabel{#1}\def\ttdualU{U}\def\ttdualV{V}%
  \ifx\ttduallabel\ttdualU V\else
    \ifx\ttduallabel\ttdualV U\else #1\fi
  \fi\endgroup}
\newcommand{\TTDiagramDual}[1]{%
  \begingroup
  \let\TTCouponLabel\TTDualCouponLabel
  \tikzset{every picture/.append style={yscale=-1}}%
  #1%
  \endgroup}

\newcommand{\TTUcircle}[2]{%
  \draw[thick,unshaded] #2 circle (.4cm);
  \node at #2 {$\TTCouponLabel{#1}$};
  \node at ($#2+(0,.55)$) {$\star$};
}
\newcommand{\TTUcircleBelowStar}[2]{%
  \draw[thick,unshaded] #2 circle (.4cm);
  \node at #2 {$\TTCouponLabel{#1}$};
  \node at ($#2+(0,-.55)$) {$\star$};
}
\newcommand{\TTUcircleLeftStar}[2]{%
  \draw[thick,unshaded] #2 circle (.4cm);
  \node at #2 {$\TTCouponLabel{#1}$};
  \node at ($#2+(-.55,0)$) {$\star$};
}
\newcommand{\TTUcircleRightStar}[2]{%
  \draw[thick,unshaded] #2 circle (.4cm);
  \node at #2 {$\TTCouponLabel{#1}$};
  \node at ($#2+(.55,0)$) {$\star$};
}
\newcommand{\TTUcircleStarAngle}[3]{%
  \draw[thick,unshaded] #2 circle (.4cm);
  \node at #2 {$\TTCouponLabel{#1}$};
  \node at ($#2+(#3:.55cm)$) {$\star$};
}

\newcommand{\TTVerticalHalfBundle}[3]{%
  \pgfmathtruncatemacro{\TTHalfCount}{\TTCurrentN}%
  \foreach \tti in {1,...,\TTHalfCount}{%
    \pgfmathsetmacro{\ttx}{.20*(\tti-(\TTHalfCount+1)/2)}%
    \pgfmathtruncatemacro{\ttparity}{mod(\tti+#3,2)}%
    \ifnum\ttparity=1
      \draw[ttrho] ($#1+(\ttx,0)$)--($#1+(\ttx,#2*.75)$);%
    \else
      \draw[ttmu] ($#1+(\ttx,0)$)--($#1+(\ttx,#2*.75)$);%
    \fi
  }%
}
\newcommand{\TTVerticalCouponA}[1]{%
\begin{tikzpicture}[baseline=-.35cm]
  \coordinate(q)at(0,0);
  \TTVerticalHalfBundle{(q)}{1}{0}
  \TTVerticalHalfBundle{(q)}{-1}{1}
  \TTUcircleLeftStar{#1}{(q)}
\end{tikzpicture}}
\newcommand{\TTVerticalCouponB}[1]{%
\begin{tikzpicture}[baseline=-.35cm]
  \coordinate(q)at(0,0);
  \TTVerticalHalfBundle{(q)}{1}{1}
  \TTVerticalHalfBundle{(q)}{-1}{0}
  \TTUcircleLeftStar{#1}{(q)}
\end{tikzpicture}}

\newcommand{\TTOneClickStraightA}{%
  \pgfmathtruncatemacro{\TTInnerCount}{\TTCurrentN-1}%
  \ifnum\TTInnerCount>0
    \foreach \tti in {1,...,\TTInnerCount}{%
      \pgfmathsetmacro{\ttx}{.18*(\tti-(\TTInnerCount+1)/2)}%
      \pgfmathtruncatemacro{\tttop}{mod(\tti,2)}%
      \pgfmathtruncatemacro{\ttbottom}{mod(\tti+mod(\TTCurrentN,2)+1,2)}%
      \ifnum\tttop=1 \draw[ttrho](\ttx,.75)--(\ttx,0);%
      \else \draw[ttmu](\ttx,.75)--(\ttx,0);\fi
      \ifnum\ttbottom=1 \draw[ttrho](\ttx,0)--(\ttx,-.75);%
      \else \draw[ttmu](\ttx,0)--(\ttx,-.75);\fi
    }%
  \fi
}
\newcommand{\TTOneClickStraightB}{%
  \pgfmathtruncatemacro{\TTInnerCount}{\TTCurrentN-1}%
  \ifnum\TTInnerCount>0
    \foreach \tti in {1,...,\TTInnerCount}{%
      \pgfmathsetmacro{\ttx}{.18*(\tti-(\TTInnerCount+1)/2)}%
      \pgfmathtruncatemacro{\tttop}{mod(\tti+1,2)}%
      \pgfmathtruncatemacro{\ttbottom}{mod(\tti+mod(\TTCurrentN,2),2)}%
      \ifnum\tttop=1 \draw[ttrho](\ttx,.75)--(\ttx,0);%
      \else \draw[ttmu](\ttx,.75)--(\ttx,0);\fi
      \ifnum\ttbottom=1 \draw[ttrho](\ttx,0)--(\ttx,-.75);%
      \else \draw[ttmu](\ttx,0)--(\ttx,-.75);\fi
    }%
  \fi
}
\newcommand{\TTOneClickRotationA}[1]{%
\begin{tikzpicture}[baseline=-.35cm]
  \coordinate(q)at(0,0);
  \draw[ttmu] (-.13,-.25) arc(0:-180:.25cm)--(-.72,.75);
  \draw[ttrho] (.13,.25) arc(180:0:.25cm)--(.72,-.75);
  \TTOneClickStraightA
  \TTUcircleLeftStar{#1}{(q)}
\end{tikzpicture}}
\newcommand{\TTOneClickRotationB}[1]{%
\begin{tikzpicture}[baseline=-.35cm]
  \coordinate(q)at(0,0);
  \draw[ttrho] (-.13,-.25) arc(0:-180:.25cm)--(-.72,.75);
  \draw[ttmu] (.13,.25) arc(180:0:.25cm)--(.72,-.75);
  \TTOneClickStraightB
  \TTUcircleLeftStar{#1}{(q)}
\end{tikzpicture}}

\newcommand{\TTUnitPairCore}[3]{%
  \coordinate(ttunitpair-left)at(-.72,0);
  \coordinate(ttunitpair-right)at(.72,0);
  \pgfmathtruncatemacro{\TTRowCount}{\TTCurrentN}
  \foreach \tti in {1,...,\TTRowCount}{%
    \pgfmathsetmacro{\tty}{.20*(\tti-(\TTRowCount+1)/2)}%
    \pgfmathtruncatemacro{\ttparity}{mod(\tti,2)}%
    \ifnum\ttparity=1
      \draw[ttrho](-1.35,\tty)--($(ttunitpair-left)+(0,\tty)$);
      \draw[ttmu]($(ttunitpair-left)+(0,\tty)$)--($(ttunitpair-right)+(0,\tty)$);
      \draw[ttrho]($(ttunitpair-right)+(0,\tty)$)--(1.35,\tty);
    \else
      \draw[ttmu](-1.35,\tty)--($(ttunitpair-left)+(0,\tty)$);
      \draw[ttrho]($(ttunitpair-left)+(0,\tty)$)--($(ttunitpair-right)+(0,\tty)$);
      \draw[ttmu]($(ttunitpair-right)+(0,\tty)$)--(1.35,\tty);
    \fi
  }
  \TTUcircleStarAngle{#1}{(ttunitpair-left)}{#3}
  \TTUcircleStarAngle{#2}{(ttunitpair-right)}{#3}
}

\newcommand{\TTUnitPair}[3]{%
\begin{tikzpicture}[baseline=-.25cm]
  \TTUnitPairCore{#1}{#2}{#3}
\end{tikzpicture}}
\newcommand{\TTUnitPairCoreColorSwap}[3]{%
  \begingroup
  \colorlet{ttsavedrho}{rho}\colorlet{rho}{mu}\colorlet{mu}{ttsavedrho}%
  \TTUnitPairCore{#1}{#2}{#3}%
  \endgroup}
\newcommand{\TTUnitPairColorSwap}[3]{%
  \begin{tikzpicture}[baseline=-.25cm]\TTUnitPairCoreColorSwap{#1}{#2}{#3}\end{tikzpicture}}

\newcommand{\TTUnitPairCapCup}[4][1]{%
\begin{tikzpicture}[baseline=-.25cm]
  \coordinate(ttcap-left)at(-.72,0);
  \coordinate(ttcap-right)at(.72,0);
  \pgfmathsetmacro{\ttCapInnerY}{.18*(#1)}
  \pgfmathsetmacro{\ttCapOuterY}{-.18*(#1)}
  \pgfmathsetmacro{\ttCapControlY}{.40*(#1)}
  \pgfmathsetmacro{\ttCapApexY}{.80*(#1)}
  \pgfmathsetmacro{\ttCapExitY}{.66*(#1)}
  \draw[ttmu](-1.35,0)--(1.35,0);
  \draw[ttrho](-1.35,\ttCapOuterY)--($(ttcap-left)+(-.34,\ttCapOuterY)$);
  \draw[ttrho]($(ttcap-right)+(.34,\ttCapOuterY)$)--(1.35,\ttCapOuterY);
  \draw[ttrho]
    ($(ttcap-left)+(.34,\ttCapInnerY)$)
    ..controls($(ttcap-left)+(.64,\ttCapControlY)$)and($(ttcap-left)+(.45,\ttCapApexY)$)..($(ttcap-left)+(0,\ttCapApexY)$)
    ..controls($(ttcap-left)+(-.28,\ttCapApexY)$)and($(ttcap-left)+(-.49,\ttCapExitY)$)..(-1.35,\ttCapExitY);
  \draw[ttrho]
    ($(ttcap-right)+(-.34,\ttCapInnerY)$)
    ..controls($(ttcap-right)+(-.64,\ttCapControlY)$)and($(ttcap-right)+(-.45,\ttCapApexY)$)..($(ttcap-right)+(0,\ttCapApexY)$)
    ..controls($(ttcap-right)+(.28,\ttCapApexY)$)and($(ttcap-right)+(.49,\ttCapExitY)$)..(1.35,\ttCapExitY);
  \pgfmathtruncatemacro{\TTExtraRows}{max(0,\TTCurrentN-2)}
  \ifnum\TTExtraRows>0
    \foreach \tti in {1,...,\TTExtraRows}{%
      \pgfmathsetmacro{\tty}{-(#1)*(.34+.16*(\tti-1))}%
      \draw[ttrho](-1.35,\tty)--($(ttcap-left)+(0,\tty)$);
      \draw[ttmu]($(ttcap-left)+(0,\tty)$)--($(ttcap-right)+(0,\tty)$);
      \draw[ttrho]($(ttcap-right)+(0,\tty)$)--(1.35,\tty);
    }
  \fi
  \TTUcircleStarAngle{#2}{(ttcap-left)}{#4}
  \TTUcircleStarAngle{#3}{(ttcap-right)}{#4}
\end{tikzpicture}}

\newcommand{\TTUnitPairTrivalent}[4][1]{%
\begin{tikzpicture}[baseline=-.25cm]
  \coordinate(tttri-left)at(-.72,0);
  \coordinate(tttri-right)at(.72,0);
  \pgfmathsetmacro{\ttTriInnerY}{.18*(#1)}
  \pgfmathsetmacro{\ttTriOuterY}{-.18*(#1)}
  \pgfmathsetmacro{\ttTriControlOneY}{.40*(#1)}
  \pgfmathsetmacro{\ttTriControlTwoY}{.78*(#1)}
  \pgfmathsetmacro{\ttTriVertexY}{.96*(#1)}
  \coordinate(tttri-vertex-left)at(-.72,\ttTriVertexY);
  \coordinate(tttri-vertex-right)at(.72,\ttTriVertexY);
  \draw[ttmu](-1.55,0)--(1.55,0);
  \draw[ttrho](-1.55,\ttTriOuterY)--($(tttri-left)+(-.34,\ttTriOuterY)$);
  \draw[ttrho]($(tttri-right)+(.34,\ttTriOuterY)$)--(1.55,\ttTriOuterY);
  \draw[ttrho](-1.55,\ttTriVertexY)--(1.55,\ttTriVertexY);
  \draw[ttrho]
    ($(tttri-left)+(.34,\ttTriInnerY)$)
    ..controls($(tttri-left)+(.64,\ttTriControlOneY)$)and($(tttri-left)+(.44,\ttTriControlTwoY)$)..(tttri-vertex-left);
  \draw[ttrho]
    ($(tttri-right)+(-.34,\ttTriInnerY)$)
    ..controls($(tttri-right)+(-.64,\ttTriControlOneY)$)and($(tttri-right)+(-.44,\ttTriControlTwoY)$)..(tttri-vertex-right);
  \filldraw[rho](tttri-vertex-left)circle(.05cm);
  \filldraw[rho](tttri-vertex-right)circle(.05cm);
  \pgfmathtruncatemacro{\TTExtraRows}{max(0,\TTCurrentN-2)}
  \ifnum\TTExtraRows>0
    \foreach \tti in {1,...,\TTExtraRows}{%
      \pgfmathsetmacro{\tty}{-(#1)*(.34+.16*(\tti-1))}%
      \draw[ttrho](-1.55,\tty)--($(tttri-left)+(0,\tty)$);
      \draw[ttmu]($(tttri-left)+(0,\tty)$)--($(tttri-right)+(0,\tty)$);
      \draw[ttrho]($(tttri-right)+(0,\tty)$)--(1.55,\tty);
    }
  \fi
  \TTUcircleStarAngle{#2}{(tttri-left)}{#4}
  \TTUcircleStarAngle{#3}{(tttri-right)}{#4}
\end{tikzpicture}}

\newcommand{\TTUnitPairTrivalentColorSwap}[4][1]{%
\begin{tikzpicture}[baseline=-.25cm]
  \coordinate(tttris-left)at(-.72,0);
  \coordinate(tttris-right)at(.72,0);
  \pgfmathsetmacro{\ttTriSInnerY}{.18*(#1)}
  \pgfmathsetmacro{\ttTriSOuterY}{-.18*(#1)}
  \pgfmathsetmacro{\ttTriSControlOneY}{.40*(#1)}
  \pgfmathsetmacro{\ttTriSControlTwoY}{.78*(#1)}
  \pgfmathsetmacro{\ttTriSVertexY}{.96*(#1)}
  \coordinate(tttris-vertex-left)at(-.72,\ttTriSVertexY);
  \coordinate(tttris-vertex-right)at(.72,\ttTriSVertexY);
  \draw[ttrho](-1.55,0)--(1.55,0);
  \draw[ttmu](-1.55,\ttTriSOuterY)--($(tttris-left)+(-.34,\ttTriSOuterY)$);
  \draw[ttmu]($(tttris-right)+(.34,\ttTriSOuterY)$)--(1.55,\ttTriSOuterY);
  \draw[ttmu](-1.55,\ttTriSVertexY)--(1.55,\ttTriSVertexY);
  \draw[ttmu]
    ($(tttris-left)+(.34,\ttTriSInnerY)$)
    ..controls($(tttris-left)+(.64,\ttTriSControlOneY)$)and($(tttris-left)+(.44,\ttTriSControlTwoY)$)..(tttris-vertex-left);
  \draw[ttmu]
    ($(tttris-right)+(-.34,\ttTriSInnerY)$)
    ..controls($(tttris-right)+(-.64,\ttTriSControlOneY)$)and($(tttris-right)+(-.44,\ttTriSControlTwoY)$)..(tttris-vertex-right);
  \filldraw[mu](tttris-vertex-left)circle(.05cm);
  \filldraw[mu](tttris-vertex-right)circle(.05cm);
  \pgfmathtruncatemacro{\TTExtraRows}{max(0,\TTCurrentN-2)}
  \ifnum\TTExtraRows>0
    \foreach \tti in {1,...,\TTExtraRows}{%
      \pgfmathsetmacro{\tty}{-(#1)*(.34+.16*(\tti-1))}%
      \draw[ttmu](-1.55,\tty)--($(tttris-left)+(0,\tty)$);
      \draw[ttrho]($(tttris-left)+(0,\tty)$)--($(tttris-right)+(0,\tty)$);
      \draw[ttmu]($(tttris-right)+(0,\tty)$)--(1.55,\tty);
    }
  \fi
  \TTUcircleStarAngle{#2}{(tttris-left)}{#4}
  \TTUcircleStarAngle{#3}{(tttris-right)}{#4}
\end{tikzpicture}}

\newcommand{\TTJellyfishInput}{%
\begin{tikzpicture}[baseline=-.30cm]
  \coordinate(ttinput-vertex)at(0,0);
  \draw[ttrho](-.58,.30)--(.58,.30);
  \draw[ttmu](ttinput-vertex)--(-.52,-.56);
  \draw[ttmu](ttinput-vertex)--(0,-.56);
  \draw[ttmu](ttinput-vertex)--(.52,-.56);
  \filldraw[mu](ttinput-vertex)circle(.05cm);
\end{tikzpicture}}

\newcommand{\TTThreeCouponJellyfish}{%
\begin{tikzpicture}[baseline=-.35cm]
  \coordinate(a)at(-1.20,0);\coordinate(b)at(0,0);\coordinate(c)at(1.20,0);
  \coordinate(t)at(0,1.0);
  \draw[ttrho](-1.70,0)--(1.70,0);
  \draw[ttmu](t)--(-1.70,-.65);
  \draw[ttmu](t)--(0,-.65);
  \draw[ttmu](t)--(1.70,-.65);
  \filldraw[mu](t)circle(.05cm);
  \TTBottomRemainderA{(a)}{-.65}{2}{2}
  \TTBottomRemainderA{(b)}{-.65}{2}{2}
  \TTBottomRemainderA{(c)}{-.65}{2}{2}
  \TTUcircleStarAngle{U}{(a)}{90}
  \TTUcircleStarAngle{U}{(b)}{125}
  \TTUcircleStarAngle{V}{(c)}{90}
\end{tikzpicture}}

\newcommand{\TTSwappedJellyfishInput}{%
\begin{tikzpicture}[baseline=-.30cm]
  \coordinate(ttswap-input-vertex)at(0,0);
  \draw[ttmu](-.58,.30)--(.58,.30);
  \draw[ttrho](ttswap-input-vertex)--(-.52,-.56);
  \draw[ttrho](ttswap-input-vertex)--(0,-.56);
  \draw[ttrho](ttswap-input-vertex)--(.52,-.56);
  \filldraw[rho](ttswap-input-vertex)circle(.05cm);
\end{tikzpicture}}

\newcommand{\TTSwappedThreeCouponJellyfish}{%
\begin{tikzpicture}[baseline=-.35cm]
  \coordinate(ttswap-a)at(-1.20,0);
  \coordinate(ttswap-b)at(0,0);
  \coordinate(ttswap-c)at(1.20,0);
  \coordinate(ttswap-top)at(0,1.0);
  \draw[ttmu](-1.75,0)--(1.75,0);
  \draw[ttrho](ttswap-top)--($(ttswap-top)!1.65!(ttswap-a)$);
  \draw[ttrho](ttswap-top)--($(ttswap-top)!1.65!(ttswap-b)$);
  \draw[ttrho](ttswap-top)--($(ttswap-top)!1.65!(ttswap-c)$);
  \filldraw[rho](ttswap-top)circle(.05cm);
  \TTBottomRemainderA{(ttswap-a)}{-.68}{2}{2}
  \TTBottomRemainderB{(ttswap-b)}{-.68}{2}{2}
  \TTBottomRemainderA{(ttswap-c)}{-.68}{2}{2}
  \TTUcircleStarAngle{V}{(ttswap-a)}{90}
  \TTUcircleStarAngle{V}{(ttswap-b)}{125}
  \TTUcircleStarAngle{U}{(ttswap-c)}{90}
\end{tikzpicture}}

\newcommand{\TTOddJellyfishInputRho}{%
\begin{tikzpicture}[baseline=-.30cm]
  \coordinate(ttodd-input)at(0,0);
  \draw[ttrho](-.62,.38)--(.62,.38);
  \draw[ttmu](-.62,.20)--(.62,.20);
  \draw[ttrho](ttodd-input)--(-.52,-.56);
  \draw[ttrho](ttodd-input)--(0,-.56);
  \draw[ttrho](ttodd-input)--(.52,-.56);
  \filldraw[rho](ttodd-input)circle(.05cm);
\end{tikzpicture}}
\newcommand{\TTOddThreeCouponJellyfishRho}{%
\begin{tikzpicture}[baseline=-.35cm]
  \coordinate(ttodd-a)at(-1.20,0);
  \coordinate(ttodd-b)at(0,0);
  \coordinate(ttodd-c)at(1.20,0);
  \coordinate(ttodd-top)at(0,1.0);
  \draw[ttrho](-1.75,.10)--(1.75,.10);
  \draw[ttmu](-1.75,-.10)--(1.75,-.10);
  \draw[ttmu](ttodd-top)--(ttodd-a);
  \draw[ttmu](ttodd-top)--(ttodd-b);
  \draw[ttmu](ttodd-top)--(ttodd-c);
  \filldraw[mu](ttodd-top)circle(.05cm);
  \draw[ttrho](ttodd-a)--(-1.75,-.68);
  \draw[ttrho](ttodd-b)--(0,-.68);
  \draw[ttrho](ttodd-c)--(1.75,-.68);
  \TTUcircleStarAngle{U}{(ttodd-a)}{90}
  \TTUcircleStarAngle{U}{(ttodd-b)}{125}
  \TTUcircleStarAngle{V}{(ttodd-c)}{90}
\end{tikzpicture}}
\newcommand{\TTOddJellyfishInputMu}{%
\begin{tikzpicture}[baseline=-.30cm]
  \coordinate(ttodd-swap-input)at(0,0);
  \draw[ttmu](-.62,.38)--(.62,.38);
  \draw[ttrho](-.62,.20)--(.62,.20);
  \draw[ttmu](ttodd-swap-input)--(-.52,-.56);
  \draw[ttmu](ttodd-swap-input)--(0,-.56);
  \draw[ttmu](ttodd-swap-input)--(.52,-.56);
  \filldraw[mu](ttodd-swap-input)circle(.05cm);
\end{tikzpicture}}
\newcommand{\TTOddThreeCouponJellyfishMu}{%
\begin{tikzpicture}[baseline=-.35cm]
  \coordinate(ttodd-swap-a)at(-1.20,0);
  \coordinate(ttodd-swap-b)at(0,0);
  \coordinate(ttodd-swap-c)at(1.20,0);
  \coordinate(ttodd-swap-top)at(0,1.0);
  \draw[ttmu](-1.75,.10)--(1.75,.10);
  \draw[ttrho](-1.75,-.10)--(1.75,-.10);
  \draw[ttrho](ttodd-swap-top)--(ttodd-swap-a);
  \draw[ttrho](ttodd-swap-top)--(ttodd-swap-b);
  \draw[ttrho](ttodd-swap-top)--(ttodd-swap-c);
  \filldraw[rho](ttodd-swap-top)circle(.05cm);
  \draw[ttmu](ttodd-swap-a)--(-1.75,-.68);
  \draw[ttmu](ttodd-swap-b)--(0,-.68);
  \draw[ttmu](ttodd-swap-c)--(1.75,-.68);
  \TTUcircleStarAngle{V}{(ttodd-swap-a)}{90}
  \TTUcircleStarAngle{V}{(ttodd-swap-b)}{125}
  \TTUcircleStarAngle{U}{(ttodd-swap-c)}{90}
\end{tikzpicture}}

\newcommand{\TTDoubleSwappedJellyfishInsertion}{%
\begin{tikzpicture}[baseline=-.45cm]
  \foreach \tti in {1,...,8}{%
    \pgfmathsetmacro{\ttx}{\tti-4.5}%
    \coordinate(ttinsert-q\tti)at(\ttx,0);%
  }
  \coordinate(ttinsert-leftvertex)at(-1.5,1.0);
  \coordinate(ttinsert-rightvertex)at(1.5,1.0);
  \draw[ttmu](-4.30,0)--(4.30,0);
  \draw[ttrho]
    (-4.30,.45)--(ttinsert-q1)
    ..controls ++(-45:.60cm) and ++(225:.60cm)..(ttinsert-q2)
    --(ttinsert-leftvertex);
  \draw[ttrho]
    (ttinsert-leftvertex)--(ttinsert-q4)
    ..controls ++(-45:.60cm) and ++(225:.60cm)..(ttinsert-q5)
    --(ttinsert-rightvertex);
  \draw[ttrho]
    (ttinsert-rightvertex)--(ttinsert-q7)
    ..controls ++(-45:.60cm) and ++(225:.60cm)..(ttinsert-q8)
    --(4.30,.45);
  \draw[ttrho]
    (ttinsert-leftvertex)--(-1.5,-.40)
    arc(0:-90:.20cm)--(-4.30,-.60);
  \draw[ttrho]
    (ttinsert-rightvertex)--(1.5,-.40)
    arc(180:270:.20cm)--(4.30,-.60);
  \filldraw[rho](ttinsert-leftvertex)circle(.05cm);
  \filldraw[rho](ttinsert-rightvertex)circle(.05cm);
  \foreach \tti in {1,...,8}{%
    \pgfmathtruncatemacro{\ttphase}{mod(\tti-1,2)}%
    \ifnum\ttphase=0
      \TTBottomRemainderA{(ttinsert-q\tti)}{-.72}{2}{2}%
    \else
      \TTBottomRemainderB{(ttinsert-q\tti)}{-.72}{2}{2}%
    \fi
  }
  \TTUcircleStarAngle{U}{(ttinsert-q1)}{90}
  \TTUcircleStarAngle{V}{(ttinsert-q2)}{90}
  \TTUcircleStarAngle{V}{(ttinsert-q3)}{120}
  \TTUcircleStarAngle{U}{(ttinsert-q4)}{90}
  \TTUcircleStarAngle{V}{(ttinsert-q5)}{90}
  \TTUcircleStarAngle{V}{(ttinsert-q6)}{120}
  \TTUcircleStarAngle{U}{(ttinsert-q7)}{90}
  \TTUcircleStarAngle{V}{(ttinsert-q8)}{90}
\end{tikzpicture}}

\newcommand{\TTDoubleJellyfishInsertion}{%
\begin{tikzpicture}[baseline=-.45cm]
  \foreach \tti in {1,...,8}{%
    \pgfmathsetmacro{\ttx}{\tti-4.5}%
    \coordinate(ttnormal-q\tti)at(\ttx,0);%
  }
  \coordinate(ttnormal-leftvertex)at(-1.5,1.0);
  \coordinate(ttnormal-rightvertex)at(1.5,1.0);
  \draw[ttrho](-4.30,0)--(4.30,0);
  \draw[ttmu]
    (-4.30,.45)--(ttnormal-q1)
    ..controls ++(-45:.60cm) and ++(225:.60cm)..(ttnormal-q2)
    --(ttnormal-leftvertex);
  \draw[ttmu]
    (ttnormal-leftvertex)--(ttnormal-q4)
    ..controls ++(-45:.60cm) and ++(225:.60cm)..(ttnormal-q5)
    --(ttnormal-rightvertex);
  \draw[ttmu]
    (ttnormal-rightvertex)--(ttnormal-q7)
    ..controls ++(-45:.60cm) and ++(225:.60cm)..(ttnormal-q8)
    --(4.30,.45);
  \draw[ttmu]
    (ttnormal-leftvertex)--(-1.5,-.40)
    arc(0:-90:.20cm)--(-4.30,-.60);
  \draw[ttmu]
    (ttnormal-rightvertex)--(1.5,-.40)
    arc(180:270:.20cm)--(4.30,-.60);
  \filldraw[mu](ttnormal-leftvertex)circle(.05cm);
  \filldraw[mu](ttnormal-rightvertex)circle(.05cm);
  \foreach \tti in {1,...,8}{%
    \pgfmathtruncatemacro{\ttphase}{mod(\tti-1,2)}%
    \ifnum\ttphase=0
      \TTBottomRemainderA{(ttnormal-q\tti)}{-.72}{2}{2}%
    \else
      \TTBottomRemainderB{(ttnormal-q\tti)}{-.72}{2}{2}%
    \fi
  }
  \TTUcircleStarAngle{V}{(ttnormal-q1)}{90}
  \TTUcircleStarAngle{U}{(ttnormal-q2)}{90}
  \TTUcircleStarAngle{U}{(ttnormal-q3)}{120}
  \TTUcircleStarAngle{V}{(ttnormal-q4)}{90}
  \TTUcircleStarAngle{U}{(ttnormal-q5)}{90}
  \TTUcircleStarAngle{U}{(ttnormal-q6)}{120}
  \TTUcircleStarAngle{V}{(ttnormal-q7)}{90}
  \TTUcircleStarAngle{U}{(ttnormal-q8)}{90}
\end{tikzpicture}}

\newcommand{\TTBundleA}[2]{%
  \TTDownBundleA{#1}{#2}{\TTCurrentN}{\TTCurrentN}%
}
\newcommand{\TTBundleB}[2]{%
  \TTDownBundleB{#1}{#2}{\TTCurrentN}{\TTCurrentN}%
}

\newcommand{\TTHorizontalBundleA}{%
\begin{tikzpicture}[baseline=-.2cm]
  \pgfmathtruncatemacro{\TTLineCount}{\TTCurrentN+1}
  \foreach \tti in {1,...,\TTLineCount}{%
    \pgfmathsetmacro{\tty}{.18*((\TTLineCount+1)/2-\tti)}%
    \pgfmathtruncatemacro{\ttparity}{mod(\tti,2)}%
    \ifnum\ttparity=1 \draw[ttrho](-1,\tty)--(1,\tty);%
    \else \draw[ttmu](-1,\tty)--(1,\tty);\fi
  }
\end{tikzpicture}}

\newcommand{\TTNTwoUpperUVInsertion}{%
\begin{tikzpicture}[baseline=-.25cm]
  \draw[ttrho](-1.35,-.72)--(1.35,-.72);
  \TTUnitPairCoreColorSwap{U}{V}{90}
\end{tikzpicture}}
\newcommand{\TTNTwoLowerVUInsertion}{%
\begin{tikzpicture}[baseline=-.25cm]
  \draw[ttrho](-1.35,.72)--(1.35,.72);
  \TTUnitPairCore{V}{U}{90}
\end{tikzpicture}}
\newcommand{\TTNThreeUpperUVInsertion}{%
\begin{tikzpicture}[baseline=-.25cm]
  \draw[ttmu](-1.35,-.72)--(1.35,-.72);
  \TTUnitPairCore{U}{V}{90}
\end{tikzpicture}}
\newcommand{\TTNThreeLowerVUInsertion}{%
\begin{tikzpicture}[baseline=-.25cm]
  \draw[ttrho](-1.35,.72)--(1.35,.72);
  \TTUnitPairCoreColorSwap{V}{U}{90}
\end{tikzpicture}}

\newcommand{\TTNThreeUpperCapCup}{%
\begin{tikzpicture}[baseline=-.30cm]
  \coordinate(l)at(-.72,0);\coordinate(r)at(.72,0);
  \draw[ttrho](-1.55,.22)--(l);\draw[ttmu](l)--(r);\draw[ttrho](r)--(1.55,.22);
  \draw[ttmu](-1.55,0)--(l);\draw[ttrho](l)--(r);\draw[ttmu](r)--(1.55,0);
  \draw[ttrho](-1.55,-.22)--(l);\draw[ttrho](r)--(1.55,-.22);
  \draw[ttmu]($(l)+(.34,-.18)$)
    ..controls($(l)+(.64,-.40)$)and($(l)+(.45,-.78)$)..($(l)+(0,-.78)$)
    ..controls($(l)+(-.28,-.78)$)and($(l)+(-.49,-.66)$)..(-1.55,-.66);
  \draw[ttmu]($(r)+(-.34,-.18)$)
    ..controls($(r)+(-.64,-.40)$)and($(r)+(-.45,-.78)$)..($(r)+(0,-.78)$)
    ..controls($(r)+(.28,-.78)$)and($(r)+(.49,-.66)$)..(1.55,-.66);
  \TTUcircleStarAngle{U}{(l)}{90}\TTUcircleStarAngle{V}{(r)}{90}
  \draw[ttmu](-.38,.22)--(.38,.22);
\end{tikzpicture}}
\newcommand{\TTNThreeUpperTrivalent}{%
\begin{tikzpicture}[baseline=-.30cm]
  \coordinate(l)at(-.72,0);\coordinate(r)at(.72,0);
  \coordinate(tl)at(-.72,-.88);\coordinate(tr)at(.72,-.88);
  \draw[ttrho](-1.55,.22)--(l);\draw[ttmu](l)--(r);\draw[ttrho](r)--(1.55,.22);
  \draw[ttmu](-1.55,0)--(l);\draw[ttrho](l)--(r);\draw[ttmu](r)--(1.55,0);
  \draw[ttrho](-1.55,-.22)--(l);\draw[ttrho](r)--(1.55,-.22);
  \draw[ttmu](tl)--(tr);
  \draw[ttmu]($(l)+(.34,-.18)$)..controls(-.18,-.45)and(-.45,-.70)..(tl);
  \draw[ttmu]($(r)+(-.34,-.18)$)..controls(.18,-.45)and(.45,-.70)..(tr);
  \draw[ttmu](-1.55,-.66)--(tl);\draw[ttmu](tr)--(1.55,-.66);
  \filldraw[mu](tl)circle(.05cm);\filldraw[mu](tr)circle(.05cm);
  \TTUcircleStarAngle{U}{(l)}{90}\TTUcircleStarAngle{V}{(r)}{90}
  \draw[ttmu](-.38,.22)--(.38,.22);
\end{tikzpicture}}

\newcommand{\TTNThreeLowerCapCup}{%
\begin{tikzpicture}[baseline=-.30cm]
  \coordinate(l)at(-.72,0);\coordinate(r)at(.72,0);
  \draw[ttmu](-1.55,.22)--(l);\draw[ttmu](r)--(1.55,.22);
  \draw[ttrho](-1.55,0)--(l);\draw[ttmu](l)--(r);\draw[ttrho](r)--(1.55,0);
  \draw[ttmu](-1.55,-.22)--(l);\draw[ttrho]($(l)+(0,-.22)$)--($(r)+(0,-.22)$);\draw[ttmu](r)--(1.55,-.22);
  \draw[ttrho]($(l)+(.34,.18)$)
    ..controls($(l)+(.64,.40)$)and($(l)+(.45,.78)$)..($(l)+(0,.78)$)
    ..controls($(l)+(-.28,.78)$)and($(l)+(-.49,.66)$)..(-1.55,.66);
  \draw[ttrho]($(r)+(-.34,.18)$)
    ..controls($(r)+(-.64,.40)$)and($(r)+(-.45,.78)$)..($(r)+(0,.78)$)
    ..controls($(r)+(.28,.78)$)and($(r)+(.49,.66)$)..(1.55,.66);
  \TTUcircleStarAngle{V}{(l)}{90}\TTUcircleStarAngle{U}{(r)}{90}
\end{tikzpicture}}
\newcommand{\TTNThreeLowerTrivalent}{%
\begin{tikzpicture}[baseline=-.30cm]
  \coordinate(l)at(-.72,0);\coordinate(r)at(.72,0);
  \coordinate(tl)at(-.72,.88);\coordinate(tr)at(.72,.88);
  \draw[ttmu](-1.55,.22)--(l);\draw[ttmu](r)--(1.55,.22);
  \draw[ttrho](-1.55,0)--(l);\draw[ttmu](l)--(r);\draw[ttrho](r)--(1.55,0);
  \draw[ttmu](-1.55,-.22)--(l);\draw[ttrho]($(l)+(0,-.22)$)--($(r)+(0,-.22)$);\draw[ttmu](r)--(1.55,-.22);
  \draw[ttrho](tl)--(tr);
  \draw[ttrho]($(l)+(.34,.18)$)..controls(-.18,.45)and(-.45,.70)..(tl);
  \draw[ttrho]($(r)+(-.34,.18)$)..controls(.18,.45)and(.45,.70)..(tr);
  \draw[ttrho](-1.55,.66)--(tl);\draw[ttrho](tr)--(1.55,.66);
  \filldraw[rho](tl)circle(.05cm);\filldraw[rho](tr)circle(.05cm);
  \TTUcircleStarAngle{V}{(l)}{90}\TTUcircleStarAngle{U}{(r)}{90}
\end{tikzpicture}}

\newcommand{\TTHorizontalPair}[2]{%
  \draw[ttrho] ($#1+(0,.10)$) -- ($#2+(0,.10)$);
  \draw[ttmu]  ($#1+(0,-.10)$) -- ($#2+(0,-.10)$);
}
\newcommand{\TTHorizontalPairSwap}[2]{%
  \draw[ttmu]  ($#1+(0,.10)$) -- ($#2+(0,.10)$);
  \draw[ttrho] ($#1+(0,-.10)$) -- ($#2+(0,-.10)$);
}

\newcommand{\TTTrainGapCounts}[4]{%
  \pgfmathtruncatemacro{\TTGapDirectCount}{#3}%
  \pgfmathtruncatemacro{\TTGapPhase}{mod(#4,2)}%
  \ifnum\TTGapDirectCount<0
    \PackageError{tt_diagram_macros}{A train-gap count cannot be negative}%
      {Use a nonnegative integer for the number of direct strands.}%
  \fi
  \ifnum\TTGapDirectCount>0
    \foreach \ttj in {1,...,\TTGapDirectCount}{%
      \pgfmathsetmacro{\tty}{.11+.14*(\ttj-(\TTGapDirectCount+1)/2)}%
      \pgfmathtruncatemacro{\ttcolour}{mod(\ttj-1+\TTGapPhase,2)}%
      \ifnum\ttcolour=0
        \draw[ttrho] ($#1+(0,\tty)$)--($#2+(0,\tty)$);%
      \else
        \draw[ttmu] ($#1+(0,\tty)$)--($#2+(0,\tty)$);%
      \fi
    }%
  \fi
  \coordinate (ttgap) at ($.5*#1+.5*#2+(0,-.50)$);
  \pgfmathtruncatemacro{\TTGapTrivalentColour}%
    {mod(\TTGapDirectCount+\TTGapPhase,2)}%
  \ifnum\TTGapTrivalentColour=0
    \filldraw[rho](ttgap)circle(.05cm);
    \draw[ttrho] #1--(ttgap)--#2;
    \draw[ttrho](ttgap)--($(ttgap)+(0,-.53)$);
  \else
    \filldraw[mu](ttgap)circle(.05cm);
    \draw[ttmu] #1--(ttgap)--#2;
    \draw[ttmu](ttgap)--($(ttgap)+(0,-.53)$);
  \fi
  \pgfmathtruncatemacro{\TTGapHalfLow}{floor(\TTGapDirectCount/2)}%
  \pgfmathtruncatemacro{\TTGapHalfHigh}{\TTGapDirectCount-\TTGapHalfLow}%
  \ifnum\TTGapPhase=0
    \pgfmathtruncatemacro{\TTGapRhoUsed}{\TTGapHalfHigh}%
    \pgfmathtruncatemacro{\TTGapMuUsed}{\TTGapHalfLow}%
  \else
    \pgfmathtruncatemacro{\TTGapRhoUsed}{\TTGapHalfLow}%
    \pgfmathtruncatemacro{\TTGapMuUsed}{\TTGapHalfHigh}%
  \fi
  \ifnum\TTGapTrivalentColour=0
    \pgfmathtruncatemacro{\TTGapRhoUsed}{\TTGapRhoUsed+1}%
  \else
    \pgfmathtruncatemacro{\TTGapMuUsed}{\TTGapMuUsed+1}%
  \fi
}

\newcommand{\TTTrainGapA}[2]{\TTTrainGapCounts{#1}{#2}{1}{0}}
\newcommand{\TTTrainGapB}[2]{\TTTrainGapCounts{#1}{#2}{1}{1}}

\newcommand{\TTTrivalentDownRho}[3]{%
  \coordinate (ttv) at (#2,#3);
  \filldraw[rho] (ttv) circle (.05cm);
  \draw[ttrho] (ttv) -- ($(ttv)+(-.55,.45)$);
  \draw[ttrho] (ttv) -- ($(ttv)+(.55,.45)$);
  \draw[ttrho] (ttv) -- ($(ttv)+(0,-.55)$);
}
\newcommand{\TTTrivalentDownMu}[3]{%
  \coordinate (ttv) at (#2,#3);
  \filldraw[mu] (ttv) circle (.05cm);
  \draw[ttmu] (ttv) -- ($(ttv)+(-.55,.45)$);
  \draw[ttmu] (ttv) -- ($(ttv)+(.55,.45)$);
  \draw[ttmu] (ttv) -- ($(ttv)+(0,-.55)$);
}

\newcommand{\TTAttachedDomeRho}[2]{%
  \draw[ttrho] ($#1+(-#2,-.8)$) -- ($#1+(-#2,0)$) arc(180:0:#2 cm) -- ($#1+(#2,-.8)$);
  \coordinate (tta) at ($#1+(0,#2)$);
  \filldraw[rho] (tta) circle (.05cm);
  \draw[ttrho] (tta) -- #1;
}
\newcommand{\TTAttachedDomeMu}[2]{%
  \draw[ttmu] ($#1+(-#2,-.8)$) -- ($#1+(-#2,0)$) arc(180:0:#2 cm) -- ($#1+(#2,-.8)$);
  \coordinate (tta) at ($#1+(0,#2)$);
  \filldraw[mu] (tta) circle (.05cm);
  \draw[ttmu] (tta) -- #1;
}
\newcommand{\TTAttachedDomeBlue}[2]{%
  \draw[tttheta] ($#1+(-#2,-.8)$) -- ($#1+(-#2,0)$) arc(180:0:#2 cm) -- ($#1+(#2,-.8)$);
  \coordinate (tta) at ($#1+(0,#2)$);
  \filldraw[theta] (tta) circle (.05cm);
  \draw[tttheta] (tta) -- #1;
}

\newcommand{\TTFloatingRoofPath}[5]{%
  \begingroup
  \pgfmathsetmacro{\ttfloatradius}{(#4-#2)/2}%
  \pgfmathsetmacro{\ttfloatrise}{min(abs(\ttfloatradius),#5/2)}%
  \pgfmathsetmacro{\ttfloatshoulder}{#3+#5-\ttfloatrise}%
  \draw[#1] (#2,#3) -- (#2,\ttfloatshoulder)
    arc[start angle=180,end angle=0,x radius=\ttfloatradius cm,y radius=\ttfloatrise cm]
    -- (#4,#3);
  \endgroup
}
\newcommand{\TTFloatingRoofRho}[4]{\TTFloatingRoofPath{ttrho}{#1}{#2}{#3}{#4}}
\newcommand{\TTFloatingRoofMu}[4]{\TTFloatingRoofPath{ttmu}{#1}{#2}{#3}{#4}}
\newcommand{\TTFloatingRoofBlue}[4]{\TTFloatingRoofPath{tttheta}{#1}{#2}{#3}{#4}}
\newcommand{\TTFloatingRoofBlack}[4]{\TTFloatingRoofPath{ttplain}{#1}{#2}{#3}{#4}}

\newcommand{\TTWheelRho}[2]{\draw[ttrho] (#1-.20,#2) arc(180:0:.20cm);}
\newcommand{\TTWheelMu}[2]{\draw[ttmu] (#1-.20,#2) arc(180:0:.20cm);}
\newcommand{\TTWheelBlue}[2]{\draw[tttheta] (#1-.20,#2) arc(180:0:.20cm);}

\newcommand{\TTTrainColour}[2]{%
  \pgfmathtruncatemacro{\ttcolour}{mod(#2+\ttstartphase+#1+\ttmirror,2)}%
  \ifnum\ttcolour=0\def\ttstyle{ttrho}\else\def\ttstyle{ttmu}\fi
}
\newcommand{\TTTrainPort}[4]{%
  \coordinate (tp#1-#2) at ($(q#1)+(#3:.4cm)$);
  \expandafter\xdef\csname ttportangle@#1@#2\endcsname{#3}%
  \TTTrainColour{#1}{#2}%
}
\newcommand{\TTTrainCoupon}[2]{%
  \ifnum#2=0\TTUcircle{U}{(q#1)}\else\TTUcircle{V}{(q#1)}\fi
}
\newcommand{\TTTrainInteger}[2]{%
  \pgfmathparse{(#1)==floor(#1) && (#1)>=0}%
  \ifdim\pgfmathresult pt=1pt\else
    \PackageError{tt_diagram_macros}{Invalid train #2: #1}%
      {Train counts must be nonnegative integers.}%
  \fi
}
\newcommand{\TTTrainDrawDirect}[4]{\draw[\ttstyle] (tp#1-#2)--(tp#3-#4);}
\newcommand{\TTTrainEngine}[5]{%
  \def\ttmirror{#1}%
  \TTTrainInteger{\TTCurrentN}{n}%
  \ifnum\TTCurrentN<1
    \PackageError{tt_diagram_macros}{Train n must be positive}{Use \string\TTSetN{1} or larger.}%
  \fi
  \TTTrainInteger{#4}{cap count}%
  \pgfmathtruncatemacro{\ttcaps}{#4}%
  \pgfmathtruncatemacro{\ttvalence}{2*\TTCurrentN}%
  \def\ttlabel{#3}\def\ttU{U}\def\ttV{V}%
  \ifx\ttlabel\ttU\def\ttstartphase{0}%
  \else\ifx\ttlabel\ttV\def\ttstartphase{1}%
  \else
    \PackageError{tt_diagram_macros}{Invalid initial train label}%
      {Use U or V.}%
    \def\ttstartphase{0}%
  \fi\fi
  \def\ttalist{#2}%
  \ifx\ttalist\empty
    \PackageError{tt_diagram_macros}{Empty train list}%
      {Supply a_0 followed by at least one gap count.}%
  \fi
  \xdef\TTTrainLength{0}%
  \foreach \ttentry [count=\tti from 0] in {#2}{%
    \TTTrainInteger{\ttentry}{a_\tti}%
    \pgfmathtruncatemacro{\ttvalue}{\ttentry}%
    \expandafter\xdef\csname tta@\tti\endcsname{\ttvalue}%
    \pgfmathtruncatemacro{\ttnext}{\tti+1}%
    \xdef\TTTrainLength{\ttnext}%
  }%
  \edef\N{\TTTrainLength}%
  \ifnum\N<2
    \PackageError{tt_diagram_macros}{A train needs at least one gap}%
      {Use \string\TTLeftTrainOneCar\space or \string\TTRightTrainOneCar\space for a singleton.}%
  \fi
  \pgfmathtruncatemacro{\Nm}{\N-1}%
  \pgfmathsetmacro{\ttbottom}{-1.35-.18*\ttcaps}%
  \pgfmathsetmacro{\ttforkx}{-.72-.09*\ttcaps}%
  \pgfmathsetmacro{\ttforkhalf}{.24+.09*\ttcaps}%
  \foreach \tti in {1,...,\N}{%
    \coordinate(q\tti) at ({1.55*(\tti-1)},0);
    \pgfmathtruncatemacro{\ttprev}{\tti-1}%
    \ifnum\tti=1
      \pgfmathtruncatemacro{\ttleftused}{#5*(\csname tta@0\endcsname+1)}%
    \else
      \pgfmathtruncatemacro{\ttleftused}{\csname tta@\ttprev\endcsname+1}%
    \fi
    \ifnum\tti<\N
      \pgfmathtruncatemacro{\ttrightused}{\csname tta@\tti\endcsname+1}%
    \else\def\ttrightused{0}\fi
    \pgfmathtruncatemacro{\ttfree}{\ttvalence-\ttleftused-\ttrightused}%
    \ifnum\ttfree<0
      \PackageError{tt_diagram_macros}{Train coupon \tti\space exceeds 2n ports}%
        {Reduce the adjacent a-values. Every gap and the terminal fork use an extra port.}%
    \fi
    \ifnum\ttrightused>0
      \pgfmathtruncatemacro{\ttdirect}{\ttrightused-1}%
      \ifnum\ttdirect>0
        \foreach \ttt in {1,...,\ttdirect}{%
          \pgfmathtruncatemacro{\ttslot}{\ttt-1}%
          \pgfmathsetmacro{\ttangle}{60-100*\ttt/(\ttdirect+1)}%
          \TTTrainPort{\tti}{\ttslot}{\ttangle}{right-direct}%
        }%
      \fi
      \TTTrainPort{\tti}{\ttdirect}{-50}{right-triangle}%
    \fi
    \ifnum\ttleftused>0
      \pgfmathtruncatemacro{\ttdirect}{\ttleftused-1}%
      \ifnum\ttdirect>0
        \foreach \ttt in {1,...,\ttdirect}{%
          \pgfmathtruncatemacro{\ttslot}{\ttvalence-\ttt}%
          \pgfmathsetmacro{\ttangle}{120+100*\ttt/(\ttdirect+1)}%
          \TTTrainPort{\tti}{\ttslot}{\ttangle}{left-direct}%
          \ifnum\tti=1
            \pgfmathsetmacro{\ttex}{-1.25-.18*\ttcaps-.18*(\ttdirect-\ttt)}%
            \draw[\ttstyle] (tp\tti-\ttslot)
              .. controls ($(tp\tti-\ttslot)+(-.45,0)$) and (\ttex,{\ttbottom+.45})
              .. (\ttex,\ttbottom);
          \fi
        }%
      \fi
      \pgfmathtruncatemacro{\ttslot}{\ttvalence-\ttleftused}%
      \TTTrainPort{\tti}{\ttslot}{230}{left-triangle-or-fork}%
    \fi
    \ifnum\ttfree>0
      \foreach \ttt in {1,...,\ttfree}{%
        \pgfmathtruncatemacro{\ttslot}{\ttrightused+\ttt-1}%
        \pgfmathsetmacro{\ttangle}{-50-80*\ttt/(\ttfree+1)}%
        \TTTrainPort{\tti}{\ttslot}{\ttangle}{bottom}%
        \pgfmathsetmacro{\ttex}{.34-.68*\ttt/(\ttfree+1)}%
        \draw[\ttstyle] (tp\tti-\ttslot)
          .. controls ($(tp\tti-\ttslot)+(0,-.25)$)
          and ($(q\tti)+(\ttex,{\ttbottom+.35})$)
          .. ($(q\tti)+(\ttex,\ttbottom)$);
      }%
    \fi
  }%
  \foreach \tti in {1,...,\Nm}{%
    \pgfmathtruncatemacro{\ttnext}{\tti+1}%
    \pgfmathtruncatemacro{\ttdirect}{\csname tta@\tti\endcsname}%
    \ifnum\ttdirect>0
      \foreach \ttt in {1,...,\ttdirect}{%
        \pgfmathtruncatemacro{\ttslot}{\ttt-1}%
        \pgfmathtruncatemacro{\ttother}{\ttvalence-\ttt}%
        \TTTrainColour{\tti}{\ttslot}%
        \TTTrainDrawDirect{\tti}{\ttslot}{\ttnext}{\ttother}%
      }%
    \fi
    \pgfmathtruncatemacro{\ttother}{\ttvalence-1-\ttdirect}%
    \coordinate(tg\tti) at ($(q\tti)+(.775,-.78)$);
    \TTTrainColour{\tti}{\ttdirect}%
    \draw[\ttstyle] (tp\tti-\ttdirect)--(tg\tti)--(tp\ttnext-\ttother);
    \draw[\ttstyle] (tg\tti)--($(q\tti)+(.775,\ttbottom)$);
    \fill[\ttstyle] (tg\tti) circle(.045cm);
  }%
  \ifnum#5=1
    \pgfmathtruncatemacro{\ttslot}{\ttvalence-1-\csname tta@0\endcsname}%
    \TTTrainColour{1}{\ttslot}%
    \edef\ttforkcolour{\ttcolour}%
    \coordinate(tf) at (\ttforkx,-.72);
    \draw[\ttstyle] (tp1-\ttslot)--(tf);
    \draw[\ttstyle] ({\ttforkx-\ttforkhalf},\ttbottom)--(tf)
      --({\ttforkx+\ttforkhalf},\ttbottom);
    \fill[\ttstyle] (tf) circle(.045cm);
    \ifnum\ttcaps>0
      \foreach \ttt in {1,...,\ttcaps}{%
        \pgfmathsetmacro{\ttradius}{.10+.09*(\ttcaps-\ttt+1)}%
        \pgfmathtruncatemacro{\ttcolour}{mod(\ttforkcolour+\ttt,2)}%
        \ifnum\ttcolour=0\def\ttstyle{ttrho}\else\def\ttstyle{ttmu}\fi
        \draw[\ttstyle] ({\ttforkx-\ttradius},\ttbottom) arc(180:0:\ttradius cm);
      }%
    \fi
  \fi
  \foreach \tti in {1,...,\N}{%
    \pgfmathtruncatemacro{\ttphase}{mod(\ttstartphase+\tti-1,2)}%
    \TTTrainCoupon{\tti}{\ttphase}%
  }%
}
\newcommand{\TTLeftTrainCore}[3][0]{\TTTrainEngine{0}{#2}{#3}{#1}{1}}
\newcommand{\TTRightTrainCore}[3][0]{%
  \begin{scope}[xscale=-1]\TTTrainEngine{1}{#2}{#3}{#1}{1}\end{scope}}
\newcommand{\TTLeftTrain}[3][0]{%
  \begin{tikzpicture}[baseline=-.55cm]\TTLeftTrainCore[#1]{#2}{#3}\end{tikzpicture}}
\newcommand{\TTRightTrain}[3][0]{%
  \begin{tikzpicture}[baseline=-.55cm]\TTRightTrainCore[#1]{#2}{#3}\end{tikzpicture}}
\newcommand{\TTLeftTrainBasis}[3][0]{\TTLeftTrain[#1]{#2}{#3}}
\newcommand{\TTRightTrainBasis}[3][0]{\TTRightTrain[#1]{#2}{#3}}
\newcommand{\TTTrainPairCore}[3]{%
  \begingroup
  \def\ttU{U}\def\ttV{V}%
  \def\ttpairleft{#1}\def\ttpairright{#2}%
  \ifx\ttpairleft\ttU\def\ttleftisU{1}%
  \else\ifx\ttpairleft\ttV\def\ttleftisU{0}%
  \else\PackageError{tt_diagram_macros}{Invalid left pair label}{Use U or V.}\def\ttleftisU{0}\fi\fi
  \ifx\ttpairright\ttV\def\ttrightisV{1}%
  \else\ifx\ttpairright\ttU\def\ttrightisV{0}%
  \else\PackageError{tt_diagram_macros}{Invalid right pair label}{Use U or V.}\def\ttrightisV{0}\fi\fi
  \TTTrainInteger{#3}{pair phase}%
  \ifnum#3>1\PackageError{tt_diagram_macros}{Invalid pair phase}{Use 0 or 1.}\fi
  \ifnum#3=0
    \pgfmathtruncatemacro{\ttpairexponent}{\ttleftisU-\ttrightisV}%
    \TTTrainEngine{0}{0,1}{V}{0}{0}%
  \else
    \def\ttpairexponent{0}%
    \TTTrainEngine{0}{0,1}{U}{0}{0}%
  \fi
  \ifnum\ttpairexponent=1
    \node[anchor=east] at (-.55,-.4) {$\omega_u$};%
  \else\ifnum\ttpairexponent=-1
    \node[anchor=east] at (-.55,-.4) {$\omega_u^{-1}$};%
  \fi\fi
  \endgroup
}
\newcommand{\TTTrainPair}[3]{%
  \begin{tikzpicture}[baseline=-.5cm]\TTTrainPairCore{#1}{#2}{#3}\end{tikzpicture}}
\newcommand{\TTTrainOneCarCore}[2]{%
  \def\ttmirror{#1}\def\ttlabel{#2}\def\ttU{U}\def\ttV{V}%
  \ifx\ttlabel\ttU\def\ttstartphase{0}%
  \else\ifx\ttlabel\ttV\def\ttstartphase{1}%
  \else\PackageError{tt_diagram_macros}{Invalid one-car label}{Use U or V.}\def\ttstartphase{0}\fi\fi
  \pgfmathtruncatemacro{\tttotal}{2*\TTCurrentN}%
  \coordinate(q1) at (0,0);
  \foreach \ttt in {1,...,\tttotal}{%
    \pgfmathtruncatemacro{\ttslot}{\ttt-1}%
    \pgfmathsetmacro{\ttangle}{-30-120*\ttt/(\tttotal+1)}%
    \TTTrainPort{1}{\ttslot}{\ttangle}{one-car}%
    \pgfmathsetmacro{\ttex}{.4-.8*\ttt/(\tttotal+1)}%
    \draw[\ttstyle] (tp1-\ttslot)..controls ($(tp1-\ttslot)+(0,-.25)$)
      and (\ttex,-.7)..(\ttex,-1.05);
  }%
  \TTUcircle{#2}{(q1)}%
}
\newcommand{\TTLeftTrainOneCar}[1]{%
  \begin{tikzpicture}[baseline=-.4cm]\TTTrainOneCarCore{0}{#1}\end{tikzpicture}}
\newcommand{\TTRightTrainOneCar}[1]{%
  \begin{tikzpicture}[baseline=-.4cm,xscale=-1]\TTTrainOneCarCore{1}{#1}\end{tikzpicture}}
\newcommand{\TTOneCarLeft}[1]{\TTLeftTrainOneCar{#1}}
\newcommand{\TTOneCarRight}[1]{\TTRightTrainOneCar{#1}}
\newcommand{\TTBoundaryColourStyle}[1]{%
  \def\ttboundarycolour{#1}\def\ttr{r}\def\ttm{m}%
  \ifx\ttboundarycolour\ttr\def\ttboundarystyle{ttrho}%
  \else\ifx\ttboundarycolour\ttm\def\ttboundarystyle{ttmu}%
  \else\PackageError{tt_diagram_macros}{Invalid boundary colour}{Use r or m.}\fi\fi}
\newcommand{\TTBoundarySide}[1]{%
  \def\ttboundaryside{#1}\def\ttl{l}\def\ttr{r}%
  \ifx\ttboundaryside\ttl\def\ttboundarysign{-1}%
  \else\ifx\ttboundaryside\ttr\def\ttboundarysign{1}%
  \else\PackageError{tt_diagram_macros}{Invalid boundary side}{Use l or r.}\fi\fi}
\newcommand{\TTCouponWithSideCap}[3]{%
  \begin{tikzpicture}[baseline=-.4cm]
    \TTTrainOneCarCore{0}{#1}%
    \TTBoundarySide{#2}\TTBoundaryColourStyle{#3}%
    \draw[\ttboundarystyle] ({.75*\ttboundarysign-.2},-1.05) arc(180:0:.2cm);
  \end{tikzpicture}}
\newcommand{\TTCouponWithNestedSideCaps}[4]{%
  \begin{tikzpicture}[baseline=-.4cm]
    \TTTrainOneCarCore{0}{#1}%
    \TTBoundarySide{#2}\TTBoundaryColourStyle{#3}%
    \TTTrainInteger{#4}{nested side cap count}%
    \pgfmathtruncatemacro{\ttsidecaps}{#4}%
    \pgfmathsetmacro{\ttsidecenter}{.73+.12*\ttsidecaps}%
    \ifnum\ttsidecaps>0
      \foreach \ttcapindex in {1,...,\ttsidecaps}{%
        \pgfmathsetmacro{\ttcapradius}{.2+.12*(\ttsidecaps-\ttcapindex)}%
        \pgfmathtruncatemacro{\ttcapodd}{mod(\ttcapindex,2)}%
        \ifx\ttboundarycolour\ttr
          \ifnum\ttcapodd=1\def\ttcapstyle{ttrho}\else\def\ttcapstyle{ttmu}\fi
        \else
          \ifnum\ttcapodd=1\def\ttcapstyle{ttmu}\else\def\ttcapstyle{ttrho}\fi
        \fi
        \draw[\ttcapstyle] ({\ttsidecenter*\ttboundarysign-\ttcapradius},-1.05)
          arc(180:0:\ttcapradius cm);
      }%
    \fi
  \end{tikzpicture}}
\newcommand{\TTFloatingRoofCoupon}[2]{%
  \begin{tikzpicture}[baseline=-.4cm]
    \TTTrainOneCarCore{0}{#1}%
    \TTBoundaryColourStyle{#2}%
    \TTFloatingRoofPath{\ttboundarystyle}{-.85}{-1.05}{.85}{1.9}%
  \end{tikzpicture}}
\newsavebox{\TTFloatingRoofContentBox}
\newcommand{\TTWithFloatingRoof}[2]{%
  \begingroup
  \sbox{\TTFloatingRoofContentBox}{#2}%
  \begin{tikzpicture}[baseline=.65cm]
    \node[anchor=south west,inner sep=0pt,outer sep=0pt] at (0,0)
      {\usebox{\TTFloatingRoofContentBox}};
    \pgfmathsetmacro{\ttcontextwidth}{\wd\TTFloatingRoofContentBox/1cm}%
    \pgfmathsetmacro{\ttcontextheight}{(\ht\TTFloatingRoofContentBox+\dp\TTFloatingRoofContentBox)/1cm}%
    \pgfmathsetmacro{\ttcontextouterright}{\ttcontextwidth+.25}%
    \pgfmathsetmacro{\ttcontextouterheight}{\ttcontextheight+.4}%
    \TTBoundaryColourStyle{#1}%
    \TTFloatingRoofPath{\ttboundarystyle}{-.25}{.01406}{\ttcontextouterright}{\ttcontextouterheight}%
  \end{tikzpicture}
  \endgroup}
\newcommand{\TTSingleRoofPair}[1]{%
  \begin{tikzpicture}[baseline=-.5cm]
    \pgfmathtruncatemacro{\ttroofgap}{\TTCurrentN-1}%
    \TTTrainEngine{0}{0,\ttroofgap}{#1}{0}{0}%
  \end{tikzpicture}}
\newcommand{\TTSingleRoofTrain}[3][0]{%
  \begin{tikzpicture}[baseline=-.55cm]
    \TTBoundarySide{#3}%
    \pgfmathtruncatemacro{\ttroofgap}{\TTCurrentN-1}%
    \ifnum\ttboundarysign=-1
      \TTTrainEngine{0}{0,\ttroofgap}{#2}{#1}{1}%
    \else
      \begin{scope}[xscale=-1]\TTTrainEngine{1}{0,\ttroofgap}{#2}{#1}{1}\end{scope}%
    \fi
  \end{tikzpicture}}
\newcommand{\TTRedRoofRightTrainSkein}[1]{%
  \begingroup
  \ifnum\TTCurrentN<2
    \PackageError{tt_diagram_macros}{Red roof skein needs a direct red edge}{Use n at least 2.}%
  \fi
  \let\TTUnresolvedDirectEdge\TTTrainDrawDirect
  \renewcommand{\TTTrainDrawDirect}[4]{%
    \ifnum##1=1
      \ifnum##2=0\else\TTUnresolvedDirectEdge{##1}{##2}{##3}{##4}\fi
    \else\TTUnresolvedDirectEdge{##1}{##2}{##3}{##4}\fi}%
  \begin{tikzpicture}[baseline=-.55cm,xscale=-1]
    \pgfmathtruncatemacro{\ttroofgap}{\TTCurrentN-1}%
    \pgfmathtruncatemacro{\ttredother}{2*\TTCurrentN-1}%
    \TTTrainEngine{1}{0,\ttroofgap}{U}{0}{1}%
    \ifx#1c
      \draw[ttrho] (tp1-0)
        ..controls(.72,.35)and(.53,.95)..(0,.95)
        ..controls(-.7,.95)and(-1.2,.5)..(-1.2,0)--(-1.2,-1.35);
      \draw[ttrho] (tp2-\ttredother)
        ..controls(.83,.35)and(1.02,.95)..(1.55,.95)
        ..controls(2.25,.95)and(2.75,.5)..(2.75,0)--(2.75,-1.35);
    \else
      \pgfmathsetmacro{\ttroofleftangle}{acos(-.775/1.975)}%
      \pgfmathsetmacro{\ttroofrightangle}{acos(.775/1.975)}%
      \pgfmathsetmacro{\ttroofvertexy}{.15+1.3*sin(\ttroofleftangle)}%
      \draw[ttrho] (-1.2,-1.35)--(-1.2,.15)
        arc[start angle=180,end angle=\ttroofleftangle,x radius=1.975cm,y radius=1.3cm]
        arc[start angle=\ttroofleftangle,end angle=\ttroofrightangle,x radius=1.975cm,y radius=1.3cm]
        arc[start angle=\ttroofrightangle,end angle=0,x radius=1.975cm,y radius=1.3cm]
        --(2.75,-1.35);
      \draw[ttrho] (tp1-0)..controls(.7,.5)and(.3,1.08)..(0,\ttroofvertexy);
      \draw[ttrho] (tp2-\ttredother)..controls(.85,.5)and(1.25,1.08)..(1.55,\ttroofvertexy);
      \fill[rho] (0,\ttroofvertexy)circle(.045cm);
      \fill[rho] (1.55,\ttroofvertexy)circle(.045cm);
    \fi
  \end{tikzpicture}
  \endgroup}
\newcommand{\TTRedRoofRightTrainCapCup}{\TTRedRoofRightTrainSkein{c}}
\newcommand{\TTRedRoofRightTrainTrivalent}{\TTRedRoofRightTrainSkein{t}}
\newcommand{\TTBlackRoofLeftTrainSkein}[1]{%
  \begingroup
  \ifnum\TTCurrentN<2
    \PackageError{tt_diagram_macros}{Black roof skein needs a direct black edge}{Use n at least 2.}%
  \fi
  \let\TTUnresolvedDirectEdge\TTTrainDrawDirect
  \renewcommand{\TTTrainDrawDirect}[4]{%
    \ifnum##1=1
      \ifnum##2=0\else\TTUnresolvedDirectEdge{##1}{##2}{##3}{##4}\fi
    \else\TTUnresolvedDirectEdge{##1}{##2}{##3}{##4}\fi}%
  \begin{tikzpicture}[baseline=-.55cm]
    \pgfmathtruncatemacro{\ttroofgap}{\TTCurrentN-1}%
    \pgfmathtruncatemacro{\ttblackother}{2*\TTCurrentN-1}%
    \TTTrainEngine{0}{0,\ttroofgap}{U}{0}{1}%
    \ifx#1c
      \draw[ttmu] (tp1-0)
        ..controls(.72,.35)and(.53,.95)..(0,.95)
        ..controls(-.7,.95)and(-1.2,.5)..(-1.2,0)--(-1.2,-1.35);
      \draw[ttmu] (tp2-\ttblackother)
        ..controls(.83,.35)and(1.02,.95)..(1.55,.95)
        ..controls(2.25,.95)and(2.75,.5)..(2.75,0)--(2.75,-1.35);
    \else
      \pgfmathsetmacro{\ttroofleftangle}{acos(-.775/1.975)}%
      \pgfmathsetmacro{\ttroofrightangle}{acos(.775/1.975)}%
      \pgfmathsetmacro{\ttroofvertexy}{.15+1.3*sin(\ttroofleftangle)}%
      \draw[ttmu] (-1.2,-1.35)--(-1.2,.15)
        arc[start angle=180,end angle=\ttroofleftangle,x radius=1.975cm,y radius=1.3cm]
        arc[start angle=\ttroofleftangle,end angle=\ttroofrightangle,x radius=1.975cm,y radius=1.3cm]
        arc[start angle=\ttroofrightangle,end angle=0,x radius=1.975cm,y radius=1.3cm]
        --(2.75,-1.35);
      \draw[ttmu] (tp1-0)..controls(.7,.5)and(.3,1.08)..(0,\ttroofvertexy);
      \draw[ttmu] (tp2-\ttblackother)..controls(.85,.5)and(1.25,1.08)..(1.55,\ttroofvertexy);
      \fill[black] (0,\ttroofvertexy)circle(.045cm);
      \fill[black] (1.55,\ttroofvertexy)circle(.045cm);
    \fi
  \end{tikzpicture}
  \endgroup}
\newcommand{\TTBlackRoofLeftTrainCapCup}{\TTBlackRoofLeftTrainSkein{c}}
\newcommand{\TTBlackRoofLeftTrainTrivalent}{\TTBlackRoofLeftTrainSkein{t}}
\newcommand{\TTForkedRoofCoupon}[3]{%
  \begin{tikzpicture}[baseline=-.45cm]
    \coordinate(q) at (0,0);
    \TTBoundarySide{#3}\TTBoundaryColourStyle{#2}%
    \pgfmathtruncatemacro{\ttrooflower}{2*\TTCurrentN-1}%
    \foreach \tti in {1,...,\ttrooflower}{%
      \pgfmathsetmacro{\ttangle}{-150+120*\tti/(\ttrooflower+1)}%
      \pgfmathsetmacro{\ttex}{-.4+.8*\tti/(\ttrooflower+1)}%
      \pgfmathtruncatemacro{\ttodd}{mod(\tti,2)}%
      \ifx\ttboundarycolour\ttr
        \ifnum\ttodd=1\def\ttstyle{ttmu}\else\def\ttstyle{ttrho}\fi
      \else
        \ifnum\ttodd=1\def\ttstyle{ttrho}\else\def\ttstyle{ttmu}\fi
      \fi
      \draw[\ttstyle] (\ttangle:.4cm)
        .. controls ($(q)+(\ttangle:.4cm)+(0,-.25)$) and (\ttex,-.7)
        .. (\ttex,-1.05);
    }%
    \draw[\ttboundarystyle] (-.85,{-.875-.175*\ttboundarysign})
      --(-.85,0) arc(180:0:.85cm)--(.85,{-.875+.175*\ttboundarysign});
    \draw[\ttboundarystyle] (0,.85)--(q);
    \fill[\ttboundarystyle] (0,.85) circle(.05cm);
    \coordinate(ttrooffork) at ({.85*\ttboundarysign},-.7);
    \draw[\ttboundarystyle] ({.85*\ttboundarysign-.17},-1.05)
      --(ttrooffork)--({.85*\ttboundarysign+.17},-1.05);
    \fill[\ttboundarystyle] (ttrooffork) circle(.045cm);
    \def\ttlabel{#1}\def\ttU{U}\def\ttangle{180}%
    \ifx\ttlabel\ttU
      \ifx\ttboundarycolour\ttr\def\ttangle{45}\fi
    \else\ifx\ttboundarycolour\ttm\def\ttangle{45}\fi\fi
    \TTUcircleStarAngle{#1}{(q)}{\ttangle}%
  \end{tikzpicture}}
\newcommand{\TTTrainRoof}[1]{%
  \pgfmathsetmacro{\xr}{1.55*(\N-1)+.5}%
  \ifx#1r\TTFloatingRoofRho{-1.35}{.55}{\xr}{1.2}%
  \else\ifx#1m\TTFloatingRoofMu{-1.35}{.55}{\xr}{1.2}%
  \else
    \pgfmathsetmacro{\midr}{1.55*(\N-1)/2}%
    \pgfmathsetmacro{\ttroofleftend}{\midr-.08}%
    \pgfmathsetmacro{\ttroofrightstart}{\midr+.08}%
    \TTFloatingRoofRho{-1.35}{.55}{\ttroofleftend}{1.0}%
    \TTFloatingRoofMu{\ttroofrightstart}{.55}{\xr}{1.0}%
  \fi\fi
}
\newcommand{\TTLeftTrainRoofedRho}[3][0]{%
  \begin{tikzpicture}[baseline=-.55cm]\TTLeftTrainCore[#1]{#2}{#3}\TTTrainRoof{r}\end{tikzpicture}}
\newcommand{\TTLeftTrainRoofedMu}[3][0]{%
  \begin{tikzpicture}[baseline=-.55cm]\TTLeftTrainCore[#1]{#2}{#3}\TTTrainRoof{m}\end{tikzpicture}}
\newcommand{\TTLeftTrainDoubleRoof}[3][0]{%
  \begin{tikzpicture}[baseline=-.55cm]\TTLeftTrainCore[#1]{#2}{#3}\TTTrainRoof{d}\end{tikzpicture}}
\newcommand{\TTRightTrainRoofedRho}[3][0]{%
  \begin{tikzpicture}[baseline=-.55cm,xscale=-1]\TTTrainEngine{1}{#2}{#3}{#1}{1}\TTTrainRoof{r}\end{tikzpicture}}
\newcommand{\TTRightTrainRoofedMu}[3][0]{%
  \begin{tikzpicture}[baseline=-.55cm,xscale=-1]\TTTrainEngine{1}{#2}{#3}{#1}{1}\TTTrainRoof{m}\end{tikzpicture}}
\newcommand{\TTRightTrainDoubleRoof}[3][0]{%
  \begin{tikzpicture}[baseline=-.55cm,xscale=-1]\TTTrainEngine{1}{#2}{#3}{#1}{1}\TTTrainRoof{d}\end{tikzpicture}}

\newcommand{\TTMergedTrain}[2]{%
\begin{tikzpicture}[baseline=-.5cm]
  \pgfmathtruncatemacro{\N}{#1}
  \foreach \i in {1,...,\N}{%
    \pgfmathsetmacro{\x}{1.35*(\i-1)}
    \coordinate (q\i) at (\x,0);
  }
  \pgfmathtruncatemacro{\Nm}{\N-1}
  \ifnum\Nm>0
    \foreach \i in {1,...,\Nm}{\pgfmathtruncatemacro{\j}{\i+1}\TTHorizontalPair{(q\i)}{(q\j)}}
  \fi
  \pgfmathsetmacro{\mid}{.675*(\N-1)}
  \foreach \i in {1,...,\N}{\draw[ttmu] (q\i) -- (\mid,-1.15);}
  \draw[ttmu] (\mid,-1.15)--(\mid,-1.7);
  \foreach \i in {1,...,\N}{%
    \pgfmathtruncatemacro{\ttdegree}{2}%
    \ifnum\N=1
      \pgfmathtruncatemacro{\ttdegree}{0}%
    \else
      \ifnum\i=1 \pgfmathtruncatemacro{\ttdegree}{1}\fi
      \ifnum\i=\N \pgfmathtruncatemacro{\ttdegree}{1}\fi
    \fi
    \pgfmathtruncatemacro{\pp}{mod(\i+#2,2)}%
    \ifnum\pp=0
      \TTBottomRemainderA{(q\i)}{-.72}{\ttdegree}{\ttdegree+1}%
    \else
      \TTBottomRemainderB{(q\i)}{-.72}{\ttdegree}{\ttdegree+1}%
    \fi
  }
  \foreach \i in {1,...,\N}{%
    \pgfmathtruncatemacro{\pp}{mod(\i+#2,2)}
    \ifnum\pp=0 \TTUcircle{U}{(q\i)}\else\TTUcircle{V}{(q\i)}\fi
  }
\end{tikzpicture}}

\newcommand{\TTTwoCar}[2]{%
  \ifnum#1=0
    \ifnum#2=0\TTTrainPair{U}{U}{0}\else\TTTrainPair{U}{V}{0}\fi
  \else
    \ifnum#2=0\TTTrainPair{V}{U}{0}\else\TTTrainPair{V}{V}{0}\fi
  \fi}

\newcommand{\TTConjugationCheckN}{%
  \ifnum\TTCurrentN=2\else
    \PackageError{tt_diagram_macros}{This conjugation diagram requires n=2}{Use TTSetN with 2.}%
  \fi}
\newcommand{\TTConjugationPorts}[3]{%
  \coordinate(cq#1) at (#2,0);
  \foreach \tta [count=\tts from 0] in {#3}{%
    \coordinate(cp#1-\tts) at ($(cq#1)+(\tta:.4cm)$);}}

\newcommand{\TTConjugatedRedBlackRoof}{%
  \TTConjugationCheckN
  \begin{tikzpicture}[baseline=-.75cm]
    \coordinate(cq0) at (-1.20,-.80);
    \coordinate(cq1) at (0,0);
    \coordinate(cq2) at (1.20,-.80);
    \TTFloatingRoofPath{ttmu}{-1.30}{-.80}{1.30}{2.30}
    \TTFloatingRoofPath{ttrho}{-1.10}{-.80}{1.10}{1.70}
    \draw[ttrho] (-1.30,-.80)--(-1.30,-1.70);
    \draw[ttmu]  (-1.10,-.80)--(-1.10,-1.70);
    \draw[ttrho] (-.30,0)--(-.30,-1.70);
    \draw[ttmu]  (-.10,0)--(-.10,-1.70);
    \draw[ttrho] ( .10,0)--( .10,-1.70);
    \draw[ttmu]  ( .30,0)--( .30,-1.70);
    \draw[ttmu]  (1.10,-.80)--(1.10,-1.70);
    \draw[ttrho] (1.30,-.80)--(1.30,-1.70);
    \TTUcircleLeftStar{U}{(cq0)}
    \TTUcircle{U}{(cq1)}
    \TTUcircleRightStar{V}{(cq2)}
  \end{tikzpicture}}

\newcommand{\TTConjugatedLeftCaps}{%
  \TTConjugationCheckN
  \begin{tikzpicture}[baseline=-.55cm]
    \TTConjugationPorts{0}{0}{-40,-75,-105,-140}
    \TTConjugationPorts{1}{1.55}{45,-45,-105,-140}
    \TTConjugationPorts{2}{3.1}{-40,-75,-135,135}
    \draw[ttrho] (cp0-3)--(-0.36,-1.55);
    \draw[ttmu] (cp0-2)--(-0.12,-1.55);
    \draw[ttrho] (cp0-1)--(0.12,-1.55);
    \draw[ttmu] (cp0-0)--(0.36,-1.55);
    \draw[ttrho] (cp1-3)--(1.2,-1.55);
    \draw[ttmu] (cp1-2)--(1.55,-1.55);
    \draw[ttmu] (cp2-1)--(3.1,-1.55);
    \draw[ttrho] (cp2-0)--(3.45,-1.55);
    \draw[ttmu] (cp1-0)--(cp2-3);
    \draw[ttrho] (cp1-1)--(cp2-2);
    \TTUcircle{U}{(cq0)}
    \TTUcircle{U}{(cq1)}
    \TTUcircle{V}{(cq2)}
  \end{tikzpicture}}

\newcommand{\TTConjugatedLeftForkCap}{%
  \TTConjugationCheckN
  \begin{tikzpicture}[baseline=-.55cm]
    \TTConjugationPorts{0}{0}{-40,-75,-105,-140}
    \TTConjugationPorts{1}{1.55}{45,-45,-105,140}
    \TTConjugationPorts{2}{3.1}{45,-45,-135,135}
    \TTConjugationPorts{3}{4.65}{-40,-75,-135,135}
    \draw[ttrho] (cp0-3)--(-0.36,-1.55);
    \draw[ttmu] (cp0-2)--(-0.12,-1.55);
    \draw[ttrho] (cp0-1)--(0.12,-1.55);
    \draw[ttrho] (cp1-2)--(1.55,-1.55);
    \draw[ttmu] (cp3-1)--(4.65,-1.55);
    \draw[ttrho] (cp3-0)--(5,-1.55);
    \draw[ttrho] (cp1-0)--(cp2-3);
    \coordinate(cv0-1) at (0.775,-.85);
    \draw[ttmu] (cp0-0)--(cv0-1);
    \draw[ttmu] (cp1-3)--(cv0-1);
    \draw[ttmu] (cv0-1)--(0.775,-1.55);
    \fill[mu] (cv0-1) circle(.045cm);
    \coordinate(cv1-2) at (2.325,-.85);
    \draw[ttmu] (cp1-1)--(cv1-2);
    \draw[ttmu] (cp2-2)--(cv1-2);
    \draw[ttmu] (cv1-2)--(2.325,-1.55);
    \fill[mu] (cv1-2) circle(.045cm);
    \draw[ttmu] (cp2-0)--(cp3-3);
    \draw[ttrho] (cp2-1)--(cp3-2);
    \TTUcircle{U}{(cq0)}
    \TTUcircle{V}{(cq1)}
    \TTUcircle{U}{(cq2)}
    \TTUcircle{V}{(cq3)}
  \end{tikzpicture}}

\newcommand{\TTConjugatedLeftShortTrain}{%
  \TTConjugationCheckN
  \begin{tikzpicture}[baseline=-.55cm]
    \TTConjugationPorts{0}{0}{45,-45,-105,-140}
    \TTConjugationPorts{1}{1.55}{-40,-105,-140,135}
    \TTConjugationPorts{2}{3.1}{45,-45,-105,140}
    \TTConjugationPorts{3}{4.65}{-40,-75,-135,135}
    \draw[ttrho] (cp0-3)--(-0.36,-1.55);
    \draw[ttmu] (cp0-2)--(-0.12,-1.55);
    \draw[ttmu] (cp1-1)--(1.55,-1.55);
    \draw[ttmu] (cp2-2)--(3.1,-1.55);
    \draw[ttmu] (cp3-1)--(4.65,-1.55);
    \draw[ttrho] (cp3-0)--(5,-1.55);
    \draw[ttmu] (cp0-0)--(cp1-3);
    \coordinate(cv0-1) at (0.775,-.85);
    \draw[ttrho] (cp0-1)--(cv0-1);
    \draw[ttrho] (cp1-2)--(cv0-1);
    \draw[ttrho] (cv0-1)--(0.775,-1.55);
    \fill[rho] (cv0-1) circle(.045cm);
    \coordinate(cv1-2) at (2.325,-.85);
    \draw[ttrho] (cp1-0)--(cv1-2);
    \draw[ttrho] (cp2-3)--(cv1-2);
    \draw[ttrho] (cv1-2)--(2.325,-1.55);
    \fill[rho] (cv1-2) circle(.045cm);
    \draw[ttmu] (cp2-0)--(cp3-3);
    \draw[ttrho] (cp2-1)--(cp3-2);
    \TTUcircle{U}{(cq0)}
    \TTUcircle{V}{(cq1)}
    \TTUcircle{U}{(cq2)}
    \TTUcircle{V}{(cq3)}
  \end{tikzpicture}}

\newcommand{\TTConjugatedLeftLongTrain}{%
  \TTConjugationCheckN
  \begin{tikzpicture}[baseline=-.55cm]
    \TTConjugationPorts{0}{0}{45,-45,-105,-140}
    \TTConjugationPorts{1}{1.55}{45,-45,-135,135}
    \TTConjugationPorts{2}{3.1}{45,-45,-135,135}
    \TTConjugationPorts{3}{4.65}{45,-45,-135,135}
    \TTConjugationPorts{4}{6.2}{45,-45,-135,135}
    \TTConjugationPorts{5}{7.75}{-40,-75,-135,135}
    \draw[ttrho] (cp0-3)--(-0.36,-1.55);
    \draw[ttmu] (cp0-2)--(-0.12,-1.55);
    \draw[ttmu] (cp5-1)--(7.75,-1.55);
    \draw[ttrho] (cp5-0)--(8.1,-1.55);
    \draw[ttmu] (cp0-0)--(cp1-3);
    \coordinate(cv0-1) at (0.775,-.85);
    \draw[ttrho] (cp0-1)--(cv0-1);
    \draw[ttrho] (cp1-2)--(cv0-1);
    \draw[ttrho] (cv0-1)--(0.775,-1.55);
    \fill[rho] (cv0-1) circle(.045cm);
    \draw[ttrho] (cp1-0)--(cp2-3);
    \coordinate(cv1-2) at (2.325,-.85);
    \draw[ttmu] (cp1-1)--(cv1-2);
    \draw[ttmu] (cp2-2)--(cv1-2);
    \draw[ttmu] (cv1-2)--(2.325,-1.55);
    \fill[mu] (cv1-2) circle(.045cm);
    \draw[ttmu] (cp2-0)--(cp3-3);
    \coordinate(cv2-3) at (3.875,-.85);
    \draw[ttrho] (cp2-1)--(cv2-3);
    \draw[ttrho] (cp3-2)--(cv2-3);
    \draw[ttrho] (cv2-3)--(3.875,-1.55);
    \fill[rho] (cv2-3) circle(.045cm);
    \draw[ttrho] (cp3-0)--(cp4-3);
    \coordinate(cv3-4) at (5.425,-.85);
    \draw[ttmu] (cp3-1)--(cv3-4);
    \draw[ttmu] (cp4-2)--(cv3-4);
    \draw[ttmu] (cv3-4)--(5.425,-1.55);
    \fill[mu] (cv3-4) circle(.045cm);
    \draw[ttmu] (cp4-0)--(cp5-3);
    \draw[ttrho] (cp4-1)--(cp5-2);
    \TTUcircle{U}{(cq0)}
    \TTUcircle{V}{(cq1)}
    \TTUcircle{U}{(cq2)}
    \TTUcircle{V}{(cq3)}
    \TTUcircle{U}{(cq4)}
    \TTUcircle{V}{(cq5)}
  \end{tikzpicture}}

\newcommand{\TTNthreeConjugatedRoof}{%
  \ifnum\TTCurrentN=3\else
    \PackageError{tt_diagram_macros}{This conjugation requires n=3}{Use TTSetN with 3.}%
  \fi
  \begin{tikzpicture}[baseline=-.80cm]
    \coordinate(cq0) at (-1.55,-.90);
    \coordinate(cq1) at (0,0);
    \coordinate(cq2) at (1.55,-.90);
    \TTFloatingRoofPath{ttmu}{-1.75}{-.90}{1.75}{3.00}
    \TTFloatingRoofPath{ttrho}{-1.55}{-.90}{1.55}{2.40}
    \TTFloatingRoofPath{ttmu}{-1.35}{-.90}{1.35}{1.80}
    \draw[ttrho] (-1.75,-.90)--(-1.75,-1.75);
    \draw[ttmu]  (-1.55,-.90)--(-1.55,-1.75);
    \draw[ttrho] (-1.35,-.90)--(-1.35,-1.75);
    \draw[ttrho] (-.36,-.12)--(-.50,-.65)--(-.50,-1.75);
    \draw[ttmu]  (-.30,0)--(-.30,-1.75);
    \draw[ttrho] (-.10,0)--(-.10,-1.75);
    \draw[ttmu]  ( .10,0)--( .10,-1.75);
    \draw[ttrho] ( .30,0)--( .30,-1.75);
    \draw[ttmu]  ( .36,-.12)--( .50,-.65)--( .50,-1.75);
    \draw[ttrho] (1.35,-.90)--(1.35,-1.75);
    \draw[ttmu]  (1.55,-.90)--(1.55,-1.75);
    \draw[ttrho] (1.75,-.90)--(1.75,-1.75);
    \TTUcircleLeftStar{U}{(cq0)}
    \TTUcircle{U}{(cq1)}
    \TTUcircleRightStar{V}{(cq2)}
  \end{tikzpicture}}

\newsavebox{\TTFitDisplayBox}
\newcommand{\fitdisplay}[1]{%
  \begingroup
  \sbox{\TTFitDisplayBox}{$\displaystyle #1$}%
  \ifdim\wd\TTFitDisplayBox>.98\textwidth
    \resizebox{.98\textwidth}{!}{\usebox{\TTFitDisplayBox}}%
  \else\usebox{\TTFitDisplayBox}\fi
  \endgroup}

\section{Graphical proof of ordered two-Fibonacci uniqueness}
\label{app:two-fibonacci-uniqueness}

We prove \cref{thm:marked-uniqueness} and record the algebraic presentation
used in \cref{sec:characteristic-zero-fields}.

Fix the standard Fibonacci fusion category $\Fib$ over $\CC$.  Let $\rho$
and $\mu$ be the nontrivial simple images of two specified tensor embeddings
of $\Fib$ into a fusion category $\calC$, and assume that they tensor-generate
$\calC$.  Write $u_k$ and $v_k$ for the alternating words of length $k$
beginning with $\rho$ and $\mu$, respectively.  If $n$ is the least positive
integer for which $u_n\cong v_n$, then \cref{prop:standard-form} gives
\begin{equation}
 \Irr(\calC)=
 \{\id,u_1,\ldots,u_{n-1},v_1,\ldots,v_{n-1},c_n\},
 \qquad c_n:=u_n\cong v_n.
 \label{eq:standard-form}
\end{equation}

\subsection{Spherical normalization and inverse coupons}

\begin{lemma}[Automatic pseudo-unitarity]
\label{lem:pseudo-unitary}
The category $\calC$ is pseudo-unitary.  Its canonical spherical structure
restricts to the standard positive spherical structure on each distinguished
Fibonacci subcategory.
\end{lemma}

\begin{proof}
The squared norm of the nontrivial simple object of $\Fib$ is $\ph^2$.
Squared norms are multiplicative, while every alternating word in
\eqref{eq:standard-form} is simple.  Hence
\[
 |u_k|^2=|v_k|^2=\ph^{2k},\qquad |c_n|^2=\ph^{2n}.
\]
The same powers are the squares of the Frobenius--Perron dimensions, so
\[
 \Dim(\calC)=1+2\sum_{k=1}^{n-1}\ph^{2k}+\ph^{2n}
             =\FPdim(\calC).
\]
The canonical spherical structure of a pseudo-unitary fusion category
\cite[Proposition~8.23]{ENO2005} gives every simple object its
Frobenius--Perron dimension.  Its restrictions to the two marked copies of
$\Fib$ are therefore the standard spherical structures.
\end{proof}

All diagrams in this appendix are read in this spherical structure.  Red and
black strands denote $\rho$ and $\mu$, respectively.  For
$\kappa\in\{\rho,\mu\}$ choose the cyclic Fibonacci trivalent vertex so that
a circle evaluates to $\ph$, a tadpole to zero, a bigon to the identity,
and
\begin{equation}
\begin{tikzpicture}[baseline=-.30cm]
 \draw[thick,rho] (-.16,.65)--(-.16,-.65);
 \draw[thick,rho] (.16,.65)--(.16,-.65);
\end{tikzpicture}
=\frac1\ph
\underbrace{\begin{tikzpicture}[baseline=-.30cm]
 \draw[thick,rho](-.24,-.38)arc(180:0:.24cm);
 \draw[thick,rho](-.24,.38)arc(-180:0:.24cm);
\end{tikzpicture}}_{\mathsf C}
+
\underbrace{\begin{tikzpicture}[baseline=-.30cm]
 \coordinate(t)at(0,.25);\coordinate(b)at(0,-.25);
 \filldraw[rho](t)circle(.045cm);\filldraw[rho](b)circle(.045cm);
 \draw[thick,rho](t)--(-.34,.65);\draw[thick,rho](t)--(.34,.65);
 \draw[thick,rho](t)--(b);
 \draw[thick,rho](b)--(-.34,-.65);\draw[thick,rho](b)--(.34,-.65);
\end{tikzpicture}}_{\mathsf T}.
\label{eq:fibonacci-resolution-graphical}
\end{equation}
The labels $\mathsf C$ and $\mathsf T$ denote the cap--cup and trivalent
channels, respectively.  The identical relations hold in black.  This is the
Fibonacci gauge of \cite[Section~5]{IzumiMorrisonPenneys2013}.  Each cyclic
trivalent vertex is determined up to sign, and these sign choices preserve
\eqref{eq:fibonacci-resolution-graphical}.

Choose inverse isomorphisms
\[
 U:u_n\longrightarrow v_n,
 \qquad V:=U^{-1}:v_n\longrightarrow u_n.
\]
A star records the distinguished boundary interval.  Stacking $V$ above
$U$, or $U$ above $V$, gives the identity on the corresponding alternating
word.

For every displayed basis diagram $T$, let $T^\vee$ be its literal vertical
reflection, with the star reflected together with its coupon and the labels
$U,V$ interchanged.  No cyclic rotation is applied.  On a diagrammatic span,
extend $T\mapsto T^\vee$ linearly.

Let $\mathcal R_+$ and $\mathcal R_-$ be clockwise and counterclockwise
one-click rotations of a starred coupon, with the exterior strings bent by
the spherical duality maps.

\begin{lemma}[Inverse normalization and rotation]
\label{lem:inverse-rotation}
The inverse pair can be chosen so that
\begin{equation}
 \mathcal R_-(U)=V,
 \qquad \mathcal R_+(U)=\omega_UV,
 \qquad \omega_U^n=1.
 \label{eq:rotation-normalization}
\end{equation}
Consequently
\begin{equation}
 \mathcal R_+(V)=U,
 \qquad \mathcal R_-(V)=\omega_U^{-1}U.
 \label{eq:inverse-rotation-relations}
\end{equation}
The residual change preserving the first equality in
\eqref{eq:rotation-normalization} is $(U,V)\mapsto(-U,-V)$.
\end{lemma}

\begin{proof}
For any inverse pair $U_0,V_0$, one-dimensionality gives
$\mathcal R_-(U_0)=cV_0$ with $c\ne0$.  Choose $t^2c=1$ and replace the pair
by $(tU_0,t^{-1}V_0)$.  This gives the first equality and defines the
nonzero scalar $\omega_U$ in the second.  A full rotation consists of $2n$
one-click rotations.  The coupon labels alternate between $U$ and $V$, and
each consecutive pair contributes $\omega_U$; hence a full rotation gives
$\omega_U^n=1$.  Rotating the inverse relations gives
\eqref{eq:inverse-rotation-relations}.  A further rescaling preserves
$\mathcal R_-(U)=V$ exactly when its square is one.
\end{proof}

\subsection{The complete local jellyfish relations}

The graphical strategy in this subsection is the two-resolution argument of
\cite[Lemma~3.5 and Section~5]{IzumiMorrisonPenneys2013}.  We give it with
the inverse coupon $V=U^{-1}$ and the reflection $T\mapsto T^\vee$ fixed
above.

For $n=2$, put
\[
 A_0:=\TTJellyfishInput,\qquad J_0:=\TTThreeCouponJellyfish,
\]
\[
 A_1:=\TTSwappedJellyfishInput,\qquad
 J_1:=\TTSwappedThreeCouponJellyfish.
\]
Thus $A_0,J_0\in\Hom(\rho\mu,\rho\mu^2)$ and
$A_1,J_1\in\Hom(\mu\rho,\mu\rho^2)$.  The fusion rules give
\begin{align*}
 \Hom(\rho\mu,\rho\mu^2)
 &\cong\Hom(\rho\mu,\rho)\oplus\End(\rho\mu)\cong\CC,\\
 \Hom(\mu\rho,\mu\rho^2)
 &\cong\Hom(\mu\rho,\mu)\oplus\End(\mu\rho)\cong\CC.
\end{align*}

For $n=3$, define the typed diagrams
\[
 A^{\mathrm{odd}}_0:=\TTOddJellyfishInputRho,
 \qquad J^{\mathrm{odd}}_0:=\TTOddThreeCouponJellyfishRho,
\]
\[
 A^{\mathrm{odd}}_1:=\TTOddJellyfishInputMu,
 \qquad J^{\mathrm{odd}}_1:=\TTOddThreeCouponJellyfishMu.
\]
Here
\begin{align*}
 A^{\mathrm{odd}}_0,J^{\mathrm{odd}}_0
  &\in\Hom(u_3,u_3\rho),&
 \Hom(u_3,u_3\rho)&\cong\Hom(u_3,u_2)\oplus\End(u_3)\cong\CC,\\
 A^{\mathrm{odd}}_1,J^{\mathrm{odd}}_1
  &\in\Hom(v_3,v_3\mu),&
 \Hom(v_3,v_3\mu)&\cong\Hom(v_3,v_2)\oplus\End(v_3)\cong\CC.
\end{align*}
Each odd coupon has three red and three black incident strands, displayed
separately in these diagrams.

Closing $J_i$ against $J_i^\vee$ cancels the three inverse pairs and all
bigons, leaving a disjoint union of Fibonacci circles.  The same calculation
applies to $J_i^{\mathrm{odd}}(J_i^{\mathrm{odd}})^\vee$.  These closures are
nonzero powers of $\ph$.
Closing each $A$-diagram against its reflection and reducing the displayed
trivalent bigons likewise leaves a nonzero power of $\ph$.  Consequently
all of the displayed $A$- and $J$-diagrams are nonzero.  There are unique
$x,y\in\CC^\times$ such that
\begin{equation}
 A_0=yJ_0,\qquad A_1=xJ_1
 \label{eq:local-jellyfish-scalars}
\end{equation}
for $n=2$, and the same notation is used for
\begin{equation}
 A^{\mathrm{odd}}_0=yJ^{\mathrm{odd}}_0,
 \qquad A^{\mathrm{odd}}_1=xJ^{\mathrm{odd}}_1
 \label{eq:odd-local-jellyfish-scalars}
\end{equation}
when $n=3$.

\begin{lemma}[Local scalar calculation]
\label{lem:local-jellyfish-scalars}
For both $n=2$ and $n=3$, the scalars in
\eqref{eq:local-jellyfish-scalars} and
\eqref{eq:odd-local-jellyfish-scalars} satisfy
\[
 x^2=\omega_U,
 \qquad y^2=1.
\]
The residual replacement $(U,V)\mapsto(-U,-V)$ sends
$(x,y)\mapsto(-x,-y)$.  Choose the inverse pair so that $y=1$, and put
\begin{equation}
 \sigma_U:=x,
 \qquad \sigma_U^2=\omega_U.
 \label{eq:sigma-square}
\end{equation}
\end{lemma}

\begin{proof}
\TTSetN{2}
We first perform the calculation for $n=2$.  Begin with the three-strand
identity.  Insert the identity $UV$ in the upper two strands and apply
\eqref{eq:fibonacci-resolution-graphical} to the two adjacent red strands
created below the coupon pair.  This gives
\begin{equation}
\fitdisplay{
 \TTHorizontalBundleA
 =\TTNTwoUpperUVInsertion
 =\frac1\ph\,\TTUnitPairCapCup[-1]{U}{V}{90}
  +\TTUnitPairTrivalent[-1]{U}{V}{90}.}
\label{eq:n2-upper-inverse-resolution}
\end{equation}
Alternatively, insert the identity $VU$ in the lower two strands and resolve
the two adjacent red strands above the coupon pair:
\begin{equation}
\fitdisplay{
 \TTHorizontalBundleA
 =\TTNTwoLowerVUInsertion
 =\frac1\ph\,\TTUnitPairCapCup{V}{U}{90}
  +\TTUnitPairTrivalent{V}{U}{90}.}
\label{eq:n2-lower-inverse-resolution}
\end{equation}
The cap--cup terms themselves satisfy
\begin{equation}
\fitdisplay{
 \TTUnitPairCapCup[-1]{U}{V}{90}
 =\underbrace{\TTUnitPairCapCup{V}{U}{90}}_{\mathcal K}.}
\label{eq:n2-cap-cup-identification}
\end{equation}
To compare the two cuts, re-read the upper coupon using
$\mathcal R_+(U)=\omega_UV$ and the lower coupon using
$\mathcal R_-(V)=\omega_U^{-1}U$.  The two one-click factors cancel, leaving
the same cap--cup diagram.  Since the
left-hand sides of \eqref{eq:n2-upper-inverse-resolution} and
\eqref{eq:n2-lower-inverse-resolution} are the same three-strand identity,
their trivalent summands agree:
\begin{equation}
\fitdisplay{
 \TTUnitPairTrivalent{V}{U}{90}
 =\underbrace{\TTUnitPairTrivalent[-1]{U}{V}{90}}_{\mathcal D}.}
\label{eq:D-reflected-identification}
\end{equation}
Insert $A_1=xJ_1$ at the two red
trivalent vertices.  The six new coupons cancel in inverse pairs; after the
stars are read in their original boundary intervals, exactly one
counterclockwise rotation of $V$ remains.  Hence
\begin{equation}
\fitdisplay{\mathcal D
 =x^2\,\TTDoubleSwappedJellyfishInsertion
 =x^2\omega_U^{-1}\,\TTUnitPairTrivalent{V}{U}{90}
 =x^2\omega_U^{-1}\mathcal D.}
\label{eq:D-double-insertion}
\end{equation}
Closing $\mathcal D\mathcal D^\vee$ cancels two inverse-coupon
pairs and the two bigons, leaving a nonzero power of $\ph$.
Thus $\mathcal D\ne0$ and $x^2=\omega_U$.

Apply the same two resolutions after interchanging the strand colours.  If
$\mathcal E$ denotes the resulting common trivalent summand, then inserting
$A_0=yJ_0$ twice leaves one re-reading
$\mathcal R_-(U)=V$, whose coefficient is one:
\begin{equation}
\fitdisplay{\mathcal E
 :=\TTUnitPairTrivalentColorSwap[-1]{V}{U}{90}
 =y^2\,\TTDoubleJellyfishInsertion
 =y^2\,\TTUnitPairTrivalentColorSwap{U}{V}{90}
 =y^2\mathcal E.}
\label{eq:E-double-insertion}
\end{equation}
The closure $\mathcal E\mathcal E^\vee$ is again a nonzero power of $\ph$,
so $y^2=1$.

For $n=3$ the same comparison starts from the four-strand identity.  The
upper $UV$ insertion and the lower $VU$ insertion give, respectively,
\TTSetN{3}
\begin{equation}
\fitdisplay{
 \TTHorizontalBundleA
 =\TTNThreeUpperUVInsertion
 =\frac1\ph\,\TTNThreeUpperCapCup
  +\TTNThreeUpperTrivalent.}
\label{eq:n3-upper-inverse-resolution}
\end{equation}
and
\begin{equation}
\fitdisplay{
 \TTHorizontalBundleA
 =\TTNThreeLowerVUInsertion
 =\frac1\ph\,\TTNThreeLowerCapCup
  +\TTNThreeLowerTrivalent.}
\label{eq:n3-lower-inverse-resolution}
\end{equation}
The two cap--cup terms satisfy
\begin{equation}
\fitdisplay{
 \underbrace{\TTNThreeUpperCapCup}_{\mathcal K_3}
 =\TTNThreeLowerCapCup.}
\label{eq:n3-cap-cup-identification}
\end{equation}
The opposite one-click factors in the cap--cup terms cancel as in
\eqref{eq:n2-cap-cup-identification}.  Since the two expansions have the same
left-hand side, their trivalent summands are equal:
\begin{equation}
\fitdisplay{
 \underbrace{\TTNThreeUpperTrivalent}_{\mathcal O}
 =\TTNThreeLowerTrivalent.}
\label{eq:n3-trivalent-identification}
\end{equation}
The two spectator strands at each trivalent vertex give the typed diagrams in
\eqref{eq:odd-local-jellyfish-scalars}.  Substitute
$A^{\mathrm{odd}}_1=xJ^{\mathrm{odd}}_1$ at the two black vertices in the
left-hand representative of
\eqref{eq:n3-trivalent-identification}.  The six new coupons cancel in
inverse pairs and leave one counterclockwise rotation of $V$, so
\[
 \mathcal O=x^2\omega_U^{-1}\mathcal O.
\]
Substituting $A^{\mathrm{odd}}_0=yJ^{\mathrm{odd}}_0$ at the two red
vertices in the right-hand representative leaves only the normalized
re-reading $\mathcal R_-(U)=V$ and hence
\[
 \mathcal O=y^2\mathcal O.
\]
Closing $\mathcal O\mathcal O^\vee$ cancels two
inverse-coupon pairs and the trivalent bigons, leaving a nonzero power of
$\ph$.
Thus $\mathcal O\ne0$, and again $x^2=\omega_U$ and $y^2=1$.  Finally,
replacing the inverse pair by $(-U,-V)$ changes each three-coupon diagram by
a minus sign, so it changes both $x$ and $y$.  The stated normalization
follows.
\end{proof}

For $n=3$, changing the sign of the trivalent basis of exactly one colour,
with $U,V$ fixed, sends $(x,y)$ to $(-x,-y)$.  Restoring $y=1$ by the
residual simultaneous sign change of $U,V$ leaves $\sigma_U$ unchanged.

\TTSetN{2}
The four local roof relations at $n=2$ are
\begin{equation}
\fitdisplay{\TTFloatingRoofCoupon{U}{r}
 =\frac{\omega_U}{\ph}\,
   \TTCouponWithSideCap{V}{l}{r}
   +\sigma_U\,\TTSingleRoofTrain{U}{l}}
\label{eq:red-roof-U}
\end{equation}
\begin{equation}
\fitdisplay{\TTFloatingRoofCoupon{V}{r}
 =\frac{1}{\ph\omega_U}\,
   \TTCouponWithSideCap{U}{r}{r}
   +\sigma_U^{-1}\,\TTSingleRoofTrain{V}{r}}
\label{eq:red-roof-V}
\end{equation}
\begin{equation}
\fitdisplay{\TTFloatingRoofCoupon{U}{m}
 =\frac1\ph\,
   \TTCouponWithSideCap{V}{r}{m}
   +\TTSingleRoofTrain{U}{r}}
\label{eq:black-roof-U}
\end{equation}
\begin{equation}
\fitdisplay{\TTFloatingRoofCoupon{V}{m}
 =\frac1\ph\,
   \TTCouponWithSideCap{U}{l}{m}
   +\TTSingleRoofTrain{V}{l}}
\label{eq:black-roof-V}
\end{equation}
In each case, cut the roof at the starred interval and apply
\eqref{eq:fibonacci-resolution-graphical}.  The cap--cup branch is the listed
one-click re-reading of the coupon.  The trivalent branch has the same
ordered exterior strands and is identified by the listed local relation;
no other coupon is moved.  The scalar bookkeeping is
\[
\begin{array}{c|c|c}
\text{roof and coupon}&\text{cap--cup re-reading}&
   \text{trivalent substitution}\\ \hline
\text{red},U&\mathcal R_+(U)=\omega_UV&A_1=\sigma_UJ_1\\
\text{red},V&\mathcal R_-(V)=\omega_U^{-1}U&
   J_1=\sigma_U^{-1}A_1\\
\text{black},U&\mathcal R_-(U)=V&A_0=J_0\\
\text{black},V&\mathcal R_+(V)=U&A_0=J_0.
\end{array}
\]
Together with \eqref{eq:fibonacci-resolution-graphical}, the inverse
relations, and \eqref{eq:rotation-normalization}--
\eqref{eq:inverse-rotation-relations}, these equations cover the two coupon
labels and the two possible roof colours for $n=2$.

For $n=3$, the same four resolutions use the typed diagrams in
\eqref{eq:odd-local-jellyfish-scalars}; each coupon has six incident strands.
\TTSetN{3}
\begin{equation}
\fitdisplay{\TTFloatingRoofCoupon{U}{r}
 =\frac{\omega_U}{\ph}\,
   \TTCouponWithSideCap{V}{l}{r}
   +\sigma_U\,\TTSingleRoofTrain{U}{l}}
\label{eq:odd-red-roof-U}
\end{equation}
\begin{equation}
\fitdisplay{\TTFloatingRoofCoupon{V}{r}
 =\frac{1}{\ph\omega_U}\,
   \TTCouponWithSideCap{U}{r}{r}
   +\sigma_U^{-1}\,\TTSingleRoofTrain{V}{r}}
\label{eq:odd-red-roof-V}
\end{equation}
\begin{equation}
\fitdisplay{\TTFloatingRoofCoupon{U}{m}
 =\frac1\ph\,
   \TTCouponWithSideCap{V}{r}{m}
   +\TTSingleRoofTrain{U}{r}}
\label{eq:odd-black-roof-U}
\end{equation}
\begin{equation}
\fitdisplay{\TTFloatingRoofCoupon{V}{m}
 =\frac1\ph\,
   \TTCouponWithSideCap{U}{l}{m}
   +\TTSingleRoofTrain{V}{l}}
\label{eq:odd-black-roof-V}
\end{equation}
These formulas follow by resolving the monochromatic trivalent vertices in
\eqref{eq:odd-local-jellyfish-scalars}.

\subsection{Train notation and its trace pairing}

A left train $L_d(a_0;a_1,\ldots,a_k;U)$ consists of $k+1$ alternating
$U,V$ coupons.  Beginning at the star of the leftmost coupon, $a_0$ exterior
strands precede a terminal trivalent fork containing $d$ nested caps;
$a_i$ is the number of direct strands in the $i$th gap, in addition to its
trivalent connection.  A right train is the horizontally reflected
construction, with the strand colours determined from the star of its first
coupon.  The following three-coupon diagrams are drawn for $n=3$ with
$a_0=1$, $a_1=2$, and $a_2=1$.  The cap at the terminal fork represents a
block of $d$ nested caps.

\TTSetN{3}
\begin{center}
\begin{tikzpicture}[baseline=-.65cm]
 \TTLeftTrainCore[1]{1,2,1}{U}
 \node[font=\scriptsize,anchor=east] at (-1.30,-.70) {$a_0=1$};
 \node[font=\scriptsize] at (.775,.90) {$a_1=2$};
 \node[font=\scriptsize] at (2.325,.90) {$a_2=1$};
 \node[font=\scriptsize,anchor=east] (lcaps) at (-1.32,-1.42)
   {$d$ nested caps};
 \draw[thin,->] (lcaps.east) -- (-.93,-1.24);
 \node[font=\small] at (1.55,-1.92)
   {$L_d(1;2,1;U)$};
\end{tikzpicture}
\hspace{1.5em}
\begin{tikzpicture}[baseline=-.65cm]
 \TTRightTrainCore[1]{1,2,1}{U}
 \node[font=\scriptsize,anchor=west] at (1.30,-.70) {$a_0=1$};
 \node[font=\scriptsize] at (-.775,.90) {$a_1=2$};
 \node[font=\scriptsize] at (-2.325,.90) {$a_2=1$};
 \node[font=\scriptsize,anchor=west] (rcaps) at (1.32,-1.42)
   {$d$ nested caps};
 \draw[thin,->] (rcaps.west) -- (.93,-1.24);
 \node[font=\small] at (-1.55,-1.92)
   {$R_d(1;2,1;U)$};
\end{tikzpicture}
\end{center}

Each coupon has $2n$ incident strands, so the numbers of lower exterior
strands are
\begin{equation}
 b_1=2n-a_0-a_1-2,
 \quad b_i=2n-a_{i-1}-a_i-2,
 \quad b_{k+1}=2n-a_k-1.
 \label{eq:train-port-counts}
\end{equation}
The train is admissible when every $b_i$ is nonnegative.  Its initial coupon
label and gap data determine the remaining coupon labels and strand colours
by alternation.

Read the lower boundary as outgoing.  For an ordered boundary word $W$, let
$\operatorname{Train}(W)\subseteq\Hom(\mathbf1,W)$ be the linear span of
its train diagrams.  Literal vertical reflection sends
$T\in\Hom(\mathbf1,W)$ to $T^\vee\in\Hom(W,\mathbf1)$.  Define its spherical
closure and the resulting bilinear pairing by
\[
 \operatorname{tr}_W(S,T):=T^\vee\circ S\in\End(\mathbf1)=\CC,
 \qquad
 \langle S,T\rangle_W:=\operatorname{tr}_W(S,T).
\]

The finite train families used in the two coefficient comparisons are fixed
next.  We regard $C_0$ and $B_0$ as one-car trains whose terminal signatures
are the displayed nested side caps.  For $n=2$, put
\TTSetN{2}
\[
\fitdisplay{C_0:=\TTCouponWithNestedSideCaps{U}{r}{r}{2},\qquad
C_1:=\TTRightTrain[1]{0,1}{V}.}
\]
\[
\fitdisplay{C_2:=\TTRightTrain{1,0}{V},\qquad
C_3:=\TTRightTrain{1,1,1,1}{V}.}
\]
For $n=3$, put
\TTSetN{3}
\[
\fitdisplay{B_0:=\TTCouponWithNestedSideCaps{V}{l}{r}{3},\qquad
B_1:=\TTLeftTrain[2]{0,2}{U},\qquad
B_2:=\TTLeftTrain[1]{1,1}{U}.}
\]
\[
\fitdisplay{B_3:=\TTLeftTrain[1]{1,2,2,2}{U},\qquad
B_4:=\TTLeftTrain{2,0}{U}.}
\]
\[
\fitdisplay{B_5:=\TTLeftTrain{2,2,1,2}{U},\qquad
B_6:=\TTLeftTrain{2,1,2,2}{U}.}
\]
\[
\fitdisplay{B_7:=\TTLeftTrain{2,2,2,1}{U},\qquad
B_8:=\TTLeftTrain{2,2,2,2,2,2}{U}.}
\]

The two common boundary objects are, in left-to-right order,
\[
 W_C=\rho\mu\rho\mu\rho\mu\mu\rho,
 \qquad
 W_B=\rho\mu\rho\rho\mu\rho\mu\rho\mu\rho\mu\rho.
\]
Thus $C_0,\ldots,C_3\in\Hom(\mathbf1,W_C)$ and
$B_0,\ldots,B_8\in\Hom(\mathbf1,W_B)$, so the pairings used below are
literally the closures just defined.

\begin{lemma}[Train independence]
\label{lem:train-independence}
The diagrams $C_0,\ldots,C_3$ are linearly independent, as are
$B_0,\ldots,B_8$.  In each displayed order the reflection Gram matrix is
diagonal, with diagonal entries
\[
 (\ph^4,\ph^3,\ph^3,\ph^2),
 \qquad
 (\ph^6,\ph^5,\ph^5,\ph^4,\ph^5,
   \ph^4,\ph^4,\ph^4,\ph^3).
\]
\end{lemma}

\begin{proof}
Close a train against its reflection and cancel coupon pairs from left to
right.  Each terminal cap gives a circle and each paired fork or gap
junction gives a bigon.  The number of remaining circles is determined by
the displayed terminal signature and gap data; in particular the one-car
trains $C_0$ and $B_0$ close to $\ph^4$ and $\ph^6$.  In general,
\[
\begin{tikzpicture}[baseline=-.36cm]
 \node[draw,rounded corners,minimum width=1.25cm,minimum height=.55cm] (a)
   at (0,.42) {$T$};
 \node[draw,rounded corners,minimum width=1.25cm,minimum height=.55cm] (b)
   at (0,-.42) {$T^\vee$};
 \draw[ttrho] (-.42,.14)--(-.42,-.14);
 \draw[ttmu] (0,.14)--(0,-.14);
 \draw[ttrho] (.42,.14)--(.42,-.14);
\end{tikzpicture}
\quad\longrightarrow\quad
\underbrace{
\begin{tikzpicture}[baseline=-.12cm]
 \draw[ttrho](-.28,0)circle(.18cm);
 \draw[ttmu](.28,0)circle(.18cm);
\end{tikzpicture}
\ \cdots\ 
\begin{tikzpicture}[baseline=-.12cm]
 \draw[ttrho](0,0)circle(.18cm);
\end{tikzpicture}}_{m(T)\text{ circles}},
\qquad
\langle T,T\rangle_W=\ph^{m(T)},
\]

For two different terminal signatures, or at the first unequal gap count,
two legs of the unmatched trivalent junction close to a tadpole.  Locally,
\[
\begin{tikzpicture}[baseline=-.3cm]
 \node[draw,rounded corners,minimum width=1.25cm,minimum height=.5cm] (a)
   at (0,.36) {$S$};
 \node[draw,rounded corners,minimum width=1.25cm,minimum height=.5cm] (b)
   at (0,-.36) {$T^\vee$};
 \draw[ttrho] (0,.11)--(0,-.11);
\end{tikzpicture}
\quad\longrightarrow\quad
\begin{tikzpicture}[baseline=-.22cm]
 \draw[ttrho](0,.05)circle(.20cm);
 \draw[ttrho](0,-.15)--(0,-.62);
\end{tikzpicture}
=0,
\qquad S\ne T.
\]
Thus every off-diagonal entry is zero.  Reading the circle counts from the
displayed signatures gives the two exponent lists in the statement.  They
are nonzero, so both train families are independent.
\end{proof}

\subsection{The coefficient comparison for \texorpdfstring{$n=2$}{n=2}}

\TTSetN{2}
Name the morphisms displayed below, in the standard operadic shorthand of
\cite[Section~5]{IzumiMorrisonPenneys2013}, by
\[
 \mathsf L_2:=\rho\mu(U),
 \qquad \mathsf R_2:=U(\mu\rho(U))V.
\]
Here and below juxtaposition means insertion into the indicated coloured
strands; identity strands and associators are suppressed.
The left-hand expression has a black roof inside a red roof.  First apply
\eqref{eq:black-roof-U}, retaining the complete outer red roof:
\begin{equation}
\fitdisplay{
\mathsf L_2
=\TTWithFloatingRoof{r}{\TTFloatingRoofCoupon{U}{m}}
=\frac1\ph\,
 \TTWithFloatingRoof{r}{\TTCouponWithSideCap{V}{r}{m}}
+\TTWithFloatingRoof{r}{\TTSingleRoofTrain{U}{r}}.}
\label{eq:n2-left-first-resolution}
\end{equation}
Apply the red-roof relations to the first term and resolve the red edge in
the second term into its cap--cup and trivalent channels:
\begin{equation}
\fitdisplay{\mathsf L_2=
 \frac{1}{\ph^2\omega_U}\TTCouponWithNestedSideCaps{U}{r}{r}{2}
 +\frac{1}{\ph\sigma_U}\TTSingleRoofTrain[1]{V}{r}
 +\frac1\ph\,\TTRedRoofRightTrainCapCup
 +\TTRedRoofRightTrainTrivalent.}
\label{eq:n2-left-second-resolution}
\end{equation}
The last two diagrams reduce locally as
\begin{equation}
\fitdisplay{
\TTRedRoofRightTrainCapCup
=\omega_U^{-1}\omega_U\,C_2=C_2,
\qquad
\TTRedRoofRightTrainTrivalent
=\sigma_U^{-1}\sigma_U\,C_3=C_3.}
\label{eq:n2-left-local-reductions}
\end{equation}
The second diagram in \eqref{eq:n2-left-second-resolution} is $C_1$, and the
first is $C_0$.  Hence
\begin{align}
\mathsf L_2={}&
 \frac{1}{\ph^2\omega_U}C_0
 +\frac{1}{\ph\sigma_U}C_1
 +\frac1\ph C_2+C_3.
\label{eq:n2-left-expansion}
\end{align}

For the other ordering, the red roof is inside the black roof.  The two
successive resolutions are
\begin{equation}
\fitdisplay{
\mu\rho(U)
=\TTWithFloatingRoof{m}{\TTFloatingRoofCoupon{U}{r}}
=\frac{\omega_U}{\ph}\,
 \TTWithFloatingRoof{m}{\TTCouponWithSideCap{V}{l}{r}}
+\sigma_U\,\TTWithFloatingRoof{m}{\TTSingleRoofTrain{U}{l}}}
\label{eq:n2-right-first-resolution}
\end{equation}
and
\begin{equation}
\fitdisplay{\mu\rho(U)=
 \frac{\omega_U}{\ph^2}\TTCouponWithNestedSideCaps{U}{l}{m}{2}
 +\frac{\omega_U}{\ph}\TTSingleRoofTrain[1]{V}{l}
 +\frac{\sigma_U}{\ph}\,\TTBlackRoofLeftTrainCapCup
 +\sigma_U\,\TTBlackRoofLeftTrainTrivalent.}
\label{eq:n2-right-second-resolution}
\end{equation}
Here
\begin{equation}
\fitdisplay{
\TTBlackRoofLeftTrainCapCup=\TTLeftTrain{1,0}{V},
\qquad
\TTBlackRoofLeftTrainTrivalent=\TTLeftTrain{1,1,1,1}{V}.}
\label{eq:n2-right-local-reductions}
\end{equation}

Insert the exterior inverse coupons.  Before cancellation this gives
\begin{equation}
\fitdisplay{\begin{aligned}
\mathsf R_2={}&\frac{\omega_U}{\ph^2}\,\TTConjugatedLeftCaps
 +\frac{\omega_U}{\ph}\,\TTConjugatedLeftForkCap\\
&+\frac{\sigma_U}{\ph}\,\TTConjugatedLeftShortTrain
 +\sigma_U\,\TTConjugatedLeftLongTrain.
\end{aligned}}
\label{eq:n2-conjugated-expansion}
\end{equation}
In each summand the rightmost newly inserted coupon cancels with its inverse.
No coupon is rotated, and the resulting diagrams are, in the displayed order,
\begin{equation}
 (C_0,C_2,C_1,C_3).
 \label{eq:n2-conjugated-trains}
\end{equation}
Thus
\begin{equation}
\mathsf R_2=
 \frac{\omega_U}{\ph^2}C_0
 +\frac{\sigma_U}{\ph}C_1
 +\frac{\omega_U}{\ph}C_2
 +\sigma_U C_3.
\label{eq:n2-right-expansion}
\end{equation}

The same cancellation gives the diagrammatic identity
\begin{equation}
\fitdisplay{
 \mathsf R_2
 =\TTConjugatedRedBlackRoof
 =\TTWithFloatingRoof{r}{\TTFloatingRoofCoupon{U}{m}}
 =\mathsf L_2.}
\label{eq:n2-conjugation-identity}
\end{equation}
The middle equality cancels the exterior inverse coupons with coefficient
one.  By \cref{lem:train-independence}, the equality is coefficientwise:
\[
\begin{array}{c|c|c|c}
 &\mathsf L_2&\mathsf R_2&\text{consequence}\\ \hline
C_0&(\ph^2\omega_U)^{-1}&\omega_U/\ph^2&\omega_U^2=1\\
C_1&(\ph\sigma_U)^{-1}&\sigma_U/\ph&\sigma_U^2=1\\
C_2&1/\ph&\omega_U/\ph&\omega_U=1\\
C_3&1&\sigma_U&\sigma_U=1
\end{array}
\]
All four equations, including $\sigma_U^2=\omega_U$, are therefore
satisfied exactly when
\begin{equation}
 n=2:\qquad \omega_U=\sigma_U=1.
 \label{eq:n2-omega-sigma-one}
\end{equation}

\subsection{The coefficient comparison for \texorpdfstring{$n=3$}{n=3}}

\TTSetN{3}
Put
\[
 \mathsf L_3:=\rho\mu\rho(U),
 \qquad \mathsf R_3:=U(\mu\rho\mu(U))V.
\]
First reduce the inner red and black roofs of $\mathsf L_3$, while retaining the
outer red roof.  Substitution of
\eqref{eq:odd-red-roof-U}--\eqref{eq:odd-black-roof-V} gives
\begin{equation}
\fitdisplay{\mathsf L_3=\frac{\omega_U}{\ph^2}
 \TTWithFloatingRoof{r}{\TTCouponWithNestedSideCaps{U}{l}{m}{2}}
 +\frac{\omega_U}{\ph}\TTWithFloatingRoof{r}{\TTLeftTrain[1]{0,2}{V}}
 +\frac{\sigma_U}{\ph}\TTWithFloatingRoof{r}{\TTLeftTrain{1,1}{V}}
 +\sigma_U\TTWithFloatingRoof{r}{\TTLeftTrain{1,2,2,2}{V}}.}
\label{eq:n3-left-first-resolution}
\end{equation}
The four remaining red roofs reduce as follows:
\begin{equation}
\fitdisplay{\TTWithFloatingRoof{r}{\TTCouponWithNestedSideCaps{U}{l}{m}{2}}
 =\frac{\omega_U}{\ph}B_0+\sigma_U B_1.}
\label{eq:n3-left-outer-0}
\end{equation}
\begin{equation}
\fitdisplay{\TTWithFloatingRoof{r}{\TTLeftTrain[1]{0,2}{V}}
 =\frac1\ph B_2+B_3.}
\label{eq:n3-left-outer-1}
\end{equation}
\begin{equation}
\fitdisplay{\TTWithFloatingRoof{r}{\TTLeftTrain{1,1}{V}}
 =\frac1\ph B_4+B_5.}
\label{eq:n3-left-outer-2}
\end{equation}
\begin{equation}
\fitdisplay{\TTWithFloatingRoof{r}{\TTLeftTrain{1,2,2,2}{V}}
 =\frac1\ph B_6+\frac1\ph B_7-\frac1\ph B_8.}
\label{eq:n3-left-outer-3}
\end{equation}
For \eqref{eq:n3-left-outer-3}, resolve the two top red edges from left to
right, and let $X_{PQ}$ denote the complete surrounding summand obtained by
choosing $P,Q\in\{\mathsf C,\mathsf T\}$.  The outer legs and the original
leftmost star remain fixed during these reductions.  In $X_{\mathsf C
\mathsf C}$ a cap closes two legs of the other cap--cup branch to a tadpole,
so this summand vanishes.  In the two mixed branches, cancel the adjacent
inverse coupons and remove the resulting bigon.  The remaining direct-gap
data are, respectively, $(2;1,2,2;U)$ and $(2;2,2,1;U)$, with the same
leftmost star as $B_6$ and $B_7$.  Since the cap--cup resolution contributes
$\ph^{-1}$, this gives
\begin{equation}\label{eq:n3-left-four-branches}
 X_{\mathsf C\mathsf C}=0,\qquad
 X_{\mathsf C\mathsf T}=\ph^{-1}B_6,\qquad
 X_{\mathsf T\mathsf C}=\ph^{-1}B_7,\qquad
 X_{\mathsf T\mathsf T}=-\ph^{-1}B_8.
\end{equation}
For the last equality, the three red vertices form a triangle.  Rotate
\eqref{eq:fibonacci-resolution-graphical} by a quarter-turn and join its two
upper endpoints to an additional trivalent vertex.  The cap--cup term on
the left vanishes by the tadpole relation, while the two terms on the right
give
\[
 0=\frac1\ph
 \begin{tikzpicture}[baseline=-.15cm]
  \coordinate(t)at(0,.10);
  \draw[ttrho](0,.75)--(t)--(-.65,-.65);
  \draw[ttrho](t)--(.65,-.65);
  \filldraw[rho](t)circle(.045cm);
 \end{tikzpicture}
 +\begin{tikzpicture}[baseline=-.15cm]
  \coordinate(t)at(0,.30);
  \coordinate(l)at(-.35,-.16);\coordinate(r)at(.35,-.16);
  \draw[ttrho](t)--(l)--(r)--cycle;
  \draw[ttrho](0,.75)--(t);
  \draw[ttrho](-.65,-.65)--(l);\draw[ttrho](.65,-.65)--(r);
  \filldraw[rho](t)circle(.045cm);
  \filldraw[rho](l)circle(.045cm);\filldraw[rho](r)circle(.045cm);
 \end{tikzpicture}.
\]
Thus the triangle is $-\ph^{-1}$ times the three-legged tree.
After this replacement, the ordered gap data are
$(2;2,2,2,2,2;U)$ and every exterior leg and star is exactly that of $B_8$.
Consequently
\begin{align}
\mathsf L_3={}&\frac{\omega_U^2}{\ph^3}B_0
 +\frac{\omega_U\sigma_U}{\ph^2}B_1
 +\frac{\omega_U}{\ph^2}B_2+\frac{\omega_U}{\ph}B_3\notag\\
&+\frac{\sigma_U}{\ph^2}B_4
 +\frac{\sigma_U}{\ph}(B_5+B_6+B_7-B_8).
\label{eq:n3-left-expansion}
\end{align}

Define the right trains $D_0,\ldots,D_8$ by
\[
\fitdisplay{D_0:=\TTCouponWithNestedSideCaps{V}{r}{m}{3},\qquad
D_1:=\TTRightTrain[2]{0,2}{U},\qquad
D_2:=\TTRightTrain[1]{1,1}{U}.}
\]
\[
\fitdisplay{D_3:=\TTRightTrain[1]{1,2,2,2}{U},\qquad
D_4:=\TTRightTrain{2,0}{U},\qquad
D_5:=\TTRightTrain{2,2,1,2}{U}.}
\]
\[
\fitdisplay{D_6:=\TTRightTrain{2,1,2,2}{U},\qquad
D_7:=\TTRightTrain{2,2,2,1}{U}.}
\]
\[
\fitdisplay{D_8:=\TTRightTrain{2,2,2,2,2,2}{U}.}
\]
The colour-reversed inner roofs of $\mu\rho\mu(U)$ give
\begin{equation}
\fitdisplay{\mu\rho\mu(U)=\frac{1}{\ph^2\omega_U}
 \TTWithFloatingRoof{m}{\TTCouponWithNestedSideCaps{U}{r}{r}{2}}
 +\frac{1}{\ph\sigma_U}\TTWithFloatingRoof{m}{\TTRightTrain[1]{0,2}{V}}
 +\frac1\ph\TTWithFloatingRoof{m}{\TTRightTrain{1,1}{V}}
 +\TTWithFloatingRoof{m}{\TTRightTrain{1,2,2,2}{V}}.}
\label{eq:n3-right-first-resolution}
\end{equation}
The remaining black roofs satisfy
\begin{equation}
\fitdisplay{\TTWithFloatingRoof{m}{\TTCouponWithNestedSideCaps{U}{r}{r}{2}}
 =\frac1\ph D_0+D_1.}
\label{eq:n3-right-outer-0}
\end{equation}
\begin{equation}
\fitdisplay{\TTWithFloatingRoof{m}{\TTRightTrain[1]{0,2}{V}}
 =\frac1\ph D_2+D_3.}
\label{eq:n3-right-outer-1}
\end{equation}
\begin{equation}
\fitdisplay{\TTWithFloatingRoof{m}{\TTRightTrain{1,1}{V}}
 =\frac1\ph D_4+D_5.}
\label{eq:n3-right-outer-2}
\end{equation}
\begin{equation}
\fitdisplay{\TTWithFloatingRoof{m}{\TTRightTrain{1,2,2,2}{V}}
 =\frac1\ph D_6+\frac1\ph D_7-\frac1\ph D_8.}
\label{eq:n3-right-outer-3}
\end{equation}
Indeed, the right-to-left version of the preceding branch calculation gives
\[
 Y_{\mathsf C\mathsf C}=0,\qquad
 Y_{\mathsf C\mathsf T}=\ph^{-1}D_6,\qquad
 Y_{\mathsf T\mathsf C}=\ph^{-1}D_7,\qquad
 Y_{\mathsf T\mathsf T}=-\ph^{-1}D_8.
\]
Here the original rightmost star is retained, and the three nonzero branches
have ordered gap data $(2;1,2,2;U)$, $(2;2,2,1;U)$, and
$(2;2,2,2,2,2;U)$ in the right-train convention.
It follows that the coefficient vector of $\mu\rho\mu(U)$ in the displayed
right-train basis is
\begin{equation}
 \left(
 \frac1{\ph^3\omega_U},\frac1{\ph^2\omega_U},
 \frac1{\ph^2\sigma_U},\frac1{\ph\sigma_U},
 \frac1{\ph^2},\frac1\ph,\frac1\ph,\frac1\ph,-\frac1\ph
 \right).
\label{eq:n3-right-train-vector}
\end{equation}

Attach the three exterior strands to the new $U$ on the left and the three
remaining strands to the new $V$ on the right.  In each summand the new
coupon cancels with its adjacent inverse, without moving a star or rotating
a coupon.  Thus, with coefficient one,
\begin{equation}
 (UD_0V,\ldots,UD_8V)
 =(B_0,B_4,B_2,B_7,B_1,B_6,B_5,B_3,B_8).
 \label{eq:n3-conjugation-permutation}
\end{equation}
The same inverse cancellation gives the full diagrammatic identity
\begin{equation}
\fitdisplay{
 \mathsf R_3
 =\TTNthreeConjugatedRoof
 =\TTWithFloatingRoof{r}{\TTWithFloatingRoof{m}{\TTFloatingRoofCoupon{U}{r}}}
 =\mathsf L_3.}
\label{eq:n3-conjugation-identity}
\end{equation}
Substitution of \eqref{eq:n3-right-train-vector} into
\eqref{eq:n3-conjugation-permutation} therefore gives
\begin{align}
\mathsf R_3={}&\frac1{\ph^3\omega_U}B_0+\frac1{\ph^2}B_1
 +\frac1{\ph^2\sigma_U}B_2+\frac1\ph B_3
 +\frac1{\ph^2\omega_U}B_4\notag\\
&+\frac1\ph B_5+\frac1\ph B_6
 +\frac1{\ph\sigma_U}B_7-\frac1\ph B_8.
\label{eq:n3-right-expansion}
\end{align}

By \cref{lem:train-independence}, the coefficient equations are
\[
\begin{array}{c|c|c|c}
 &\mathsf L_3&\mathsf R_3&\text{consequence}\\ \hline
B_0&\omega_U^2/\ph^3&1/(\ph^3\omega_U)&\omega_U^3=1\\
B_1&\omega_U\sigma_U/\ph^2&1/\ph^2&\omega_U\sigma_U=1\\
B_2&\omega_U/\ph^2&1/(\ph^2\sigma_U)&\omega_U\sigma_U=1\\
B_3&\omega_U/\ph&1/\ph&\omega_U=1\\
B_4&\sigma_U/\ph^2&1/(\ph^2\omega_U)&\omega_U\sigma_U=1\\
B_5&\sigma_U/\ph&1/\ph&\sigma_U=1\\
B_6&\sigma_U/\ph&1/\ph&\sigma_U=1\\
B_7&\sigma_U/\ph&1/(\ph\sigma_U)&\sigma_U^2=1\\
B_8&-\sigma_U/\ph&-1/\ph&\sigma_U=1
\end{array}
\]
In particular, the $B_3$ and $B_5$ rows give
\begin{equation}
 n=3:\qquad \omega_U=\sigma_U=1.
 \label{eq:n3-omega-sigma-one}
\end{equation}
All other rows, including $\sigma_U^2=\omega_U$, are then satisfied.

\subsection{Evaluation and the universal marked category}

\begin{lemma}[Two-colour polygon reduction]
\label{lem:two-colour-polygon-reduction}
Fix $n=2$ or $3$.  Let $D$ be a closed jellyfish diagram, with no external
strand endpoints and with at least two boundary coupons.  Suppose that its
coupon-free interior has no circle,
bigon, tadpole, or edge joining two trivalent vertices.  Then either $D=0$
by a one-boundary-point component, or two consecutive boundary coupons are
joined by at least $n$ consecutive strands.  In the latter case, if
$\omega_U=\sigma_U=1$, the two coupons cancel and leave a closed jellyfish diagram
with two fewer coupons.
\end{lemma}

\begin{proof}
A vertex-free region containing a strand that returns to the same coupon is
zero: choose an innermost return strand.  Its endpoints are consecutive among
the ports of that colour, so alternation leaves one port of the other colour
inside, giving a one-boundary-point component.  Since no two trivalent
vertices are joined, every leg of a remaining trivalent vertex ends at a
coupon.  If two legs end at the same coupon, an innermost such sector gives
the same one-boundary-point reduction; thus that case is zero, and otherwise
the three coupons are distinct.  First suppose there is no trivalent vertex.
We may now assume that there is no return strand.
For two coupons the conclusion is immediate.  For three coupons, all $2n$
ports of any coupon join its two neighbours in consecutive blocks, so one
block has at least $n$ strands.  For at least four coupons, forgetting the
colours, regard the boundary coupons as the vertices of a polygon and
the strands joining nonconsecutive coupons as noncrossing diagonals.  Complete
these diagonals to a triangulation.  A triangulated polygon has two
nonadjacent ears.  At either ear all $2n$ coupon ports join its two
neighbours.  The ports going to either neighbour form a consecutive block,
so one block has at least $n$ strands.

Now suppose that trivalent vertices occur.  Cut the external boundary away
from the coupons and linearly order them.  Isotope the three legs of every
trivalent vertex monotonically toward the boundary.  If the legs of a vertex
meet coupons with indices $i<j<\ell$, call $\ell-i$ its span.  Choose a
vertex $v$ of minimal span.  A trivalent vertex lying in either bounded
sector of $v$ would, by planarity, have all three legs in the corresponding
proper boundary interval and hence smaller span.  Thus both bounded sectors
of $v$ contain no trivalent vertex.  Write $U_1,U_2,U_3$ for the distinct
coupons met by its legs in order.  If there is a coupon strictly between two
successive $U_i$, the corresponding empty sector is a polygon whose
distinguished side is the pair of legs through $v$.  Its other strands are
noncrossing diagonals, since a return strand in this vertex-free sector gives
the zero case above.  If the sector contains exactly one interior coupon,
all its $2n$ ports join the two endpoint coupons in consecutive blocks.
Otherwise the sector polygon has at least four vertices, and the two-ear
argument supplies an interior coupon incident to no diagonal, since the
endpoints of the distinguished side are adjacent.  In either case its $2n$
ports split into two consecutive blocks, one of size at least $n$.
If there is no coupon between either $U_1,U_2$ or
$U_2,U_3$, the coupon $U_2$ has one port joined to $v$ and its remaining
$2n-1$ ports split between its two neighbours; one of the two consecutive
blocks therefore has at least $n$ strands.

The colours alternate around every coupon.  Hence $n$ consecutive strands
form a full alternating side after moving the star by one-click rotations.
The relations in \cref{lem:inverse-rotation}, together with
$\omega_U=\sigma_U=1$, make every such rereading coefficient one.  Literal
reflection reverses the facing coupon, so the full bundle is an inverse
composition $UV$ or $VU$.  Cancelling it replaces the pair by identity
strands.  Removing any circles or bigons thereby created preserves
jellyfish form and leaves two fewer coupons.
\end{proof}

\begin{proposition}[Closed-diagram evaluation]
\label{prop:closed-evaluation}
For $n=2,3$, the two Fibonacci skein calculi, the inverse and rotation
relations, and the applicable four local roof relations
\eqref{eq:red-roof-U}--\eqref{eq:black-roof-V} or
\eqref{eq:odd-red-roof-U}--\eqref{eq:odd-black-roof-V}, with
$\omega_U=\sigma_U=1$, evaluate every closed diagram in
$\rho,\mu,U,V$ to a scalar.
\end{proposition}

\begin{proof}
The evaluation has two stages.  First use the roof relations only as a
pulling algorithm.  Give a coupon its depth, the least number of strands
crossed by a path from that coupon to the external boundary.  Choose a coupon
of maximal positive depth and a minimizing path.  The roof move is supported
in a narrow neighbourhood of that path.  Its new coupons have strictly
smaller depth, while the old minimizing paths for coupons outside that
neighbourhood show that their depths do not increase.  Thus the multiset of
positive coupon depths decreases in the multiset extension of the usual
order on the nonnegative integers.  Repetition therefore writes the original
diagram as a linear combination of jellyfish diagrams, in which every coupon
meets the external boundary.

Now fix one of these boundary diagrams and do not use the roof relations
again.  Apply the two Fibonacci skein calculi in its coupon-free interior:
remove circles and bigons, expand pairs of same-colour trivalent vertices by
\eqref{eq:fibonacci-resolution-graphical}, and discard every summand
containing a tadpole.  We induct on the number $k$ of remaining boundary
coupons.

If $k=0$, the two Fibonacci skein calculi evaluate the diagram.  If $k=1$,
remove the coupon disk and consider the reduced monochromatic webs incident
to its alternating red and black boundary ports.  For $n=2$, a nonzero
reduced web of either colour joining its two ports is an arc; that arc
separates the two ports of the other colour and forces a component with one
boundary point.  For $n=3$, either a monochromatic component already has one
boundary point, or the three ports of one colour lie in a reduced trivalent
tree whose arms separate the three alternating ports of the other colour.
Again a component with one boundary point occurs.  Such a component is zero
because $\Hom(\mathbf1,\rho)=\Hom(\mathbf1,\mu)=0$, equivalently because an
innermost cap reduces it to a tadpole.  Thus every one-coupon summand
vanishes.

Suppose $k\ge2$.  The interior reductions preserve closedness and introduce
no external strand endpoints.  Every resulting summand therefore satisfies the
hypotheses of \cref{lem:two-colour-polygon-reduction}: an edge between two
same-colour trivalent vertices is exactly a reducible Fibonacci pair, and a
trivalent vertex with two legs at one coupon gives the one-boundary-point
zero used above.  The lemma therefore makes the summand zero or cancels two
coupons with coefficient one.  Induction on $k$ completes the evaluation.
This is the two-colour form of the polygon argument in
\cite[Theorem~3.9 and Section~5]{IzumiMorrisonPenneys2013}.
\end{proof}

\begin{proof}[Proof of \cref{thm:marked-uniqueness}]
Let $\calP_n$ be the free $\CC$-linear spherical category generated by the
two marked Fibonacci calculi and inverse coupons $U,V$, modulo the displayed
inverse, rotation, and local roof relations with
$\omega_U=\sigma_U=1$.  The relations give a spherical tensor functor
$\calP_n\to\calC$.

This functor is full.  The Fibonacci splittings reduce every tensor word to
the alternating simples in \eqref{eq:standard-form}.  Hom spaces between
distinct members of that list vanish; each endomorphism space is scalar;
and the only identification represented by two different alternating words
is generated by $U$ and $V$.

Let $\mathcal N_n$ be the negligible ideal of $\calP_n$.  By
\cref{prop:closed-evaluation}, every closure has a realization-independent
value, so $\mathcal N_n$ is independent of $\calC$.  The kernel of
$\calP_n\to\calC$ equals $\mathcal N_n$: one inclusion is immediate, while
the converse follows from fullness and the nondegenerate spherical trace
pairing in the pseudo-unitary fusion category $\calC$.  Hence every
realization is the Karoubi completion of $\calP_n/\mathcal N_n$.

The standard realizations exist.  For $n=2$ this quotient is
$\Fib\boxtimes\Fib$; for $n=3$ it is the rank-six category $\TTthree$
constructed in \cite[Section~5]{IzumiMorrisonPenneys2013}.  The generators
and coupons in $\calP_n$ retain their ordered source labels, so the resulting
equivalence intertwines the two marked embeddings.
\end{proof}

\section*{Acknowledgments}
Both authors acknowledge support from the Beijing Natural Science
Foundation Key Program (Grant~Z220002) and the Beijing Natural Science
Foundation (Grant~Z221100002722017).

\par\medskip
\noindent\textbf{AI-use disclosure.}
The AI tool was used as a research assistant.
The authors retain full responsibility for the final
version of the manuscript.

\bibliographystyle{amsplain}
\bibliography{references}

\end{document}